\documentclass{amsart}
\usepackage{bm}

\usepackage{amsmath}
\usepackage{amsthm}
\usepackage{amssymb}
\usepackage{enumitem}
\usepackage{hyperref}

\usepackage{graphicx}

\usepackage{mathrsfs} %scriptfont \mathscr

\usepackage{color}

\makeatletter
    
    \@addtoreset{equation}{section}
  \makeatother

\newcommand{\N}{\mathbb{N}}
\newcommand{\R}{\mathbb{R}}
\newcommand{\C}{\mathbb{C}}
\newcommand{\lr}[1]{\langle #1 \rangle}
\newcommand{\eps}{\varepsilon}

\newcommand{\dt}{\partial_t}
\newcommand{\dx}{\partial_x}
\newcommand{\fd}{\lr{\dx}}

\newcommand{\No}{\mathcal{N}}
\newcommand{\Noe}{\wt{\No}}

\newcommand{\wt}[1]{\widetilde{ #1 }}

\newcommand{\ttilde}[1]{\wt{\wt{ #1 }}}

\newtheorem{thm}{Theorem}[section]
\newtheorem{prop}[thm]{Proposition}
\newtheorem*{dfn}{Definition}
\newtheorem{lem}[thm]{Lemma}
\newtheorem{cor}[thm]{Corollary}

\newtheorem{claim}{Claim}

\theoremstyle{remark}
\newtheorem{rmk}[thm]{Remark}
\newtheorem*{rmk*}{Remark}

\DeclareMathOperator{\supp}{supp}

\DeclareMathOperator*{\wlim}{weak-lim}
\allowdisplaybreaks[1]

\newcommand{\Lo}{\mbox{{{\bf{\L}}}}}

\title[Scattering for 1d NLKG with exp nonlinearity]{Scattering for the one-dimensional Klein-Gordon equation with exponential nonlinearity}

\author[M. Ikeda]{Masahiro Ikeda}
\address[M. Ikeda]{Department of Mathematics,
Faculty of Science and Technology, Keio University,
3-14-1, Hiyoshi, Kohoku-ku, Yokohama, 223-8522, Japan/
Center for Advanced Intelligence Project, RIKEN, Japan
}
\email{masahiro.ikeda@keio.jp/masahiro.ikeda@riken.jp}

\author[T. Inui]{Takahisa Inui}
\address[T. Inui]{
Department of Mathematics, Graduate School of Science, Osaka University Toyonaka, Osaka 560-0043, Japan
}
\email{inui@math.sci.osaka-u.ac.jp}

\author[M. Okamoto]{Mamoru Okamoto}
\address[M. Okamoto]{
Division of Mathematics and Physics, Faculty of Engineering,
Shinshu University,
4-17-1 Wakasato, Nagano City 380-8553, Japan
}
\email{m\_okamoto@shinshu-u.ac.jp}

\date{\today}

\begin{document}

\begin{abstract}
We consider the asymptotic behavior of solutions to the Cauchy problem for the defocusing nonlinear Klein-Gordon equation (NLKG) with exponential nonlinearity in the one spatial dimension with data in the energy space $H^1(\R) \times L^2(\R)$.
We prove that any energy solution has a global bound of the $L^6_{t,x}$ space-time norm, and hence scatters in $H^1(\R) \times L^2(\R)$ as $t\rightarrow\pm \infty$. The proof is based on the argument by Killip-Stovall-Visan (Trans. Amer. Math. Soc. 364 (2012), no. 3, 1571--1631). However, since well-posedness in $H^{1/2}(\R) \times H^{-1/2}(\R)$ for NLKG with the exponential nonlinearity holds only for small initial data, we use the $L_t^6 W^{s-1/2,6}_x$-norm for some $s>\frac{1}{2}$ instead of the $L_{t,x}^6$-norm, where $W_x^{s,p}$ denotes the $s$-th order $L^p$-based Sobolev space.
\end{abstract}

\maketitle

\tableofcontents

\section{Introduction}

We consider the asymptotic behavior of solutions to the Cauchy problem to the Klein-Gordon equation with exponential nonlinearity
\begin{equation} \label{NLKGexp}
u_{tt} - u_{xx} +u + \No (u) =0,
\end{equation}
where $u: \R _t \times \R _x \rightarrow \R$ is an unknown function and $\No:\R\rightarrow \R$ is defined by
\[
\No (u) := \left\{\exp (|u|^2)-1-|u|^2\right\}u.
\]
In this paper, we assume that the initial data belong to the energy space, i.e. $u(0) \in H^1(\R)$ and $u_t(0) \in L^2(\R)$.
Here the energy space is the set of data for which the energy of the solution
\begin{align} \label{energy}
E(u(t), u_t(t)) =  \int _{\R} \left\{ \frac{1}{2} |u_t (t,x)|^2 + \frac{1}{2} |u_x (t,x)|^2 + \frac{1}{2} |u(t,x)|^2 + \frac{1}{2} \Noe(u(t,x)) \right\} dx
\end{align}
is finite, where $\Noe:\R\rightarrow\R$ is defined by
\[
\Noe (u) := \exp (|u|^2)-1-|u|^2-\frac{1}{2} |u|^4.
\]
It is known that the energy of the solution to (\ref{NLKGexp}) is conserved.
We note that the identity $\Noe' (u) = 2 \No (u)$ holds, or equivalently the equation $\wt{N}(u) = \No(u)u$ is valid for any $u\in \mathbb{R}$.
We also note that as the nonlinearity $\No (u)$ belongs to the energy-subcritical and defocusing ($\Noe (u)\ge 0$) case, global well-posedness in the energy space $H^1(\R)\times L^2(\R)$ is a simple consequence of the energy conservation law. However the asymptotic behavior of the solution is not clear at all.

Klein-Gordon equations are fundamental hyperbolic ones and their nonlinear problems are intensively studied by many researchers (see for example \cite{NakOza01, Nak08, KSV12, KilVis13} and references therein).
Especially, Nakamura and Ozawa \cite{NakOza01} first studied well-posedness for the Klein-Gordon equation with exponential-type nonlinearity.
We note that the exponential-type nonlinearity is corresponding to energy-critical in two space dimensions (see for example \cite{NakOza98, IMM06}).
Ibrahim, Masmoudi and Nakanishi \cite{IMN11} (see also \cite{IMMN09}) studied long-time behavior of solutions in the focusing case below the ground state.
The scattering threshold for the focusing energy-critical nonlinear Klein-Gordon equation is given by that for the massless stationary equation.

On the other hand, Nakanishi \cite{Nak08} proved that space-time bounds for the focusing mass-critical Klein-Gordon equation imply space-time bounds for the corresponding nonlinear Schr\"{o}dinger equation.
This is a reflection of the fact that the Klein-Gordon equation degenerates to the Schr\"{o}dinger equation in the nonrelativistic limit.
By using this fact, Killip, Stovall and Visan \cite{KSV12} proved that any energy solution tends to a free solution in the energy space $H^1(\R^{2})\times L^2(\R^{2})$ in the defocusing case and characterized the dichotomy between this behavior and blowup for initial data with energy less than that of the ground state in the focusing case.
According to this argument, the scattering threshold for \eqref{NLKGexp} is given by that for the corresponding nonlinear Schr\"{o}dinger equation.
Moreover, by taking $L^2$-scaling transformation, we expect that the $L^2$-supercritical part vanishes, and hence harmless.

The equation \eqref{NLKGexp} is regarded as mass-critical because the nonlinearity $\No (u)$ contains the quintic part $\frac{1}{2}|u|^4u$ by the Taylor expansion.
However, there is no energy-critical nonlinearity in one spatial dimension, and thus \eqref{NLKGexp} belogns to the energy-subcritical case.

From the viewpoint of Trudinger-Moser's inequality, the growth rate as $\exp (|u|^2)$ at infinity seems to be optimal at the level of $H^{1/2}(\R)$.
We note that the $L^{\infty}(\R)$-norm is out of control of the $H^{1/2} (\R)$-norm even when the latter is small.
Accordingly, well-posedness in $H^{1/2}(\R)$ for the exponential-type nonlinearity requires the smallness of initial data (see \cite{NakOza01} for more detail).
In order to treat large initial data, it needs to assume $(u_0, u_1) \in H^s(\R) \times H^{s-1}(\R)$ for $s>\frac{1}{2}$.
This fact causes several technical difficulties in our analysis.

Our main goal in this paper is to prove that global strong solutions exist and have finite space-time norms.

\begin{thm} \label{thm:scat}
Let $(u_0, u_1) \in H^1(\R) \times L^2(\R)$ and $s \in (\frac{1}{2}, \frac{11}{12})$.
Then, there exists a unique global solution $u$ to \eqref{NLKGexp} with initial data $(u(0), u_t(0))=(u_0,u_1)$.
Moreover, this solution obeys the global space-time bounds
\[
\| u \|_{L_t^{\infty} H_x^1} + \| u_t \|_{L_t^{\infty} L_x^2} + \| \fd^{s-1/2} u \|_{L_{t,x}^6} \le C(E(u_0,u_1)).
\]
As a consequence, the solution scatters both forward and backward in time, that is, there exist scattering states $(u_0^{\pm}, u_1^{\pm}) \in H^1(\R) \times L^2(\R)$ such that
\[
\begin{pmatrix} u(t) \\ u_t(t) \end{pmatrix}
- \begin{pmatrix} \cos (\fd t) & \fd^{-1} \sin (\fd t) \\ - \fd \sin (\fd t) & \cos (\fd t) \end{pmatrix}
\begin{pmatrix} u_0^{\pm} \\ u_1^{\pm} \end{pmatrix}
\to
\begin{pmatrix} 0 \\ 0 \end{pmatrix}
\]
in $H^1(\R) \times L^2(\R)$ as $t \to \pm \infty$, where double-sign corresponds.
\end{thm}

The main point in the theorem is to derive the $L_{t,x}^6$-spacetime bound and concomitant proof of scattering.

On one hand, the lower bound of $s$ ($s>\frac{1}{2}$) comes from the well-posedness result (Proposition \ref{prop:WP}) as mentioned above.
We note that if the exponential-type nonlinearity $\No (u)$ is replaced with the power-type nonlinearity $|u|^4u$, which belongs to mass-critical case, the limiting case $s=\frac{1}{2}$ is allowed (see Corollary \ref{cor:scat} below more precisely).
Since we rely on the well-posedness in $H^s(\R)$ with $s>\frac{1}{2}$ to obtain the theorem, the stability theory (Proposition \ref{prop:stability}), which is a variant of well-posedness, requires some regularity.
Accordingly, we have to treat the inverse Strichartz estimate (Theorem \ref{IS}) (or the linear profile decomposition (Theorem \ref{lpd})) and the nonlinear decoupling (Proposition \ref{denon}) with some regularity.
On the other hand, the upper bound of $s$ ($s< \frac{11}{12}$) is needed to show the inverse Strichartz estimate (Theorem \ref{IS}).
More precisely, we use this gap between the upper bound and $1$ when we apply the Littlewood-Paley decoupling (Lemma \ref{AD}).
However, since it is enough to set $s$ slightly larger than $\frac{1}{2}$, this does not cause any problem. 

For the Klein-Gordon equation with the focusing exponential-type nonlinearity $-\No (u)$, we can expect that the dichotomy between scattering and blowup for initial data with energy less than that of the ground state holds. However it is not known existence of the ground state to the equation. Moreover we do not have Trudinger-Moser type inequality in one spatial dimension, which is useful in two spatial dimensions. We hence only consider the defocusing case.

The exponential-type nonlinearity $\No (u)$ contains the quintic nonlinearity $\frac{1}{2}|u|^4u=\frac{1}{2}u^5$, which belongs to the mass-critical case in one spatial dimension, and an infinite sum of higher-order nonlinearities, which belong to the mass-supercritical and energy-subcritical case.
By neglecting the higher-order part, the similar result as Theorem \ref{thm:scat} holds for the Klein-Gordon equation with the single power-type nonlinearity:
\begin{equation} \label{NLKG}
u_{tt} - u_{xx} +u + \mu \frac{1}{2} u^5 =0,
\end{equation}
where $\mu=+1$, which is known as the defocusing equation.
When $\mu=-1$, the nonlinearity is said to be focusing.
With a slight abuse of notation, we define the energy of \eqref{NLKG} by
\[
E(u(t), u_t(t)) =  \int _{\R} \left\{ \frac{1}{2} |u_t (t,x)|^2 + \frac{1}{2} |u_x (t,x)|^2 + \frac{1}{2} |u(t,x)|^2 + \frac{\mu}{12} |u(t,x)|^6 \right\} dx.
\]
Since the local well-posedness in $H^{1/2}(\R)\times H^{-1/2}(\R)$ to (\ref{NLKG}) is valid for arbitrary initial data in $H^{1/2}(\R)\times H^{-1/2}(\R)$, $s=\frac{1}{2}$ in the corresponding theorem for (\ref{NLKG}) of Theorem \ref{thm:scat} is allowed. Different from the case with the exponential-type nonlinearity, we can treat (\ref{NLKG}) with the focusing nonlinearity ($\mu=-1$).
In the focusing case, a static solution $u(t,x) = \sqrt[4]{2} Q(x)$ is expected as the threshold that the dichotomy between scattering and blowup holds. Here, $Q$ is the unique positive even Schwartz solution to the elliptic equation
\begin{equation} \label{GroundS}
Q'' + Q^5 =Q,\ \ \ \text{in}\ H^1(\R).
\end{equation}
Namely, $Q$ can be written as
\[
Q(x) = \frac{\sqrt[4]{3}}{\sqrt{\cosh 2x}}.
\]
It is not appropriate to measure the ‘size’ of the initial data purely in terms of the energy because of the negative sign appearing in front of the potential energy term.
For this reason, we introduce a second notion of size, namely, the mass, which is defined by
\[
M(u(t)) := \int_{\R} |u(t,x)|^2 dx.
\]
Unlike the energy, this is not conserved.
We note that $E(\sqrt[4]{2}Q,0) = \frac{1}{2} M(\sqrt[4]{2}Q)$ holds.
We can obtain the following theorem for (\ref{NLKG}):

\begin{cor} \label{cor:scat}
Let $(u_0, u_1) \in H^1(\R) \times L^2(\R)$ and in the focusing case ($\mu=-1$) assume also that $M(u_0) < \sqrt{2} M(Q)$ and $E(u_0,u_1) < E( \sqrt[4]{2} Q,0)$.
Then, there exists a unique global solution $u$ to \eqref{NLKG} with initial data $(u(0), u_t(0))=(u_0,u_1)$.
Moreover, this solution obeys the global space-time bounds
\[
\| u \|_{L_t^{\infty} H_x^1} + \| u_t \|_{L_t^{\infty} L_x^2} + \| u \|_{L_{t,x}^6} \le C(E(u_0,u_1)).
\]
As a consequence, the solution scatters both forward and backward in time, that is, there exist $(u_0^{\pm}, u_1^{\pm}) \in H^1(\R) \times L^2(\R)$ such that
\[
\begin{pmatrix} u(t) \\ u_t(t) \end{pmatrix}
- \begin{pmatrix} \cos (\fd t) & \fd^{-1} \sin (\fd t) \\ - \fd \sin (\fd t) & \cos (\fd t) \end{pmatrix}
\begin{pmatrix} u_0^{\pm} \\ u_1^{\pm} \end{pmatrix}
\to
\begin{pmatrix} 0 \\ 0 \end{pmatrix}
\]
in $H^1(\R) \times L^2(\R)$ as $t \to \pm \infty$, where double-sign corresponds.
\end{cor}

We mention a remark on the proof of this corollary in \S \ref{S:6.1}.
Because the remaining argument is almost the same as that in \cite{KSV12}, we omit the details in this paper.

\subsection{Notation} \label{S:notation}

Let $\N _0$ denote the set of nonnegative integers $\N \cup \{ 0 \}$.

We denote the Fourier transform of $f$ by $\widehat{f}$, namely
\[
\widehat{f} (\xi ) := \frac{1}{\sqrt{2\pi}} \int _{\R} e^{-ix \xi} f(x) dx .
\]
We denote the characteristic function of an interval $I$ by $\bm{1}_{I}$.
We abbreviate $\bm{1}_{(0, \infty )}$ to $\bm{1} _{>0}$.
For an interval $I \subset \R$, we set
\[
S_I (u) := \left( \int_I \int_{\R} |u(t,x)|^6 dxdt \right)^{1/6}.
\]

Let $\sigma \in C_c^{\infty}(\R )$ with $\sigma (\xi ) = \begin{cases} 1, & \text{if } |\xi | \le 1, \\ 0, & \text{if } |\xi| \ge 2 . \end{cases}$
For $N \in 2^{\N _0}$, we define
\[
\widehat{P_N f} (\xi ) := \begin{cases} \sigma (\xi ) \widehat{f}( \xi ) , & \text{if } N=1, \\ \left( \sigma \big( \frac{\xi}{N} \big) - \sigma \big( \frac{2\xi}{N} \big) \right) \widehat{f}(\xi ), & \text{otherwise} . \end{cases}
\]

We denote the dual pair of a Banach space $X$ and its dual space $X'$ by $\lr{\cdot , \cdot}_{X,X'}$.
We denote the inner product in a Hilbert space $H$ by $\lr{\cdot, \cdot}_{H}$.

In estimates, we use $C$ to denote a positive constant that can change from line to line.
If $C$ is absolute or depends only on parameters that are considered fixed, then we often write $X \lesssim Y$, which means $X \le CY$.
When an implicit constant depends on a parameter $a$, we sometimes write $X \lesssim_{a} Y$.
We define $X \ll Y$ to mean $X \le C^{-1} Y$ and $X \sim Y$ to mean $C^{-1} Y \le X \le C Y$.

\section{Fundamental tools}

It is more convenient for us to recast Klein-Gordon equations as a first-order equation for a complex-valued function $v$ via the map
\[
(u,u_t) \mapsto v:= u+i \fd^{-1} u_t .
\]
This is easily seen to be a bijection between real-valued solutions of \eqref{NLKGexp} and complex-valued solutions of
\begin{equation} \label{rNLKGexp}
-iv _t + \fd v+ \fd ^{-1} \No(\Re v) =0.
\end{equation}
As such, local or global theories for the two equations are equivalent.

We will consistently use the letter $u$ to denote solutions to \eqref{NLKGexp} and $v$ the corresponding solutions to \eqref{rNLKGexp}.
Note that the energy and the scattering norm to (\ref{rNLKGexp}) are written as
\begin{align*}
& E(v(t)) := \int _{\R} \left\{ \frac{1}{2} |\fd v(t,x)|^2 + \frac{1}{2} \Noe(\Re v(t,x)) \right\} dx, \\
& S_I(u) = S_I (v) = \left( \int_I \int_{\R} |\Re v (t,x)|^6 dx dt \right)^{1/6}.
\end{align*}
We note that Sobolev's embedding in one dimension implies that
\[
\| v \|_{L_x^{\infty}} \le \| v \|_{H_x^1}.
\]
Hence, the energy is finite if $v(0) \in H^1(\R)$.

\subsection{Strichartz estimates}

Since we will use the Strichartz estimates with the scaling parameter $\lambda$, for the reader's convenience, we record them (see for example \cite{GinVel85}).

\begin{lem}[Dispersive estimate] \label{lem:dispersive}
For $t \neq 0$ and $2 \le p < \infty$, 
\[
\| e^{-i \lambda ^2t\lr{\lambda ^{-1} \dx}} f \| _{L_x^{p}}
\lesssim |t|^{1/p-1/2} \| \lr{\lambda ^{-1} \dx}^{3(1/2-1/p)} f \| _{L_x^{p'}} .
\]
Here, the implicit constant is independent of $\lambda$ and $p':= \frac{p}{p-1}$.
\end{lem}

We call a pair $(q,r)$ admissible if $4<q \le \infty$, $2 \le r <\infty$, and $\frac{2}{q}+\frac{1}{r}=\frac{1}{2}$.

\begin{lem}[Strichartz estimate] \label{Str}
For $\lambda >0$ and any admissible pairs $(q,r)$ and $(\wt{q}, \wt{r})$, we have
\begin{gather*}
\| e^{-i \lambda ^2t\lr{\lambda ^{-1} \dx}} f \| _{L_t^q L_x^r} \lesssim \| \lr{\lambda ^{-1} \dx}^{3(r-2)/4r} f \| _{L^2_x} , \\
\left\| \int _0^t e^{-i\lambda ^2(t-s) \lr{\lambda ^{-1} \dx}}  F(s) ds \right\| _{L_t^q L_x^r}
\lesssim \| \lr{\lambda ^{-1} \dx}^{3(1/\wt{r}'-1/r)/2} F \| _{L_t^{\wt{q}'} L_x^{\wt{r}'}} .
\end{gather*}
Here, the implicit constants are independent of $\lambda$ and $\wt{q}':= \frac{\wt{q}}{\wt{q}-1}$, $\wt{r}':= \frac{\wt{r}}{\wt{r}-1}$.
\end{lem}

\subsection{Symmetries}

The following symmetries play an important role in our analysis.
We define the following operators and observe these properties:

\begin{itemize}
\item Translations: for any $y \in \R$, we define $[T_yf] (x) := f(x-y)$.
\item Lorentz boosts: For any $\nu \in \R \backslash \{ 0 \}$, we define
\[
L_{\nu} (t,x) := ( \lr{\nu} t- \nu x, \lr{\nu} x -\nu t).
\]
\begin{itemize}
\item $L_{\nu}$ preserves spacetime volume.
\item $L_{\nu}^{-1} = L_{-\nu}$ holds true.
\item $u$ is a solution to \eqref{NLKGexp} if and only if $u \circ L_{\nu}^{-1}$ is a solution to \eqref{NLKGexp}.
\end{itemize}
Indeed, by setting $(\wt{t} , \wt{x}) = (\lr{\nu} t- \nu x, \lr{\nu} x -\nu t)$, we have
\[
\left| \frac{\partial (\wt{t} , \wt{x})}{\partial (t,x)} \right| = \det \begin{pmatrix} \lr{\nu} & - \nu \\ -\nu & \lr{\nu} \end{pmatrix} = \lr{\nu}^2 - \nu ^2 =1.
\]
From 
\[
\partial _{\wt{t}}^2 = \lr{\nu} ^2 \dt^2 + 2 \nu \lr{\nu} \partial _{t} \dx + \nu ^2 \dx^2, \quad
\partial _{\wt{x}}^2 = \nu ^2 \dt^2 + 2 \nu \lr{\nu} \partial _{t} \dx + \lr{\nu} ^2 \dx^2,
\]
we have $\partial _{\wt{t}}^2 - \partial _{\wt{x}}^2 = \dt^2 - \dx^2$.
Moreover, $L_{-\nu} (\wt{t}, \wt{x}) = (t,x)$ holds.

\item $\displaystyle [\Lo _{\nu} f] (x) := [e^{-i \cdot \fd } f] \circ L_{\nu} (0,x) = \frac{1}{\sqrt{2\pi}} \int _{\R} e^{i \lr{\nu} x \xi} e^{i \nu x \lr{\xi}} \widehat{f}(\xi ) d\xi$.

\item Scaling (the scaling is useful although the Klein-Gordon equation does not possess scaling-invariant):
For any $\lambda >0$, $[D_{\lambda} f](x) := \lambda ^{-1/2} f \big( \frac{x}{\lambda} \big)$.
\end{itemize}

The action of $\Lo_{\nu}$ is easily understood on the Fourier side.
In particular, we have the following:

\begin{lem} \label{LB}

\begin{enumerate}[label=\rm{(\roman*)}]
\item \label{LB-inv}
$\Lo _{\nu}^{-1} = \Lo _{-\nu}$.

\item \label{LB-Fou}
$\mathcal{F}[ \Lo _{\nu}^{-1} f] (\wt{\xi}) = \frac{\lr{\xi}}{\lr{\wt{\xi}}} \widehat{f}(\xi )$, where $\wt{\xi} = l _{\nu}(\xi ) := \lr{\nu} \xi - \nu \lr{\xi}$. 

\item \label{LB-comm}
For $\nu , \, y \in \R$, $\Lo _{\nu}^{-1} T_y e^{i\tau \fd } = T_{\wt{y}} e^{i \wt{\tau} \fd } \Lo _{\nu}^{-1}$,
where $(\wt{\tau}, \wt{y}) = L_{\nu} (\tau ,y)$.

\item \label{LB-inn}
For any $s \in \R$, 
\[
\lr{\Lo _{\nu}^{-1} f,g}_{H^s} = \lr{f, m_s( -i\dx ) \Lo _{\nu} g}_{H^s}, \quad (\text{namely $\Lo _{\nu}$ is unitary in $H^{1/2}(\R )$})
\]
where $m_s(\xi ) := \big( \frac{\lr{\ell _{\nu}(\xi )}}{\lr{\xi}} \big) ^{2s-1}$.
Moreover, $\| m_s \| _{L^{\infty}_{\xi}} + \| m_s^{-1} \| _{L^{\infty}_{\xi}} \lesssim \lr{\nu}^{|2s-1|}$.

\item \label{LB-lin}
$[e^{-it \fd } \Lo_{\nu}^{-1} f] (x) = [e^{-i \cdot \fd } f] \circ L_{\nu} ^{-1}(t,x)$, 
$[e^{it \fd } \Lo_{\nu}^{-1} f] (x) = [e^{i \cdot \fd } f] \circ L_{\nu}(t,x)$.

\item \label{LB-dou}
For $\mu ,\nu \in \R$, $\Lo _{\mu}^{-1} \Lo _{\nu}^{-1} = \Lo _{\mu \lr{\nu} +\lr{\mu} \nu}^{-1}$.
\end{enumerate}
\end{lem}

\begin{proof}
A direct calculation shows that 
\begin{equation} \label{Jbracxi}
\lr{\wt{\xi}} = \lr{\nu} \lr{\xi} - \nu \xi
\end{equation}
and $l_{\nu}^{-1} = l_{-\nu}$.
From
\[
\frac{d \wt{\xi}}{d \xi} = \lr{\nu} - \frac{\nu \xi}{\lr{\xi}} = \frac{\lr{\wt{\xi}}}{\lr{\xi}},
\]
we have
\begin{equation} \label{LB-Foucal}
\begin{aligned}
\left[ \Lo _{-\nu} f \right] (x)
& = \frac{1}{\sqrt{2\pi}} \int _{\R} e^{i \lr{\nu} x \xi} e^{-i \nu x \lr{\xi}} \widehat{f}(\xi ) d\xi
= \frac{1}{\sqrt{2\pi}} \int _{\R} e^{i x \wt{\xi}} \widehat{f}(\xi ) d\xi \\
&= \frac{1}{\sqrt{2\pi}} \int _{\R} e^{i x \wt{\xi}} \widehat{f} ( \xi ) \frac{\lr{\xi}}{\lr{\wt{\xi}}} d\wt{\xi}
= \mathcal{F}^{-1}_{\xi} \left[ \frac{\lr{l_{-\nu}(\xi)}}{\lr{\xi}} \widehat{f} (l_{-\nu}(\xi)) \right] (x)
\end{aligned}
\end{equation}
Similarly,
\begin{equation} \label{LB-Foucal'}
\mathcal{F} [\Lo _{\nu} f] (\xi ) = \frac{\lr{\wt{\xi}}}{\lr{\xi }} \widehat{f}( \wt{\xi} ) ,
\end{equation}
which shows \ref{LB-inv}.
Accordingly, combining \eqref{LB-Foucal} with \ref{LB-inv}, we get \ref{LB-Fou}.

By \eqref{Jbracxi} and $-\wt{y} \wt{\xi} + \wt{\tau} \lr{\wt{\xi}} =  -y \xi + \tau \lr{\xi}$, we have
\begin{align*}
\mathcal{F}[\Lo _{\nu}^{-1} T_y e^{i\tau \fd } f] (\wt{\xi})
& = \frac{\lr{\xi}}{\lr{\wt{\xi}}} e^{-iy \xi} e^{i \tau \lr{\xi}} \widehat{f} (\xi )
= e^{-i \wt{y} \wt{\xi}} e^{i \wt{\tau} \lr{\wt{\xi}}} \mathcal{F}[ \Lo _{\nu}^{-1} f ] (\wt{\xi}) \\
& = \mathcal{F} [ T_{\wt{y}} e^{i \wt{\tau} \fd } \Lo _{\nu}^{-1} f] (\wt{\xi}),
\end{align*}
which concludes \ref{LB-comm}.

From \ref{LB-Fou} and \eqref{LB-Foucal'},
\begin{align*}
\lr{\Lo _{\nu}^{-1} f,g}_{H^s}
& = \int _{\R} \lr{\wt{\xi}}^{2s} \widehat{\Lo _{\nu}^{-1} f} (\wt{\xi}) \overline{\widehat{g}(\wt{\xi})} d\wt{\xi}
= \int _{\R} \lr{\wt{\xi}}^{2s} \widehat{f} (\xi) \overline{\widehat{g}(\wt{\xi})} d \xi \\
& = \int _{\R} \lr{\wt{\xi}}^{2s} \frac{\lr{\xi}}{\lr{\wt{\xi}}} \widehat{f} (\xi) \overline{\widehat{\Lo _{\nu} g}(\xi )} d \xi
= \lr{f, m_s( -i \dx ) \Lo _{\nu} g}_{H^s} .
\end{align*}
By \eqref{Jbracxi}, we have
\[
\frac{1}{\lr{\nu} + |\nu |} = \lr{\nu} - |\nu| \le \frac{\lr{\wt{\xi}}}{\lr{\xi}} = \lr{\nu} - \nu \frac{\xi}{\lr{\xi}} \le \lr{\nu}+|\nu |,
\]
which implies
\[
\| m_s \| _{L^{\infty}_{\xi}} + \| m_s^{-1} \| _{L^{\infty}_{\xi}} \lesssim \lr{\nu}^{|2s-1|}.
\]
Hence, we obtain \ref{LB-inn}.

The claim \ref{LB-lin} follows from \ref{LB-Fou} and \eqref{Jbracxi}:
\begin{align*}
[e^{\mp i \cdot \fd } f] \circ L_{\mp \nu} (t,x)
& = \frac{1}{\sqrt{2\pi}} \int e^{i(\lr{\nu} x\pm \nu t) \xi} e^{\mp i(\lr{\nu} t \pm \nu x) \lr{\xi}} \widehat{f}(\xi ) d \xi \\
& = \frac{1}{\sqrt{2\pi}} \int e^{ix( \lr{\nu} \xi -\nu \lr{\xi})} e^{\mp i(\lr{\nu} \lr{\xi} -\nu \xi ) t} \widehat{f}(\xi ) d \xi \\
& = \frac{1}{\sqrt{2\pi}} \int e^{i x \wt{\xi}} e^{\mp i\lr{\wt{\xi}}t} \widehat{f}(\xi ) d \xi \\
& = \frac{1}{\sqrt{2\pi}} \int e^{i x \wt{\xi}} e^{\mp i\lr{\wt{\xi}}t} \widehat{f}(\xi ) \frac{\lr{\xi}}{\lr{\wt{\xi}}} d \wt{\xi} \\
& = \frac{1}{\sqrt{2\pi}} \int e^{i x \wt{\xi}} \mathcal{F}[e^{\mp i\fd t} \Lo _{\nu}^{-1} f] (\wt{\xi}) d \wt{\xi} \\
& = [e^{\mp it \fd } \Lo_{\nu}^{-1} f] (x).
\end{align*}

Owing to \ref{LB-Fou}, we have
\[
\mathcal{F}[ \Lo _{\mu}^{-1} \Lo _{\nu}^{-1} f] (l_{\mu} (l _{\nu} (\xi)))
= \frac{\lr{l_{\nu} (\xi )}}{\lr{l_{\mu} (l _{\nu} (\xi))}} \mathcal{F} [ \Lo_{\nu}^{-1} f] ( l_{\nu} (\xi ))
= \frac{\lr{\xi}}{\lr{l_{\mu} (l _{\nu} (\xi))}} \widehat{f} ( \xi ) .
\]
Here, \eqref{Jbracxi} yields
\[
l_{\mu} (l _{\nu} (\xi)) = \lr{\mu} l_{\nu} (\xi ) - \mu \lr{l_{\nu}(\xi )}
= (\lr{\mu} \lr{\nu} + \mu \nu ) \xi - ( \mu \lr{\nu} + \lr{\mu} \nu ) \lr{\xi}
= l_{\mu \lr{\nu} + \lr{\mu} \nu} ( \xi ),
\]
which shows \ref{LB-dou}.
\end{proof}

\subsection{Useful lemmas}

In this subsection, we observe certain manipulations of symmetries that we will need in the proof of the inverse Strichartz inequality, Theorem \ref{IS}.

\begin{lem} \label{lem:compact}
Fix $h \in L^2(\R)$ and $B>0$.
Then, the set
\[
\mathcal{K} := \{ D_{\lambda} ^{-1} \Lo _{\nu}^{-1} m_0 ( -i \dx )^{-1} e^{i \nu x} D_{\lambda }h \colon |\nu | \le B, \, B^{-1} \le \lambda < \infty \}
\]
is precompact in $L^2(\R )$.
Moreover, the closure $\overline{\mathcal{K}}$ of $\mathcal{K}$ in $L^2(\R )$ does not contain $0$ unless $h \equiv 0$.

If $\widehat{h}$ is the characteristic function of $[-1,1]$, then
\[
\supp \widehat{g} \subset \{ |\xi | \lesssim \lr{B} \}
\]
all uniformly for $g \in \mathcal{K}$.
\end{lem}

Since the proof of this Lemma is same as that of Lemma 2.8 in \cite{KSV12}, we omit the details here.

\begin{lem} \label{lem:aeconv}
\begin{enumerate}[label=\rm{(\roman*)}]
\item \label{lem:aeconvfin}
Suppose $g_n \rightharpoonup g$ weakly in $H^1(\R )$ and $\lambda _n \rightarrow \lambda \in (0, \infty )$.
Then, there is a subsequence so that
\[
[\lr{\lambda_n^{-1} \dx}^{1/2} e^{-i \lambda _n^2 t \lr{\lambda _n^{-1} \dx}} g_n ] (x) \rightarrow [ \lr{\lambda^{-1} \dx}^{1/2} e^{-i \lambda ^2 t \lr{\lambda^{-1} \dx}} g] (x)
\]
for almost every $(t,x) \in \R \times \R$.

\item \label{lem:aeconvfin2}
For $\lambda _n \rightarrow \lambda \in (0, \infty )$ and fixed $g \in H^1(\R )$,
\[
\lim _{n \rightarrow \infty} \big\| \lr{\lambda _n^{-1} \dx}^{1/2} e^{-i \lambda _n^2 t \lr{\lambda _n^{-1} \dx}} g - \lr{\lambda^{-1} \dx}^{1/2} e^{-i \lambda ^2 t \lr{\lambda^{-1} \dx}} g \big\| _{L_{t,x}^6} = 0.
\]

\item \label{lem:aeconvinf}
Let $s \ge \frac{1}{2}$.
Fix $\theta \in (0, \frac{1}{2})$ and suppose $g_n \rightharpoonup g$ weakly in $L^2(\R )$ and $\lambda _n \rightarrow \infty$.
Then, there is a subsequence so that
\[
[\lr{\lambda _n^{-1} \dx}^{s-1/2} e^{-i \lambda _n^2 t [\lr{\lambda _n^{-1} \dx}-1]} P_{\le \lambda _n^{\theta}} g_n ] (x) \rightarrow [ e^{i t \dx^2/2} g] (x)
\]
for almost every $(t,x) \in \R \times \R$.

\item \label{lem:aeconvinf2}
Let $s \ge \frac{1}{2}$.
For $\lambda _n \rightarrow \infty$ and fixed $g \in L^2(\R )$,
\[
\lim _{n \rightarrow \infty} \big\| \lr{\lambda _n^{-1} \dx}^{s-1/2} e^{-i \lambda _n^2 t [\lr{\lambda _n^{-1} \dx}-1]} P_{\le \lambda _n^{\theta}} g - e^{i t \dx^2/2} g \big\| _{L_{t,x}^6} = 0.
\]
\end{enumerate}
\end{lem}

\begin{proof}
Firstly we show that \ref{lem:aeconvfin}.
Owing to Cantor's diagonal argument, it is enough to consider almost everywhere convergence in $(t,x) \in [-L,L] ^2$ for any $L>0$.
A simple calculation yields
\begin{align*}
& \limsup _{n \rightarrow \infty} \| \lr{\dt} ^{1/4} \fd ^{1/4} \lr{\lambda_n^{-1} \dx}^{1/2} e^{-i \lambda _n^2 t \lr{\lambda _n^{-1} \dx}} g_n \| _{L^2_{t,x}([-L,L]^2)} \\
& \lesssim _{\lambda} \limsup _{n \rightarrow \infty} \| \fd g_n \| _{L^2_{t,x}([-L,L]^2)} \\
& \lesssim _{\lambda} L^{1/2} \sup _{n \in \N} \| g_n \| _{H^1} .
\end{align*}
Here, we note that the (space-time) Fourier support of $w:= e^{-i \lambda _n^2 t \lr{\lambda _n^{-1} \dx}} g_n$ is contained in $\{ (\tau ,\xi ) \in \R ^2 \colon |\tau | + \lambda _n ^2 \lr{\lambda _n^{-1} \xi} =0 \}$.
Namely, $|\dt| w = \lambda_n^2 \lr{\lambda _n^{-1} \dx} w$ holds.
Combining this with Rellich's theorem, that is compactness of the embedding $H^{1/4} ([-L,L]^2) \hookrightarrow L^2 ([-L,L]^2)$, we obtain an $L^2_{t,x}$ convergent subsequence which is also almost everywhere convergence.
We confirm that the limit is independent of $L$.
For all $F \in C_c^{\infty} ( \R \times \R )$,
\begin{align*}
& \lim _{n \rightarrow \infty} \lr{\lr{\lambda_n^{-1} \dx}^{1/2} e^{-i \lambda _n^2 t \lr{\lambda _n^{-1} \dx}} g_n , F}_{L^2_{t,x}} \\
& = \lim _{n \rightarrow \infty} \left\langle g_n, \int _{\R} \lr{\lambda_n^{-1} \dx}^{1/2} e^{i \lambda _n^2 t \lr{\lambda _n^{-1} \dx}} F dt \right\rangle _{L_x^2} \\
& = \left\langle g, \int _{\R} \lr{\lambda^{-1} \dx}^{1/2} e^{i \lambda ^2 t \lr{\lambda^{-1} \dx}} F dt \right\rangle _{L_x^2} \\
& = \lr{\lr{\lambda^{-1} \dx}^{1/2} e^{-i \lambda^2 t \lr{\lambda^{-1} \dx}} g , F}_{L^2_{t,x}} .
\end{align*}

For the proof of \ref{lem:aeconvfin2}, we may assume that $g$ is a Schwartz function because of Lemma \ref{Str}.
Lemma \ref{lem:dispersive} implies
\begin{align*}
& \big\| \lr{\lambda _n^{-1} \dx}^{1/2} e^{-i \lambda _n^2 t \lr{\lambda _n^{-1} \dx}} g \big\| _{L_{t,x}^6(|t| \ge T)} \lesssim T^{-1/6} \| \lr{\lambda_n^{-1} \dx}^{3/2} g \| _{L^{6/5}_x}, \\
& \big\| \lr{\lambda^{-1} \dx}^{1/2} e^{-i \lambda ^2 t \lr{\lambda^{-1} \dx}} g \big\| _{L_{t,x}^6(|t| \ge T)} \lesssim T^{-1/6} \| \lr{\lambda^{-1} \dx}^{3/2} g \| _{L^{6/5}_x}
\end{align*}
for all $T>0$.
Thus, we are left to control the region $|t| \le T$.
Since
\begin{gather*}
\begin{aligned}
\lr{\lambda _n^{-1} \xi}^{1/2} - \lr{\lambda^{-1} \xi}^{1/2}
& = \frac{\lr{\lambda _n^{-1} \xi} - \lr{\lambda^{-1} \xi}}{\lr{\lambda _n^{-1} \xi}^{1/2} + \lr{\lambda^{-1} \xi}^{1/2}} \\
& = - \frac{(\lambda_n - \lambda) (\lambda_n+\lambda) \xi^2}{\lambda_n^2 \lambda^2 (\lr{\lambda _n^{-1} \xi}^{1/2} + \lr{\lambda^{-1} \xi}^{1/2}) (\lr{\lambda _n^{-1} \xi} + \lr{\lambda^{-1} \xi})} ,
\end{aligned}
\\
\lambda _n^2t \lr{\lambda_n^{-1}\xi} - \lambda^2t \lr{\lambda^{-1}\xi}
= t \frac{(\lambda_n - \lambda)(\lambda_n+\lambda)(\lambda_n^2+\lambda^2+\xi^2)}{\lambda_n^2 \lr{\lambda_n^{-1}\xi} + \lambda^2 \lr{\lambda^{-1}\xi}} ,
\end{gather*}
we have
\begin{align*}
& \big\| \lr{\lambda _n^{-1} \dx}^{1/2} e^{-i \lambda _n^2 t \lr{\lambda _n^{-1} \dx}} g - \lr{\lambda^{-1} \dx}^{1/2} e^{-i \lambda ^2 t \lr{\lambda^{-1} \dx}} g \big\| _{L_t^{\infty} L_x^2} \\
& \le \big\| (\lr{\lambda _n^{-1} \dx}^{1/2} - \lr{\lambda^{-1} \dx}^{1/2}) e^{-i \lambda _n^2 t \lr{\lambda _n^{-1} \dx}} g \big\| _{L_t^{\infty} L_x^2} \\
& \qquad + \big\| \lr{\lambda^{-1} \dx}^{1/2} \big( e^{-i \lambda_n^2 t \lr{\lambda_n^{-1} \dx}} - e^{-i \lambda ^2 t \lr{\lambda^{-1} \dx}} \big) g \big\| _{L^{\infty}_t L^2_x} \\
& \lesssim \frac{|\lambda_n-\lambda| (\lambda_n+\lambda)}{\lambda_n^2 \lambda^2} \| g \| _{H^2} + T \frac{|\lambda_n-\lambda| (\lambda_n+\lambda)}{\lambda_n^2+\lambda^2} \{ ( \lambda_n^2+ \lambda^2) \| g \| _{L^2} + \| g \| _{H^2} \}
\end{align*}
On the other hand, by Lemma \ref{Str},
\[
\big\| \lr{\lambda _n^{-1} \dx}^{1/2} e^{-i \lambda _n^2 t \lr{\lambda _n^{-1} \dx}} g \big\| _{L^5_t L^{10}_x} + \big\| \lr{\lambda^{-1} \dx}^{1/2} e^{-i \lambda ^2 t \lr{\lambda^{-1} \dx}} g \big\| _{L^5_t L^{10}_x}
\lesssim_{\lambda} \| g \| _{H^{11/10}}.
\]
Interpolating those two estimates, by $(L_t^{\infty}L_x^{2},L_t^5L_x^{10})_{5/6}=L_{t,x}^6$, we obtain
\[
\lim _{n \rightarrow \infty} \big\| \lr{\lambda _n^{-1} \dx}^{1/2} e^{-i \lambda _n^2 t \lr{\lambda _n^{-1} \dx}} g - \lr{\lambda^{-1} \dx}^{1/2} e^{-i \lambda ^2 t \lr{\lambda^{-1} \dx}} g \big\| _{L^6_{t,x} (|t| \le T)} =0
\]
for each fixed $T>0$.

We consider the case \ref{lem:aeconvinf}.
As in the proof of \ref{lem:aeconvfin}, we may restrict the range of $(t,x)$ to $[-L,L]^2$.
Since
\[
\lr{\xi}^{s-1/2}-1 = \left( s-\frac{1}{2} \right) \xi^2 \int_0^1 a \lr{a \xi}^{s-5/2} da ,
\]
owling to Lemma \ref{Str}, we have
\begin{equation} \label{eq:covinfh}
\begin{aligned}
& \big\| e^{-i \lambda _n^2 t [\lr{\lambda _n^{-1} \dx}-1]} P_{\le \lambda _n^{\theta}} \big( \lr{\lambda _n^{-1}\dx}^{s-1/2} -1 \big) g_n \big\| _{L_{t,x}^6} \\
& \lesssim \lambda _n^{2(\theta -1)} \big\| P_{\le \lambda _n^{\theta}} \lr{\lambda _n^{-1}\dx}^{1/2} g _n \big\| _{L_{x}^2}
\lesssim \lambda_n^{2(\theta-1)} \| g_n \|_{L^2}
\rightarrow 0
\end{aligned}
\end{equation}
as $n \rightarrow \infty$.
It suffices to show that there is a subsequence so that
\[
[e^{-i \lambda _n^2 t [\lr{\lambda _n^{-1} \dx}-1]} P_{\le \lambda _n^{\theta}} g_n ] (x) \rightarrow [ e^{i t \dx^2/2} g] (x)
\]
for almost every $(t,x) \in [-L,L]^2$.
By
\begin{equation} \label{eq:KG-S}
\lambda _n^2 [ \lr{\lambda _n^{-1} \xi} - 1]
= \frac{\xi ^2}{\lr{\lambda _n^{-1} \xi} + 1}
= \frac{1}{2} \xi ^2 + \xi ^2 \frac{1-\lr{\lambda _n^{-1} \xi}}{2(\lr{\lambda _n^{-1} \xi} + 1)}
= \frac{1}{2} \xi ^2 -  \frac{\lambda _n^{-2} \xi ^4}{2(\lr{\lambda _n^{-1} \xi} + 1)^2} ,
\end{equation}
we deduce that for $\theta < \frac{1}{2}$,
\[
\| e^{-i \lambda _n^2 t [\lr{\lambda _n^{-1} \dx}-1]} P_{\le \lambda _n^{\theta}} - e^{i t \dx^2/2} P_{\le \lambda _n^{\theta}} \| _{L_x^2 \rightarrow L^2_x} \rightarrow 0
\]
as $n \rightarrow \infty$.
Thus, it reduces to observe that an $L^2([-L,L]^2)$-convergence of a subsequence of $e^{i t \dx^2/2} P_{\le \lambda _n^{\theta}} g_n$.
We recall the local smoothing estimate for the Schr\"odinger equation (see \cite{ConSau88, Sjo87, Veg88}):
\[
\int _{\R} \int _{[-L,L]} | [ \fd ^{1/2} e^{it \dx^2/2} f] (x) |^2 dx dt \lesssim L \| f \| _{L^2_x}^2 ,
\]
which implies
\[
\| \lr{\dt} ^{1/8} \fd ^{1/4} e^{it \dx^2/2} g_n \| _{L^2_{t,x}([-L,L]^2)} \lesssim L^{1/2} \| g_n \| _{L^2_x} .
\]
Rellich's theorem, or compactness of the embedding $H^{1/8} ([-L,L]^2) \hookrightarrow L^2 ([-L,L]^2)$, we obtain an $L^2_{t,x}$ convergent subsequence.

Finally, we prove \ref{lem:aeconvinf2}.
By \eqref{eq:covinfh}, it suffices to show that
\[
\lim _{n \rightarrow \infty} \big\| e^{-i \lambda _n^2 t [\lr{\lambda _n^{-1} \dx}-1]} P_{\le \lambda _n^{\theta}} g - e^{i t \dx^2/2} g \big\| _{L_{t,x}^6} = 0.
\]
By Lemma \ref{Str} and the Strichartz estimates for the Schr\"odinger equation, it suffices to treat the case where $g$ is a Schwartz function.
Lemma \ref{lem:dispersive} and the dispersive estimate for the Schr\"odinger equation imply
\[
\big\| e^{-i \lambda _n^2 t [\lr{\lambda _n^{-1} \dx}-1]} P_{\le \lambda _n^{\theta}} g \| _{L^6_{t,x}(|t| \ge T)} + \|  e^{i t \dx^2/2} g \big\| _{L_{t,x}^6(|t| \ge T)}
\lesssim T^{-1/6} \| g \| _{L^{6/5}_x}
\]
for all $T>0$.
Thus we are left to control the region $|t| \le T$.
Note that from \eqref{eq:KG-S}
\begin{align*}
& \big\| e^{-i \lambda _n^2 t [\lr{\lambda _n^{-1} \dx}-1]} P_{\le \lambda _n^{\theta}} g - e^{i t \dx^2/2} g \big\| _{L^{\infty}_t L^2_x (|t| \le T)} \\
& \le \big\| \big( e^{-i \lambda _n^2 t [\lr{\lambda _n^{-1} \dx}-1]} - e^{i t \dx^2/2}  \big) P_{\le \lambda _n^{\theta}} g \big\| _{L^{\infty}_t L^2_x(|t| \le T)} + \big\| e^{i t \dx^2/2} P_{\ge \lambda _n^{\theta}} g \big\| _{L^{\infty}_t L^2_x(|t| \le T)} \\
& \le ( \lambda _n^{-2} T + \lambda _n^{-4\theta}) \| g \| _{H^4_x} .
\end{align*}
On the other hand, by Lemma \ref{Str} and the Strichartz estimates for the Schr\"odinger equation,
\[
\big\| e^{-i \lambda _n^2 t [\lr{\lambda _n^{-1} \dx}-1]} P_{\le \lambda _n^{\theta}} g \big\| _{L^5_t L^{10}_x} + \big\| e^{i t \dx^2/2} g \big\| _{L^5_t L^{10}_x}
\lesssim \| g \| _{H^{11/10}_x}.
\]
Interpolating those two estimates, by $(L_t^{\infty}L_x^{2},L_t^5L_x^{10})_{5/6}=L_{t,x}^6$, we obtain
\[
\lim _{n \rightarrow \infty} \big\| e^{-i \lambda _n^2 t [\lr{\lambda _n^{-1} \dx}-1]} P_{\le \lambda _n^{\theta}} g - e^{i t \dx^2/2} g \big\| _{L^6_{t,x} (|t| \le T)} =0
\]
for each fixed $T>0$.
\end{proof}

\section{Local theory}

We summarize a well-posedness result for \eqref{NLKGexp}, which is equivalent to \eqref{rNLKGexp}.

\begin{prop} \label{prop:WP}
Let $s \ge s_0> \frac{1}{2}$ and $v_0 \in H^s (\R )$.
Then there exists a unique maximal lifespan solution $v : I \times \R \rightarrow \C$ to \eqref{rNLKGexp} with $v(0) = v_0$.
Furthermore, the following hold:
\begin{itemize}
\item If $T= \sup I$ is finite, then $\| v(t) \| _{H^{s_0}} \rightarrow \infty$ as $t \rightarrow T$.
\item The energy and momentum are finite and conserved if $s \ge 1$.
\item If $\| v_0 \| _{H^s}$ is sufficiently small, then $v$ is global and
\[
\| \fd ^{s+3/2r-3/4} v \| _{L_t^{q} (\R ; L_x^r(\R ))} \lesssim \| v_0 \| _{H^s}
\]
for any admissible pair $(q,r)$.
Moreover, there exists $v_+ \in H^s (\R )$ such that
\begin{equation} \label{eq:rscat}
\lim _{t \rightarrow \infty} \| v(t) - e^{-it \fd } v_+ \| _{H^s} =0.
\end{equation}
For each $v_+ \in H^s (\R )$, there exists a unique solution $v$ to \eqref{rNLKGexp} in a neighborhood of $+\infty$ satisfying \eqref{eq:rscat}.
In either case,
\[
E(v(t)) = \frac{1}{2} \| v_+ \| _{H^1}^2
\]
provided that $s \ge 1$.
Similar statements holds backward in time.
\item Let $J \subset \R$ and assume that $\| \Re v \| _{L_t^{\infty}(J; L_x^{\infty}(\R))} + S_J(\fd ^{s-1/2} v) <L$.
We have
\[
\| \fd ^{s+3/2r-3/4} v \| _{L_t^{q} (J ; L_x^r(\R ))} \lesssim_{L,s,q,r} \| v_0 \| _{H^s}
\]
for any admissible pair $(q,r)$.
\end{itemize}
\end{prop}

This proposition is essentially proved by Nakamura and Ozawa \cite{NakOza01}.
Hence, we omit the proof here.

We use the following stability theory, which is a consequence of the well-posedness result (Proposition \ref{prop:WP}).
The proof follows from minor modifications of that for the nonlinear Schr\"odinger equation, and hence we omit the details.
For the proof, see \cite{KKSV12, KSV12, TVZ08} for example.

\begin{prop} \label{prop:stability}
Let $s > \frac{1}{2}$ and let $I$ be an interval and $\wt{v}$ be a solution to
\[
-i\wt{v} _t + \fd \wt{v}+ \fd ^{-1} \mathcal{N}(\Re \wt{v}) +e_1+e_2=0 .
\]
Assume that
\[
\| \wt{v} \| _{L_t^{\infty} H_x^s (I \times \R )} \le M \quad \text{and} \quad
\| \fd ^{s-1/2} \Re \wt{v} \| _{L_{t,x}^6 (I \times \R )} \le L
\]
for some positive constants $M$ and $L$.
Let $t_0 \in I$ and let $v_0$ satisfy the condition
\[
\| v_0 - \wt{v}(t_0) \| _{H^s_x} \le M'
\]
for some positive constant $M'$.
If
\begin{gather*}
\| \fd ^{s-1/2} e^{-i (t-t_0) \fd } (v_0 - \wt{v}(t_0)) \| _{L_{t,x}^6 (I \times \R )} \le \eps , \\
\| \fd ^{s+1/2} e_1 \| _{L_{t,x}^{6/5} (I \times \R )} + \| e_2 \| _{L_t^1 H_x^{s} ( I \times \R )} \le \eps
\end{gather*}
for $0< \eps < \eps _1 =\eps _1 (L,M, M')$, then there exists a unique solution $v$ to \eqref{rNLKGexp} with initial data $v_0$ at time $t_0$.

Furthermore, the solution $v$ satisfies
\[
\| \fd ^{s-1/2} (v- \wt{v}) \| _{L_{t,x}^6 (I \times \R )} \le \eps C(L,M, M'), \quad
\| v- \wt{v} \| _{L_t^{\infty} H_x^{s} (I \times \R )} \le M' C(L,M, M') .
\]
\end{prop}

Next, we observe behaviors of the boosted solutions $u \circ L_{\nu}$.

\begin{lem} \label{lem:boostsol}
Given $(u_0, u_1) \in H^1(\R ) \times L^2 (\R )$, there is an $\eps >0$ and a local solution $u$ to \eqref{NLKGexp} with $(u(0), u_t(0)) = (u_0, u_1)$ and defined in the space time region $\Omega := \{ (t,x) \colon |t| - \eps |x| < \eps \}$.
Moreover,
\begin{align}
& \| u \| _{L_t^q L_x^r (\Omega )} := \| \bm{1}_{\Omega} u \| _{L_t^q L_x^r (\R \times \R )} < \infty \quad \text{for each admissible pair } (q,r), \notag \\
& \| u \| _{L_t^{\infty} (H_x^1 \times L_x^2) (\Omega )} := \sup _{t \in \R} \int _{\R} \bm{1}_{\Omega} (t,x) \left\{ |u_t(t,x)|^2+ |u_x(t,x)|^2 + |u(t,x)|^2 \right\} dx < \infty, \\
& \lim _{R \rightarrow \infty} \sup _{|t| < \eps R} \int _{|x| >R} \left\{ |u_t(t,x)|^2 + |u_x(t,x)|^2 + |u(t,x)|^2 \right\} dx =0 . \label{eq:boostsollim}
\end{align}
The solution $u$ with these properties is unique.
\end{lem}

Although the proof for this lemma follows from a minor modification of that of Corollary 3.5 in \cite{KSV12}, we will use a similar argument later in the proof of Proposition \ref{prop:isogen}, hence we give a proof this lemma.

\begin{proof}
Proposition \ref{prop:WP} shows that there exist $T_0 = T_0(u_0,u_1) >0$ and a unique local solution $u$ to \eqref{NLKGexp} having finite Strichartz norms on $[-T_0,T_0] \times \R$.
Let $\varphi$ denote a smooth cutoff function with $\varphi (x) =1$ for $|x| \ge 1$ and $\varphi (x) =0$ for $|x| \le \frac{1}{2}$.
There exists $R_0 >0$ such that
\[
\int _{\R} \left\{ \Big| \varphi \Big( \frac{x}{R_0} \Big) u_1(x) \Big|^2 + \Big| \dx \Big[ \varphi \Big( \frac{x}{R_0} \Big) u_0(x) \Big] \Big|^2 + \Big| \varphi \Big( \frac{v}{R_0} \Big) u_0(x) \Big|^2 \right\} dx
\]
is sufficiently small.
Then, by Proposition \ref{prop:WP}, there is a global solution $\wt{u}$ to \eqref{NLKGexp} with $( \wt{u}(0,x), \wt{u}_t(0,x)) = \big( \varphi \big( \frac{x}{R_0} \big) u_0(x), \varphi \big( \frac{x}{R_0} \big) u_1(x) \big)$.
From uniqueness and finite speed of propagation, $\wt{u}$ is an extension of $u$ to $\{ (t,x) \colon |x|-|t|> R_0 \}$.
Choosing $\eps < \frac{T_0}{1+R_0+T_0}$, we get a solution $u$ to \eqref{NLKGexp} on $\Omega$.

In what follows, we show \eqref{eq:boostsollim}.
Let $\wt{u}^{\wt{R}}$ be a solution \eqref{NLKGexp} with $(\wt{u}^{\wt{R}} (0,x), \wt{u}_t^{\wt{R}} (0,x)) = \big( \varphi \big( \frac{x}{\wt{R}} \big) u_0(x), \varphi \big( \frac{x}{\wt{R}} \big) u_1(x) \big)$.
Then, from small data theory in Proposition \ref{prop:WP},
\[
\lim _{\wt{R} \rightarrow \infty} \left( \| \wt{u}^{\wt{R}} \| _{L_t^{\infty} H_x^1} + \| \dt \wt{u}^{\wt{R}} \| _{L_t^{\infty} L_x^2} \right) =0.
\]
We note that $u$ and $\wt{u}^{\wt{R}}$ agree on $|x| - |t| > \wt{R}$.
By taking $\wt{R} < (1- \eps )R$ for $R>0$, we get $\{ |t| < \eps R, \, |x| >R \} \subset \{ |x| - |t| > \wt{R} \}$ and \eqref{eq:boostsollim}.
\end{proof}

\begin{rmk} \label{rmk:boostsol}
From the above proof, we find the following:
\begin{itemize}
\item Since a global solution exists, we may take any $0< \eps < 1$.
\item We can construct a solution $u$ to the linear Klein-Gordon equation with initial data $(u_0,u_1)$ satisfying \eqref{eq:boostsollim}.
\end{itemize}
\end{rmk}

Let $P(u(t))$ denote the momentum defined by
\[
P(u(t)) := - \int _{\R} \dt u(t,x) \dx u(t,x) dx.
\]

\begin{cor} \label{cor:boostsol}
Under the same assumption as in Lemma \ref{lem:boostsol}, for $\frac{|\nu|}{\lr{\nu}} <\eps$, we have the following:
\begin{enumerate}[label=\rm{(\roman*)}]
\item $(u \circ L_{\nu}) (t,x)$ is a strong solution to \eqref{NLKGexp} on $(-\eps , \eps ) \times \R$. \label{cor:boostsols1}
\item $\nu \mapsto (u \circ L_{\nu}, \dt [u \circ L_{\nu}]) (0, \cdot )$ is continuous with values in $H^1(\R ) \times L^2(\R )$. \label{cor:boostsols2}
\item Einstein's relation holds:
\[
(E(u \circ L_{\nu}, \dt[u \circ L_{\nu}]), P(u \circ L_{\nu})) = L_{\nu}^{-1} ( E(u,u_t), P(u)) . 
\]
In particular,
\[
E(u \circ L_{\nu}, \dt[u \circ L_{\nu}])^2 - P(u \circ L_{\nu})^2
\]
is independent of $t$ and $\nu$. \label{cor:boostsols3}
\end{enumerate}
\end{cor}

\begin{proof}
Let $(\wt{t}, \wt{x}) = L_{\nu} (t,x) = (\lr{\nu} t - \nu x, \lr{\nu}x - \nu t)$.
From $t = \frac{\wt{t} + \nu x}{\lr{\nu}}$, for $\frac{|\nu |}{\lr{\nu}} < \eps$, a calculation shows
\[
|\wt{t} | - \eps | \wt{x} |
< |\lr{\nu} t - \nu x| - \left| \nu x - \frac{\nu ^2}{\lr{\nu}} t \right|
\le \frac{|t|}{\lr{\nu}} .
\]
Thus, we have
\[
L_{\nu} ((-\eps , \eps ) \times \R ) \subset \Omega .
\]
Since $L_{\nu}$ preserves spacetime volume, by Lemma \ref{lem:boostsol} and Sobolev's embedding,
\begin{align*}
\int _{-\eps}^{\eps} \int _{\R} \wt{\mathcal{N}}(u \circ L_{\nu} ) (t,x) dx dt
& \le \iint _{\Omega} \Noe(u(t,x)) dx dt \\
& \le \| u \|_{L_{t,x}^6(\Omega)}^6 \sum _{l=3}^{\infty} \frac{1}{l!} \| u \| _{L_t^{\infty} H_x^1(\Omega)}^{2l-6} < \infty .
\end{align*}
From $\dt^2 (u \circ L_{\nu}) - \dx^2 (u \circ L_{\nu}) = (u_{tt} -u_{xx}) \circ L_{\nu}$, this shows that $u \circ L_{\nu}$ is a distribution solution to \eqref{NLKGexp} in $(-\eps , \eps ) \times \R$.

We will prove that
\[
(t,\nu ) \mapsto (u \circ L_{\nu}, \dt [u \circ L_{\nu}]) (t, \cdot )
\]
is a continuous function from $\big\{ |t| < \eps , \, \frac{|\nu |}{\lr{\nu}} < \eps \big\}$ to $H_x^1(\R ) \times L_x^2(\R )$.
If we obtain this continuity, then $u \circ L_{\nu}$ belongs to $C((-\eps ,\eps ) ; H^1(\R ) \times L^2(\R ))$, i.e., it is a strong solution, and we get \ref{cor:boostsols1} and \ref{cor:boostsols2}.

Let $u^{\text{lin}}$ denote the linear solution to the Klein-Gordon equation with the initial data $(u_0,u_1)$ at zero, namely
\[
u^{\text{lin}} (t,x) := \cos (t \fd ) u_0 (x) + \fd ^{-1} \sin (t \fd ) u_1(x).
\]
Let $\wt{u} = u - u^{\text{lin}}$ denote the difference.
From \ref{LB-lin} in Lemma \ref{LB},
\begin{equation} \label{LB-linear}
u^{\text{lin}} \circ L_{\nu}
= \frac{1}{2} \left\{ e^{it \fd } \Lo _{\nu}^{-1} + e^{-it \fd } \Lo _{\nu} \right\} u_0  + \frac{1}{2i } \left\{ e^{it \fd } \Lo _{\nu}^{-1} - e^{-it \fd } \Lo _{\nu} \right\} \fd ^{-1} u_1 .
\end{equation}
By \ref{LB-Fou} in Lemma \ref{LB}, the mapping
\[
(t,\nu ) \mapsto (u^{\text{lin}} \circ L_{\nu}, \dt [u^{\text{lin}} \circ L_{\nu}]) (t, \cdot )
\]
is continuous from $\big\{ (t, \nu ) \colon |t| < \eps , \, \frac{|\nu |}{\lr{\nu}} < \eps \big\}$ to $H_x^1(\R ) \times L_x^2(\R )$.
Next, we consider the effect of the Lorentz boosts on $\wt{u}$.
Firstly, we note that $\wt{u}$ satisfies
\[
\wt{u}_{tt} - \wt{u}_{xx} + \wt{u} = - \mathcal{N}(u) , \quad
\wt{u} (0,x) = \wt{u}_t (0,x) =0.
\]
By Lemma \ref{Str}, Lemma \ref{lem:boostsol}, and $\wt{u} = u - u^{\text{lin}}$, we get
\begin{equation} \label{eq:boostSt}
\| \wt{u} \| _{L_t^q L_x^r (\Omega )} + \| \dt \wt{u} \| _{L_t^{\infty} L_x^2 (\Omega )} + \| \dx \wt{u} \| _{L_t^{\infty} L_x^2 (\Omega )} < \infty
\end{equation}
for each admissible pair $(q,r)$.
From Lemma \ref{lem:boostsol} and Remark \ref{rmk:boostsol}, we also have
\begin{equation} \label{eq:boostsolest}
\lim _{R \rightarrow \infty} \sup _{|t| < \eps R} \int _{|x| >R} \left\{ |\wt{u}_t(t,x)|^2 + |\wt{u}_x(t,x)|^2+ |\wt{u}(t,x)|^2 \right\} dx =0.
\end{equation}
Let $\wt{\mathcal{T}}$ be the stress-energy tensor for the linear Klein-Gordon equation with respect to $\wt{u}$, which has the components
\begin{align*}
& \wt{\mathcal{T}}^{00} := \frac{1}{2} |\wt{u}_t|^2 + \frac{1}{2} |\wt{u}_x|^2 + \frac{1}{2} |\wt{u}|^2, \quad
\wt{\mathcal{T}}^{01} := \wt{\mathcal{T}} ^{10} = - \wt{u}_t \wt{u}_x, \\
& \wt{\mathcal{T}}^{11} := |\wt{u}_x|^2 - \wt{\mathcal{T}}^{00} + |\wt{u}_t|^2 = \frac{1}{2} |\wt{u}_t|^2 + \frac{1}{2} |\wt{u}_x|^2 - \frac{1}{2} |\wt{u}|^2 .
\end{align*}
We also define $\wt{\mathfrak{p}} = (\wt{\mathfrak{p}}^0, \wt{\mathfrak{p}}^1)$ by
\[
\wt{\mathfrak{p}}^0 := \lr{\nu} \wt{\mathcal{T}}^{00} + \nu \wt{\mathcal{T}}^{10} , \quad
\wt{\mathfrak{p}}^1 := \lr{\nu} \wt{\mathcal{T}}^{01} + \nu \wt{\mathcal{T}}^{11} .
\]
From $\partial _{t} = \lr{\nu} \partial _{\wt{t}} - \nu \partial _{\wt{x}}$ and $\partial _{x} = \lr{\nu} \partial _{\wt{x}} - \nu \partial _{\wt{t}}$,
\begin{align*}
& (\lr{\nu} \wt{\mathfrak{p}}^0 + \nu \wt{\mathfrak{p}}^1) \circ L_{\nu} \\
& = \lr{\nu}^2 \wt{\mathcal{T}}^{00} \circ L_{\nu} + 2\lr{\nu} \nu \wt{\mathcal{T}}^{01} \circ L_{\nu} + \nu ^2 \wt{\mathcal{T}}^{11} \circ L_{\nu} \\
& = \frac{\lr{\nu}^2+\nu ^2}{2} \left[ (\wt{u}_t \circ L_{\nu} )^2 + (\wt{u}_x \circ L_{\nu} )^2 \right] - 2 \lr{\nu} \nu (\wt{u}_t \circ L_{\nu}) (\wt{u}_x \circ L_{\nu}) + \frac{1}{2} (\wt{u} \circ L_{\nu} )^2 \\
& = \frac{1}{2} \left[ \{ \dt (\wt{u} \circ L_{\nu} ) \} ^2 + \{ \dx (\wt{u} \circ L_{\nu} ) \} ^2 \right] + \frac{1}{2} (\wt{u} \circ L_{\nu} )^2 .
\end{align*}
Since
\[
L_{\nu} (t, \R ) = \{ (\lr{\nu}t - \nu x , \lr{\nu} x - \nu t) \colon x \in \R \}
= \left\{ \left(\frac{t-\nu y}{\lr{\nu}} , y \right) \colon y \in \R \right\} ,
\]
a tangent vector of $L_{\nu} (t, \R )$ is $(-\nu,\lr{\nu})$.
Hence,
\begin{equation} \label{eq:stress-tens}
\begin{aligned}
\int _{L_{\nu} (t, \R )} (-\wt{\mathfrak{p}}^1dx + \wt{\mathfrak{p}}^0 dy)
& = \int _{\R} \left( \lr{\nu} \wt{\mathfrak{p}}^0 + \nu \wt{\mathfrak{p}}^1 \right) \circ L_{\nu} (t,x) dx \\
& = \frac{1}{2} \int _{\R} \left[ | \dt (\wt{u} \circ L_{\nu} ) | ^2 + | \dx (\wt{u} \circ L_{\nu} ) | ^2 +  |\wt{u} \circ L_{\nu} |^2 \right] (t,x) dx.
\end{aligned}
\end{equation}
For fixed $(t_0,\nu _0)$ with $|t_0|< \eps$ and $\frac{|\nu _0|}{\lr{\nu _0}} < \eps$, we set
\[
\Omega _{t, \nu} := \left\{ (s,y) \colon \frac{t_0- \nu _0 y}{\lr{\nu _0}} <s< \frac{t- \nu y}{\lr{\nu}} \right\} \cup \left\{ (s,y) \colon \frac{t- \nu y}{\lr{\nu}} < s< \frac{t_0- \nu _0 y}{\lr{\nu _0}} \right\}
\]
for $(t, \nu ) \in \R \times \R$.
Then,
\[
\partial \Omega _{t, \nu} = L_{\nu} ( t, \R ) \cup L_{\nu _0} (t_0, \R ) .
\]
From $|t_0| < \eps$ and $\frac{|\nu _0|}{\lr{\nu _0}}< \eps$, for $|t-t_0| \ll 1$ and $|\nu _0 - \nu| \ll 1$, we have $\Omega _{t, \nu} \subset \Omega$.
Indeed, for $(s,y) \in \Omega _{t, \nu}$,
\[
|s| - \eps |y| < \max \left\{ \frac{|t_0-\nu _0 y| - \nu _0|y|}{\lr{\nu _0}} , \frac{|t - \nu y| - \nu |y|}{\lr{\nu}} \right\} < \max (|t_0|, |t|) < \eps .
\]

By using the mollification technique, we may assume that $\wt{\mathfrak{p}}$ is smooth.
We apply Green's theorem to $\wt{\mathfrak{p}}$ on the region $\Omega _{t, \nu}$.
For $R>0$, let $\psi _{R} (s,y) = \sigma ( \frac{|s|+|y|}{R})$, where $\sigma$ is the cut-off function defined in \S \ref{S:notation}.
Then, by \eqref{eq:boostsolest} and \eqref{eq:stress-tens},
\begin{align*}
& \bigg| \frac{1}{2} \int _{\R} \left[ | \dt (\wt{u} \circ L_{\nu} ) | ^2 + | \dx (\wt{u} \circ L_{\nu} ) | ^2 +  |\wt{u} \circ L_{\nu} |^2 \right] (t,x) dx \\
& \quad - \frac{1}{2} \int _{\R} \left[ | \dt (\wt{u} \circ L_{\nu _0} ) | ^2 + | \dx (\wt{u} \circ L_{\nu _0} ) | ^2 +  |\wt{u} \circ L_{\nu _0} |^2 \right] (t_0 ,x) dx \bigg| \\
& = \bigg| \int _{L_{\nu} (t, \R )}(-\wt{\mathfrak{p}}^1dx + \wt{\mathfrak{p}}^0 dy) - \int _{L_{\nu _0} (t_0, \R )} (-\wt{\mathfrak{p}}^1dx + \wt{\mathfrak{p}}^0 dy) \bigg| \\
& = \bigg| \int _{\partial \Omega _{t, \nu}} (-\wt{\mathfrak{p}}^1dx + \wt{\mathfrak{p}}^0 dy) \bigg| \\
& = \bigg| \lim _{R \rightarrow \infty} \int _{\partial \Omega _{t, \nu}} (-\wt{\mathfrak{p}}^1 \psi_R dx + \wt{\mathfrak{p}}^0 \psi_R dy) \bigg| \\
& \le \limsup _{R \rightarrow \infty} \int _{\Omega _{t, \nu}} \left\{ | \psi _R \nabla _{s,y} \cdot \wt{\mathfrak{p}}| (s,y) + | \wt{\mathfrak{p}} \cdot \nabla _{s,y} \psi _{R}| (s,y) \right\} dy ds \\
& \le \int _{\Omega _{t, \nu}} | (\nabla _{s,y} \cdot \wt{\mathfrak{p}}) (s,y)|  dy ds + \limsup _{R \rightarrow \infty} \frac{1}{R} \int _{-\eps R}^{\eps R} \int _{|y| \sim R} | \lr{\nabla_{s,y}} \wt{u}(s,y) |^2 dy ds \\
& \le \int _{\Omega _{t, \nu}} | (\nabla _{s,y} \cdot \wt{\mathfrak{p}}) (s,y)|  dy ds .
\end{align*}
Since
\begin{align*}
\nabla _{t,x} \cdot \wt{\mathfrak{p}}
& = \dt \wt{\mathfrak{p}}^0 + \dx \wt{\mathfrak{p}}^1 \\
& =
\lr{\nu} (\wt{u}_{tt} \wt{u}_t + \wt{u}_{tx}\wt{u}_x + \wt{u}_t\wt{u}) - \nu (\wt{u}_{tt} \wt{u}_{x} + \wt{u}_t \wt{u}_{tx}) - \lr{\nu} (\wt{u}_{tx} \wt{u}_{x} + \wt{u}_t \wt{u}_{xx}) \\
& \hspace*{150pt} + \nu (\wt{u}_{tx} \wt{u}_t + \wt{u}_{xx} \wt{u}_x - \wt{u}_x \wt{u}) \\
& = - \mathcal{N}(u) (\lr{\nu} \wt{u}_t - \nu \wt{u}_x) ,
\end{align*}
by H\"older's inequality, we get
\begin{align*}
\iint _{\Omega _{t, \nu}} |\nabla _{s,y} \cdot \wt{\mathfrak{p}}| dy ds
& \le \lr{\nu} \iint _{\Omega _{t, \nu}} |\mathcal{N}(u(s,y)) \nabla _{s,y} \wt{u} (s,y)| dy ds \\
& \le \lr{\nu} \| u \| _{L_s^5 L_y^{10} (\Omega _{t, \nu})}^5 \| \nabla _{s,y} \wt{u} \| _{L_s^{\infty} L_y^2 (\Omega )} \sum _{l=2}^{\infty} \frac{1}{l!} \| u \| _{L_s^{\infty} H_y^1(\Omega)}^{2l-4} .
\end{align*}
From Lemma \ref{lem:boostsol}, \eqref{eq:boostSt}, and Lebesgue's dominated convergence theorem, this quantity on the right hand side as above goes to zero as $(t, \nu ) \rightarrow (t_0, \nu _0)$ because $\Omega _{t,\nu} \rightarrow \emptyset$.
This shows the desired continuity.

For the proof of \ref{cor:boostsols3}, we use the stress-energy tensor associated to the nonlinear Klein-Gordon equation:
\begin{align*}
& \mathcal{T}^{00} := \frac{1}{2} |u_t|^2 + \frac{1}{2} |u_x|^2 + \frac{1}{2} |u|^2+ \frac{1}{2} \Noe (u), \quad \mathcal{T}^{01} := \mathcal{T} ^{10} = - u_t u_x, \\
&\mathcal{T}^{11} := u_x^2 - \mathcal{T}^{00} + |u_t|^2 = \frac{1}{2} |u_t|^2 + \frac{1}{2} |u_x|^2 - \frac{1}{2} |u|^2 -\frac{1}{2} \Noe (u) .
\end{align*}
Since $u$ is a solution to \eqref{NLKGexp},
\[
\dt \mathcal{T}^{\alpha 0} + \dx \mathcal{T}^{\alpha 1} =0
\]
for all $\alpha \in \{ 0,1 \}$.
We also define $\mathfrak{p} = (\mathfrak{p}^0, \mathfrak{p}^1)$ and $\mathfrak{q} = (\mathfrak{q}^0, \mathfrak{q}^1)$ by
\begin{align*}
& \mathfrak{p}^0 := \lr{\nu} \mathcal{T}^{00} + \nu \mathcal{T}^{10} , \quad
\mathfrak{p}^1 := \lr{\nu} \mathcal{T}^{01} + \nu \mathcal{T}^{11} , \\
& \mathfrak{q}^0 := \nu \mathcal{T}^{00} + \lr{\nu} \mathcal{T}^{10} , \quad
\mathfrak{q}^1 := \nu \mathcal{T}^{01} + \lr{\nu} \mathcal{T}^{11} .
\end{align*}
Then,
\begin{align}
& \nabla _{t,x} \cdot \mathfrak{p}
= \lr{\nu} (\dt \mathcal{T}^{00} + \dx \mathcal{T}^{01}) + \nu ( \dt \mathcal{T}^{10} + \dx \mathcal{T}^{11}) =0, \label{eq:nlten} \\
& \nabla _{t,x} \cdot \mathfrak{q}
= \nu (\dt \mathcal{T}^{00} + \dx \mathcal{T}^{01}) + \lr{\nu} ( \dt \mathcal{T}^{10} + \dx \mathcal{T}^{11}) =0. \notag
\end{align}
Put
\[
\Xi _{t, \nu} := \left\{ (s,y) \colon 0 <s< \frac{t- \nu y}{\lr{\nu}} \right\} \cup \left\{ (s,y) \colon \frac{t- \nu y}{\lr{\nu}} < s< 0 \right\} .
\]
Then, $\Xi _{t, \nu} \subset \Omega$ for $|t|< \eps$ and $\frac{|\nu |}{\lr{\nu}} < \eps$,
\[
\partial \Xi _{t, \nu} = L_{\nu} ( t, \R ) \cup ( \{ 0 \} \times \R ) .
\]
Hence, the same calculation as in \eqref{eq:stress-tens} yields
\begin{align*}
\int _{\partial \Xi _{t, \nu}} (-\mathfrak{p}^1 dx + \mathfrak{p}^0 dy)
& = \int _{\R} \left( \lr{\nu} \mathfrak{p}^0 + \nu \mathfrak{p}^1 \right) \circ L_{\nu} (t,x) dx - \int _{\R} \mathfrak{p}^0 (t,x) dx \\
& = E(u \circ L_{\nu}, \dt [u \circ L_{\nu}]) -\left\{ \lr{\nu} E(u,u_t) +\nu P(u) \right\} , \\
\int _{\partial \Xi _{t, \nu}} (-\mathfrak{q}^1dx + \mathfrak{q}^0 dy)
& = \int _{\R} \left( \lr{\nu} \mathfrak{q}^0 + \nu \mathfrak{q}^1 \right) \circ L_{\nu} (t,x) dx - \int _{\R} \mathfrak{q}^0 (t,x) dx \\
& = P(u \circ L_{\nu}) - \left\{ \nu E(u,u_t) + \lr{\nu} P(u) \right\} .
\end{align*}
Applying Green's theorem, by \eqref{eq:boostsolest} and \eqref{eq:nlten}, we obtain
\begin{align*}
& | E(u \circ L_{\nu}, \dt [u \circ L_{\nu}]) -\left\{ \lr{\nu} E(u,u_t) +\nu P(u) \right\} | \\
& = \bigg| \lim _{R \rightarrow \infty} \int _{\partial \Xi _{t, \nu}} (-\mathfrak{p}^1 \psi_R dx + \mathfrak{p}^0 \psi_R dy) \bigg|
\le \limsup _{R \rightarrow \infty} \int _{\Xi _{t, \nu}} |\mathfrak{p} \cdot \nabla_{s,y} \psi _R|(s,y) dy ds =0 .
\end{align*}
Similarly, we have
\[
P(u \circ L_{\nu}) = \nu E(u,u_t) + \lr{\nu} P(u) .
\]
This completes the proof.
\end{proof}

\section{Refinements of the Strichartz inequality}

The goal of this section is to show an inverse Strichartz inequality (Theorem \ref{IS}), which is an essential ingredient in the concentration compactness argument.
For the exponential-type nonlinearity, we have to consider the initial data in $H^s(\R) \times H^{s-1}(\R)$ with $s>\frac{1}{2}$.
Accordingly, we need to consider $S_{\R} (\fd^{s-1/2} f)$ instead of $S_{\R} (f)$.

\begin{lem} \label{AD}
For $f \in H^{1/2}(\R)$, 
\[
\| e^{-it \fd } f \| _{L^6_{t,x}}^2 \lesssim \sup _{N \in 2^{\N _0}} \| e^{-it \fd } P_N f \| _{L^6_{t,x}} \| f \| _{H^{1/2}_x}.
\]
\end{lem}

\begin{proof}
We apply the Littlewood-Paley square function estimate and H\"older's inequality to have
\begin{align*}
& \| e^{-it \fd } f \| _{L^6_{t,x}}^6 \\
& \sim \bigg\| \bigg( \sum _{N \in 2^{\N _0}} |e^{-it \fd } P_N f|^2 \bigg) ^{1/2} \bigg\| _{L_{t,x}^6}^6 \\
& \lesssim \sum _{N_1,N_2,N_3 \in 2^{\N _0}} \int _{\R} \int _{\R} \prod _{j=1}^3 |e^{-it \fd } P_{N_j} f|^2 dx dt \\
& \lesssim \sum _{\substack{N_1,N_2,N_3 \in 2^{\N _0} \\ N_1 \le N_2 \le N_3}} \left( \prod _{j=1}^3 \| e^{-it\fd } P_{N_j} f \| _{L_{t,x}^6} \right) \| e^{-it\fd } P_{N_1} f \| _{L_t^5L_x^{10}} \\
& \hspace*{100pt} \times \prod _{k=2,3} \| e^{-it\fd } P_{N_k} f \| _{L_t^{20/3}L_x^{5}}.
\end{align*}
By the Strichartz estimate (Lemma \ref{Str}) and Schur's test (see for example Theorem 275 in \cite{HLP88}), we get
\begin{align*}
& \| e^{-it \fd } f \| _{L^6_{t,x}}^6 \\
& \lesssim \sup _{K \in 2^{\N _0}} \| e^{-it\fd } P_K f \| _{L_{t,x}^6}^3 \sum _{\substack{N_1,N_2,N_3 \in 2^{\N _0} \\ N_1 \le N_2 \le N_3}} \| P_{N_1} f \| _{H^{3/5}_x} \| P_{N_2} f \| _{H^{9/20}_x} \| P_{N_3} f \| _{H^{9/20}_x} \\
& \lesssim \sup _{K \in 2^{\N _0}} \| e^{-it\fd } P_K f \| _{L_{t,x}^6}^3 \| f \| _{H^{1/2}_x} \sum _{\substack{N_1, N_2 \in 2^{\N _0} \\ N_1 \le N_2}} \frac{N_1^{1/10}}{N_2^{1/10}} \| P_{N_1} f \| _{H^{1/2}_x} \| P_{N_2} f \| _{H^{1/2}_x} \\
& \lesssim \sup _{K \in 2^{\N _0}} \| e^{-it\fd } P_K f \| _{L_{t,x}^6}^3 \| f \| _{H^{1/2}_x} ^3. 
\end{align*}
\end{proof}

We need to divide each dyadic interval into positive and negative intervals.

For $N \in 2^{\N}$, we define by $P_N^+$ and $P_N^-$ the the Fourier multiplier with the symbol $\big( \sigma ( \frac{\xi}{N} ) - \sigma ( \frac{2\xi}{N} ) \big) \bm{1}_{>0} (\xi )$ and $\big( \sigma ( \frac{\xi}{N} ) - \sigma ( \frac{2\xi}{N} ) \big) \bm{1}_{>0} (-\xi )$ respectively.
For convenience, we set $P_{1}^{\pm}$ as $P_1$.

\begin{lem} \label{pmD}
\[
\| e^{-it \fd } P_N f \| _{L^6_{t,x}}^6 \lesssim N^{5/2} \sup _{\pm} \| e^{-it \fd } P_N^{\pm} f \| _{L^6_{t,x}} \| P_N f \| _{L^2_x}^5.
\]
\end{lem}

\begin{proof}
The Strichartz estimate (Lemma \ref{Str})  yields that
\begin{align*}
\| e^{-it \fd } P_N f \| _{L^6_{t,x}}^6
& =  \| e^{-it \fd } P_N f \| _{L^6_{t,x}}^{1+5} \\
& \lesssim \sup _{\pm} \| e^{-it \fd } P_N^{\pm} f \| _{L^6_{t,x}} \times \left( N^{1/2} \| P_N f \| _{L^2} \right) ^5 \\
& \lesssim N^{5/2} \sup _{\pm} \| e^{-it \fd } P_N^{\pm} f \| _{L^6_{t,x}} \| P_N f \| _{L^2}^5 . 
\end{align*}
\end{proof}

By taking $\nu$ appropriately, the Fourier support of $\Lo _{\nu}^{-1} P_N^{\pm} f$ lies inside an interval centered at the origin.

\begin{lem} \label{LB-cent}
For $N \in 2^{\N _0}$, let $\nu = \frac{5}{4}N$.
Then,
\[
\| \Lo _{\nu}^{-1} P_N^+ f \| _{L^6_{x}} + \| \Lo _{-\nu}^{-1} P_N^- f \| _{L^6_{x}}  \lesssim \| P_{\le 2} \Lo _{\nu}^{-1} f \| _{L^6_{x}} .
\]
\end{lem}

\begin{proof}
We only consider the estimate for $P_N^+$ and $N \in 2^{\N}$ because the remaining cases are similarly handled.
Set $\wt{\xi} := \lr{\nu} \xi - \nu \lr{\xi}$.
By $\supp \widehat{P_N^+ f} \subset \big[ \frac{N}{2}, 2N \big]$ and
\[
\wt{\xi}
= \lr{\nu} \xi - \nu \lr{\xi}
= \frac{(\xi + \nu )(\xi - \nu)}{\lr{\nu} \xi + \nu \lr{\xi}}
\]
we get
\[
|\wt{\xi}| \le \frac{\frac{13}{4} \times \frac{3}{4}}{2\times \frac{1}{2} \times \frac{5}{4}} = \frac{39}{20} <2
\]
for $\xi \in \supp \widehat{P_N^+ f}$.
Lemma \ref{LB} \ref{LB-Fou} yields that
\begin{align*}
\mathcal{F}[\Lo _{\nu}^{-1} P_N^+ f] ( \wt{\xi} )
& = \frac{\lr{\xi}}{\lr{\wt{\xi}}} \left( \sigma \Big( \frac{\xi}{N} \Big) - \sigma \Big( \frac{2\xi}{N} \Big) \right) \bm{1}_{>0} (\xi ) \widehat{f} ( \xi ) \\
& = \left( \sigma \Big( \frac{\xi}{N} \Big) - \sigma \Big( \frac{2\xi}{N} \Big) \right) \bm{1}_{>0} (\xi ) \mathcal{F}[P_{\le 2} \Lo _{\nu}^{-1} f] ( \wt{\xi} ) .
\end{align*}
We set
\[
\rho (\wt{\xi}) = \left( \sigma \Big( \frac{\xi}{N} \Big) - \sigma \Big( \frac{2\xi}{N} \Big) \right) \bm{1}_{>0} (\xi ) .
\]
Then, a computation shows that
\[
\frac{d^{\alpha} \rho}{d\wt{\xi} ^{\alpha}} (\wt{\xi}) = N^{-\alpha} \left( \frac{d^{\alpha} \sigma}{d\xi ^{\alpha}} \Big( \frac{\xi}{N} \Big) - 2^{\alpha} \frac{d^{\alpha} \sigma}{d\xi ^{\alpha}} \Big( \frac{2\xi}{N} \Big) \right) \bm{1}_{>0} (\xi ) \left( \frac{\lr{\xi}}{\lr{\wt{\xi}}} \right) ^{\alpha}
\]
and
\[
\sup _{\xi \in [N/2,2N]} \left| \frac{d^{\alpha} \rho}{d\wt{\xi} ^{\alpha}} (\wt{\xi})  \right| \lesssim _{\alpha} 1
\]
for all $\alpha \in \N _0$.
Thus, $\mathcal{F}^{-1} [\rho ] \in L^1$.
From $\Lo _{\nu}^{-1} P_N^+ f = \mathcal{F}^{-1}[\rho] \ast (P_{\le 2} \Lo _{\nu}^{-1} f)$ and Young's inequality, we have
\begin{align*}
\| \Lo _{\nu}^{-1} P_N^+ f \| _{L^6_x}
& = \| \mathcal{F}^{-1} [\rho ] \ast (P_{\le 2} \Lo _{\nu}^{-1} f) \| _{L^6_x}
\le \| \mathcal{F}^{-1} [\rho ] \| _{L^1} \| P_{\le 2} \Lo _{\nu}^{-1} f \| _{L^6_x} \\
& \lesssim \| P_{\le 2} \Lo _{\nu}^{-1} f \| _{L^6_x}.
\end{align*}
\end{proof}

We employ the following further decoupling, which is a consequence of the bilinear estimate proved by Tao \cite{Tao03} (see also \cite{KSV12, KilVis13}).

\begin{lem} \label{CD}
Assume that $f \in L^2(\R )$ and $\supp \widehat{f} \subset \{ \xi \in \R \colon |\xi | \le 4 \}$.
\[
\| e^{-it \fd } f \| _{L^6_{t,x}} ^3 \lesssim \sup _Q \left( |Q|^{-1/5} \| e^{-it \fd } P_Q f \| _{L^{10}_{t,x}} \right) \| f \| _{L^2_x}^2.
\]
Here, the supremum is take over all dyadic intervals with the length no more than eight and $P_Q f$ denotes the restriction operator (in $\xi$-space) of $f$ to $Q$.
\end{lem}

\begin{thm}[Inverse Strichartz inequality] \label{IS}
Let $\{ f_n \} \in H^1(\R )$ and let $s \in [\frac{1}{2}, \frac{11}{12})$.
Suppose that
\[
\lim _{n \rightarrow \infty} \| f_n \|_{H^1} =A  \quad \text{and} \quad \lim _{n \rightarrow \infty} \| \fd ^{s-1/2} e^{-it\fd } f_n \| _{L^6_{t,x}} =\eps >0.
\]
Then, by passing to a subsequence, there exist $\phi \in L^2(\R )$, $\{ \lambda _n \} \subset [ \frac{1}{8}, \infty )$, $\{ \nu _n \} \subset \R$, and $\{ (t_n ,x_n ) \} \subset \R \times \R$ so that we have the following:
\begin{itemize}
\item $\lambda _n \rightarrow \lambda _{\infty} \in [\frac{1}{8}, \infty ]$ and $\nu _n \rightarrow \nu$ in $\R$.
\item $\lambda _{\infty} < \infty$ $\Rightarrow$ $\phi \in H^1(\R )$.
\item By setting $\phi _n := \begin{cases} T_{x_n} e^{it_n \fd } \Lo _{\nu _n} D_{\lambda _n} \phi , & \text{if } \lambda _{\infty} < \infty , \\ T_{x_n} e^{it_n \fd } \Lo _{\nu _n} D_{\lambda _n} P_{\le \lambda _n^{\theta}} \phi , & \text{if } \lambda _{\infty} = \infty, \end{cases}$ with $\theta = \frac{1}{100}$, the following hold:
\begin{align}
& \lim _{n \rightarrow \infty} \left( \| f_n \|_{H^1}^2 - \| f_n - \phi _n \| _{H^1}^2 - \| \phi _n \|_{H^1}^2 \right) =0, \label{IS-dec} \\
& \liminf _{n \rightarrow \infty} \| \phi _n \|_{H^1} \gtrsim \eps \Big( \frac{\eps}{A} \Big) ^{113} , \label{IS-nontr} \\
& \limsup _{n \rightarrow \infty} \| \fd ^{s-1/2} e^{-it \fd } ( f_n - \phi _n) \| _{L^6_{t,x}} \le \eps \left[ 1- c \Big( \frac{\eps}{A} \Big) ^{C} \right] ^{1/6}, \label{IS-bound} \\
& D_{\lambda _n}^{-1} \Lo _{\nu _n}^{-1} T_{x_n}^{-1} e^{-it_n \fd } f_n \rightharpoonup \phi \text{ weakly in } \begin{cases} H^1(\R), & \text{if } \lambda _{\infty} < \infty , \\ L^2(\R), & \text{if } \lambda _{\infty} =\infty . \end{cases} \label{IS-wlim}
\end{align}
\end{itemize}
Here, $c$ and $C$ are positive constants.
\end{thm}

\begin{proof}
We write a subsequence with the same subscript as the original sequence.
By Lemma \ref{AD}, we can find dyadic numbers $N_n \in 2^{\N _0}$ satisfying
\[
\eps ^2 A^{-1} \lesssim \| \fd^{s-1/2} e^{-it\fd } P_{N_n} f_n \| _{L^6_{t,x}} .
\]
Applying Lemma \ref{pmD}, we have that there exists a sing $\pm _{n} \in \{ +, - \}$ such that
\begin{align*}
(\eps ^2 A^{-1} )^6
& \lesssim N_n^{5/2} \| \fd^{s-1/2} e^{-it \fd} P_{N_n}^{\pm _n} f_n \| _{L^6_{t,x}} (N_n^{s-3/2} A)^5 \\
& \sim N_n^{6s-11/2} A^5 \| e^{-it \fd } P_{N_n}^{\pm _n} f_n \| _{L^6_{t,x}} .
\end{align*}
Owing to $s<\frac{11}{12}$, we get
\begin{equation}
\eps ^{12} A^{-11} \lesssim \liminf _{n \rightarrow \infty} \| e^{-it \fd } P_{N_n}^{\pm _n} f_n \|_{L^6_{t,x}} . \label{IS-est1}
\end{equation}
On the other hand, Lemma \ref{Str} implies
\[
\| e^{-it \fd } P_{N_n}^{\pm _n} f_n \| _{L^6_{t,x}} \lesssim N_n^{-1/2} A,
\]
which concludes that $N_n \lesssim \big( \frac{A}{\eps} \big) ^{24}$.

Set $\wt{\nu} _n := \pm _n \frac{5}{4}N_n$.
Since the Lorentz boosts preserve the volume, by \ref{LB-lin} in Lemmas \ref{LB}, Lemma \ref{LB-cent}, and \eqref{IS-est1}, we have
\begin{equation} \label{IS-est2}
\begin{aligned}
\eps ^{12} A^{-11}
& \lesssim \liminf _{n \rightarrow \infty} \| [e^{-i \cdot \fd } P_{N_n}^{\pm _n} f_n] \circ L _{\wt{\nu} _n}^{-1} \| _{L^6_{t,x}} \\
& = \liminf _{n \rightarrow \infty} \| e^{-it \fd } \Lo _{\wt{\nu} _n}^{-1} P_{N_n}^{\pm _n} f_n \| _{L^6_{t,x}} \\
& \lesssim \liminf _{n \rightarrow \infty} \| e^{-it \fd } P_{\le 2} \Lo _{\wt{\nu} _n}^{-1} f_n \| _{L^6_{t,x}} .
\end{aligned}
\end{equation}
From \ref{LB-inn} in Lemma \ref{LB},
\begin{equation} \label{IS-est3}
\| P_{\le 2} \Lo _{\nu _n}^{-1} f_n \| _{L^2_x}
\lesssim \| \Lo _{\nu _n}^{-1} f_n \| _{H^{1/2}_x}
= \| f_n \|_{H^{1/2}_x}
\lesssim A .
\end{equation}
From \eqref{IS-est2}, \eqref{IS-est3}, and Lemma \ref{CD}, there exists an interval $Q_n$ with the length no more than eight such that
\begin{equation} \label{IS-CD}
(\eps ^{12} A^{-11})^3 \lesssim A ^2 \lambda _n^{1/5} \| e^{-it\fd } P_{Q_n} P_{\le 2} \Lo _{\wt{\nu}_n}^{-1} f_n \| _{L^{10}_{t,x}} ,
\end{equation}
where $\lambda _n^{-1} \in (0,8]$ is the length of $Q_n$ and $Q_n$ are in $|\xi | \le 4$.
By the $L^p$-boundedness of $P_{\le 2}$, Lemma \ref{Str}, and \eqref{IS-est3}, we have
\begin{align*}
\| e^{-it\fd } P_{Q_n} P_{\le 2} \Lo _{\wt{\nu}_n}^{-1} f_n \| _{L^{10}_{t,x}}
&\lesssim \| e^{-it\fd } P_{Q_n} \Lo _{\wt{\nu}_n}^{-1} f_n \| _{L^{10}_{t,x}} \\
& \lesssim \| e^{-it\fd } P_{Q_n} \Lo _{\wt{\nu}_n}^{-1} f_n \| _{L^6_{t,x}} ^{3/5} \| e^{-it\fd } P_{Q_n} \Lo _{\wt{\nu}_n}^{-1} f_n \| _{L^{\infty}_{t,x}} ^{2/5} \\
& \lesssim A ^{3/5} \| e^{-it\fd } P_{Q_n} \Lo _{\wt{\nu}_n}^{-1} f_n \| _{L^{\infty}_{t,x}} ^{2/5}.
\end{align*}
Combining it with \eqref{IS-CD}, we get
\[
\lambda _n^{-1/2} \eps ^{90} A^{-89}
\lesssim \| e^{-it\fd } P_{Q_n} \Lo _{\wt{\nu}_n}^{-1} f_n \| _{L^{\infty}_{t,x}} .
\]
Therefore, there exists $(\wt{t}_n, \wt{x}_n) \in \R \times \R$ so that
\begin{equation} \label{IS-p1}
\lambda _n^{-1/2} \eps ^{90} A^{-89}
\lesssim |P_{Q_n} e^{-i\wt{t}_n\fd } \Lo _{\wt{\nu}_n}^{-1} f_n|(-\wt{x}_n)
\end{equation}
Let $\xi _n$ be the center of $Q_n$.
Owing to $|\lambda _n^{-1}| \le 8$ and $|\xi _n| \lesssim 1$, by passing to a subsequence, we have the limits $\lambda _{\infty} \in [\frac{1}{8}, \infty ]$ and $\xi _{\infty} \in \R$ respectively.
By Lemma \ref{LB} \ref{LB-inn} and $\lr{\wt{\nu}_n} \lesssim N_n \lesssim \big( \frac{A}{\eps} \big) ^{24}$,
\begin{align*}
\| D_{\lambda _n}^{-1} \Lo _{\xi _n}^{-1} T_{\wt{x}_n}^{-1} e^{-i\wt{t}_n\fd } \Lo _{\wt{\nu}_n}^{-1} f_n \| _{L^2}
& = \| \Lo _{\xi _n}^{-1} T_{\wt{x}_n}^{-1} e^{-i\wt{t}_n\fd } \Lo _{\wt{\nu}_n}^{-1} f_n \| _{L^2}
\lesssim \| \Lo _{\wt{\nu}_n}^{-1} f_n \| _{L^2} \\
& \lesssim \lr{\wt{\nu}_n} \| f_n \|_{L^2}
\lesssim \Big( \frac{A}{\eps} \Big) ^{24} A.
\end{align*}
If $\lambda _{\infty} < \infty$, this sequence is also $H^1$-bounded because of $\| D_{\lambda _n}^{-1} f \| _{H^1} \lesssim \lr{\lambda _{\infty}} \| f \| _{H^1}$.
By passing to a subsequence, there exists $\phi \in L^2 (\R)$ satisfying
\begin{equation} \label{IS-est4}
\wlim _{n \rightarrow \infty} D_{\lambda _n}^{-1} \Lo _{\xi _n}^{-1} T_{\wt{x}_n}^{-1} e^{-i\wt{t}_n\fd } \Lo _{\wt{\nu}_n}^{-1} f_n = \phi \in \begin{cases} H^1(\R) , & \text{if } \lambda _{\infty} <\infty , \\ L^2(\R) , & \text{if } \lambda _{\infty} = \infty . \end{cases}
\end{equation}

Let $h := \mathcal{F}^{-1}[\bm{1}_{[-1/2,1/2]}]$.
Lemma \ref{LB} \ref{LB-inn} implies that
\begin{equation} \label{IS-est5}
\begin{aligned}
& \lambda _n^{1/2} [ P_{Q_n} e^{-i \wt{t}_n \fd } \Lo _{\wt{\nu}_n}^{-1} f_n] (-\wt{x}_n) \\
& = \lambda _n^{1/2} [ P_{\xi _n+ [-1/(2\lambda_n), 1/(2\lambda_n)]} e^{-i \wt{t}_n \fd } \Lo _{\wt{\nu}_n}^{-1} f_n] (-\wt{x}_n) \\
& = \frac{\lambda _n^{1/2}}{\sqrt{2\pi}} \int _{\R} e^{-i \wt{x}_n \xi} \widehat{h}( \lambda _n (\xi - \xi _n)) \mathcal{F} [ e^{-i \wt{t}_n \fd } \Lo _{\wt{\nu}_n}^{-1} f_n] (\xi ) d\xi \\
& = \lr{T_{\wt{x}_n} e^{i\xi _n x} D_{\lambda _n} h, e^{-i \wt{t}_n \fd } \Lo _{\wt{\nu}_n}^{-1} f_n}_{L^2_x} \\
& = \lr{\underbrace{D_{\lambda _n}^{-1} \Lo _{\xi _n}^{-1} m_0( -i\dx )^{-1} e^{i\xi _n x} D_{\lambda _n} h}_{=:h_n}, D_{\lambda _n}^{-1} \Lo _{\xi _n}^{-1} T_{\wt{x}_n}^{-1} e^{-i \wt{t}_n \fd } \Lo _{\wt{\nu}_n}^{-1} f_n }_{L^2_x} .
\end{aligned}
\end{equation}
Then, $\| h_n \| _{L^2} \lesssim 1$.
From Lemma \ref{lem:compact}, by passing to a subsequence, we have the strong limit $h_{\infty} \in L^2(\R ) \backslash \{ 0 \}$.
Hence, it follows from \eqref{IS-p1}, \eqref{IS-est4}, \eqref{IS-est5}, and a duality argument that
\begin{equation} \label{IS-duality}
\| \phi \| _{L^2}
\gtrsim \lim _{n \rightarrow \infty} | \lr{h_n, \phi }_{L^2}|
= \lim _{n \rightarrow \infty} \left| \lambda _n^{1/2} [ P_{Q_n} e^{-i \wt{t}_n \fd } \Lo _{\wt{\nu}_n}^{-1} f_n] (-\wt{x}_n) \right|
\gtrsim \eps \Big( \frac{\eps}{A} \Big) ^{89} ,
\end{equation}
which leads to $\phi \neq 0$.
By \ref{LB-comm} and \ref{LB-dou} in Lemma \ref{LB},
\[
D_{\lambda _n}^{-1} \Lo _{\xi _n}^{-1} T_{\wt{x}_n}^{-1} e^{-i\wt{t}_n\fd } \Lo _{\wt{\nu}_n}^{-1}
= D_{\lambda _n}^{-1} \Lo _{\xi _n}^{-1} \Lo _{\wt{\nu}_n}^{-1} T_{x_n}^{-1} e^{-it_n\fd }
= D_{\lambda _n}^{-1} \Lo _{\nu _n}^{-1} T_{x_n}^{-1} e^{-it_n\fd }
\]
where $(-\wt{t}_n, -\wt{x}_n) = L_{\wt{\nu}_n}(-t_n,-x_n)$ and $\nu _n = \xi _n \lr{\wt{\nu}_n} + \lr{\xi _n} \wt{\nu}_n$.
By \eqref{IS-est4}, we obtain \eqref{IS-wlim}.
Moreover, from $|\nu _n| \lesssim \lr{\xi _n} \lr{\wt{\nu}_n} \lesssim N_n \lesssim \big( \frac{A}{\eps} \big) ^{24}$, by passing to a subsequence, we have the limit $\nu \in \R$ of $\{ \nu _n \}$.

In the sequel, we only consider the case $\lambda _{\infty} = \infty$ because the case $\lambda _{\infty} < \infty$ is similarly handled.
By \ref{LB-inn} in Lemma \ref{LB},
\[
\| \phi _n \|_{H^1} = \| \Lo _{\nu _n} D_{\lambda _n} P_{\le \lambda _n^{\theta}} \phi \| _{H^1}
\gtrsim \lr{\nu _n}^{-1} \| D_{\lambda _n} P_{\le \lambda _n^{\theta}} \phi \| _{H^1}
\gtrsim \lr{\nu _n}^{-1} \| P_{\le \lambda _n^{\theta}} \phi \| _{L^2} .
\]
By \eqref{IS-duality} and $|\nu _n| \lesssim \big( \frac{A}{\eps} \big) ^{24}$, we get
\[
\liminf _{n \rightarrow \infty} \| \phi _n \|_{H^1} \gtrsim \Big( \frac{\eps}{A} \Big) ^{24} \| \phi \| _{L^2} \gtrsim \eps \Big( \frac{\eps}{A} \Big) ^{113},
\]
which shows \eqref{IS-nontr}.

Next, we show \eqref{IS-dec}.
A direct calculation yields
\[
\| f_n \|_{H^1}^2 - \| f_n - \phi _n \| _{H^1}^2 - \| \phi _n \|_{H^1}^2  = 2 \lr{f_n- \phi _n, \phi _n}_{H^1} ,
\]
By \ref{LB-inn} in Lemma \ref{LB},
\begin{align*}
& \lr{f_n- \phi _n, \phi _n}_{H^1} \\
& = \lr{T_{x_n}^{-1} e^{-it_n \fd }f_n- \Lo _{\nu _n} D_{\lambda _n} P_{\le \lambda _n^{\theta}} \phi , \Lo _{\nu _n} D_{\lambda _n} P_{\le \lambda _n^{\theta}} \phi}_{H^1} \\
& = \lr{\Lo _{\nu _n} ^{-1} T_{x_n}^{-1} e^{-it_n \fd }f_n- D_{\lambda _n} P_{\le \lambda _n^{\theta}} \phi , m_1( -i\dx ; -\nu _n )^{-1}  D_{\lambda _n} P_{\le \lambda _n^{\theta}} \phi }_{H^1} \\
& = \lr{D_{\lambda _n}^{-1} \Lo _{\nu _n} ^{-1} T_{x_n}^{-1} e^{-it_n \fd }f_n- P_{\le \lambda _n^{\theta}} \phi , \lr{\lambda _n^{-1} \dx}^2 m_1( -i \lambda _n^{-1} \dx ; -\nu _n )^{-1} P_{\le \lambda _n^{\theta}} \phi}_{L^2} .
\end{align*}
Here, $m_1( \xi ; \nu _n ) = \frac{\lr{l_{\nu_n} (\xi )}}{\lr{\xi}} = \frac{\lr{\lr{\nu _n} \xi - \nu _n \lr{\xi}}}{\lr{\xi}}$ and
\[
P_{\le \lambda _n^{\theta}} \phi \rightarrow \phi , \quad
\lr{\lambda _n^{-1} \dx}^2 m_1( -i\lambda _n^{-1} \dx ; -\nu _n)^{-1}  P_{\le \lambda _n^{\theta}} \phi \rightarrow \lr{\nu _{\infty}}^{-1} \phi \quad \text{in } L^2(\R )
\]
as $n \rightarrow \infty$.
Therefore, by \eqref{IS-wlim}, we get \eqref{IS-dec}.

Finally, we show \eqref{IS-bound}.
We note that \ref{LB-lin} in Lemma \ref{LB} implies
\[
\| e^{-it\fd} \Lo_{\nu} f\|_{L_{t,x}^6}
= \| [e^{-i\cdot \fd} f] \circ L_{-\nu}^{-1} \|_{L_{t,x}^6}
= \| e^{-it\fd} f \| _{L_{t,x}^6}.
\]
Making the change of variable $t=\lambda _n^2 \mathfrak{t}$, $x=\lambda _n \mathfrak{x}$, we have
\begin{equation} \label{changev}
\begin{aligned}
& \| \fd ^{s-1/2} e^{-it \fd } (f_n - \phi _n ) \| _{L^6_{t,x}} \\
& = \| \fd ^{s-1/2} e^{-it \fd } ( \Lo _{\nu _n}^{-1} T_{x_n}^{-1} e^{-it_n \fd } f_n - D_{\lambda _n} P_{\le \lambda _n^{\theta}} \phi ) \| _{L^6_{t,x}} \\
& = \| D_{\lambda_n} \lr{\lambda _n^{-1} \dx}^{s-1/2} e^{-it \lr{\lambda _n^{-1} \dx}} ( D_{\lambda _n}^{-1} \Lo _{\nu _n}^{-1} T_{x_n}^{-1} e^{-it_n \fd } f_n -P_{\le \lambda _n^{\theta}} \phi ) \| _{L^6_{t,x}} \\
& = \| \lr{\lambda _n^{-1} \partial _{\mathfrak{x}}}^{s-1/2} e^{-i \lambda _n^2 \mathfrak{t} [\lr{\lambda _n^{-1} \partial _{\mathfrak{x}}}-1]} ( D_{\lambda _n}^{-1} \Lo _{\nu _n}^{-1} T_{x_n}^{-1} e^{-it_n \lr{\partial _{\mathfrak{x}}}} f_n - P_{\le \lambda _n^{\theta}} \phi ) \| _{L^6_{\mathfrak{t},\mathfrak{x}}} .
\end{aligned}
\end{equation}
Set
\[
g_n := D_{\lambda _n}^{-1} \Lo _{\nu _n}^{-1} T_{x_n}^{-1} e^{-it_n \lr{\partial _{\mathfrak{x}}}} f_n.
\]
We claim that the high frequency parts tends to zero.
\begin{claim} \label{clam1}
By passing to a subsequence,
\[
\lr{\lambda _n^{-1} \partial _{\mathfrak{x}}}^{s-1/2} e^{-i \lambda _n^2 \mathfrak{t} [\lr{\lambda _n^{-1} \partial _{\mathfrak{x}}}-1]} P_{>\lambda _n^{\theta}} g_n \rightarrow 0
\]
for almost every $(\mathfrak{t},\mathfrak{x}) \in \R \times \R$.
\end{claim}

\begin{proof}
From $\lambda _n \rightarrow \infty$, we may assume $\lambda _n > n^{4s/\theta}$ by passing to a subsequence.
By Lemma \ref{Str} and Lemma \ref{LB} \ref{LB-inn}, more precisely $|\lr{l_{-\nu_n}(\xi)}| \sim_{\eps,A} \lr{\xi}$,
\begin{align*}
\| \lr{\lambda _n^{-1} \partial _{\mathfrak{x}}}^{s-1/2} e^{-i \lambda _n^2 \mathfrak{t} [\lr{\lambda _n^{-1} \partial _{\mathfrak{x}}}-1]} P_{\ge n \lambda _n} g_n \|_{L^6_{\mathfrak{t},\mathfrak{x}}}
& \lesssim
\| \lr{\lambda _n^{-1} \partial _{\mathfrak{x}}}^{s} P_{\ge n \lambda _n} g_n \|_{L^2_{\mathfrak{x}}} \\
& \lesssim_{\eps, A}
\| \lr{\partial _{\mathfrak{x}}}^{s} P_{\gtrsim n} f_n \|_{L^2_{\mathfrak{x}}}
\lesssim n^{s-1} A \rightarrow 0
\end{align*}
as $n \rightarrow \infty$.
Next, we focus on the middle frequency case $P_{\lambda_n^{\theta} < \cdot \le n \lambda _n} g_n$.
As in the proof of Lemma \ref{lem:aeconv}, we may restrict the range of $(t,x)$ to $[-L,L]^2$.
The local smoothing estimate for the Klein-Gordon equation (see, for example, \cite{ConSau88, KPV91}) yields
\begin{align*}
& \int_{\R} \int_{[-L,L]} | \lr{\lambda _n^{-1} \partial _{\mathfrak{x}}}^{s-1/2} e^{-i \lambda _n^2 \mathfrak{t}  [\lr{\lambda _n^{-1} \partial _{\mathfrak{x}}}-1]} P_{\lambda_n^{\theta} < \cdot \le n \lambda _n} g_n|^2 d\mathfrak{x} d\mathfrak{t} \\
& \lesssim L \| |\partial _{\mathfrak{x}}|^{-1/2} \lr{\lambda_n^{-1} \dx}^{s} P_{\lambda_n^{\theta} < \cdot \le n \lambda _n} g_n \|_{L^2}^2
\lesssim_{\eps,A} L n^s \lambda_n^{-\theta/2} \| f_n \|_{L^2}^2
\rightarrow 0
\end{align*}
as $n \rightarrow \infty$.
Therefore, by passing to a subsequence, we obtain the almost everywhere convergence.
\end{proof}

From \ref{lem:aeconvinf} and \ref{lem:aeconvinf2} in Lemma \ref{lem:aeconv}, \eqref{IS-wlim}, and \eqref{changev}, the inverse Fatou lemma (see for example \cite[Theorem 1.9]{LieLos01} or \cite[Lemma 2.10]{KSV12}) shows that
\begin{align*}
& \limsup _{n \rightarrow \infty} \| \fd ^{s-1/2} e^{-it \fd } (f_n - \phi _n ) \| _{L^6_{t,x}}^6 \\
& \le \limsup _{n \rightarrow \infty} \| \lr{\lambda _n^{-1} \partial _{\mathfrak{x}}}^{s-1/2} e^{-i \lambda _n^2 \mathfrak{t} [\lr{\lambda _n^{-1} \partial _{\mathfrak{x}}}-1]} g_n \| _{L^6_{\mathfrak{t},\mathfrak{x}}}^6 - \| e^{i \mathfrak{t} \partial _{\mathfrak{x}}^2/2} \phi \| _{L^6_{\mathfrak{t},\mathfrak{x}}}^6 \\
& = \limsup _{n \rightarrow \infty} \| \fd ^{s-1/2} e^{-it \fd } f_n \| _{L^6_{t,x}}^6 - \| e^{i \mathfrak{t} \partial _{\mathfrak{x}}^2/2} \phi\| _{L^6_{\mathfrak{t},\mathfrak{x}}}^6 .
\end{align*}
Thus, it suffices to show that
\begin{equation} \label{IS-boubel}
\| e^{i \mathfrak{t} \partial _{\mathfrak{x}}^2/2} \phi\| _{L^6_{\mathfrak{t},\mathfrak{x}}} \gtrsim \eps \left( \frac{\eps}{A} \right) ^{C} .
\end{equation}
From \eqref{IS-duality}, we have
\[
| \lr{h_{\infty}, \phi} _{L^2} | \gtrsim \eps \Big( \frac{\eps}{A} \Big) ^{89} .
\]
By Lemma \ref{lem:compact} and the definition of $h_{\infty}$, 
there exists $C_1>0$ such that $\supp \widehat{h}_{\infty} \subset \{ |\xi| \le ( \frac{A}{\eps} )^{C_1} \}$.
Accordingly, we have $h_{\infty} = P_{\le M} h_{\infty}$ and
\[
|\lr{h_{\infty}, P_{\le M} \phi} _{L^2}| \gtrsim \eps \Big( \frac{\eps}{A} \Big) ^{89} ,
\]
where $M := 2(\frac{A}{\eps} )^{C_1}$.
Let $\chi _r$ denote a smooth cut-off to $\{ |x| \le r \}$.
Since $\chi _r \rightarrow 1$ as $r \rightarrow \infty$, there exists $C_2>0$ so that
\[
|\lr{\wt{h} , \phi}_{L^2}| \gtrsim \eps \Big( \frac{\eps}{A} \Big) ^{89} ,
\]
where $r :=(\frac{A}{\eps} )^{C_2}$ and $\wt{h} := P_{\le M} \chi _r h_{\infty}$.
Hence, the proof of \eqref{IS-boubel} is reduced to prove
\begin{equation} \label{IS-tildehbound}
\sup _{|\mathfrak{t}| \le 1} \| e^{i \mathfrak{t} \partial _{\mathfrak{x}}^2/2} \wt{h} \| _{L_{\mathfrak{x}}^{6/5}} \lesssim \Big( \frac{A}{\eps} \Big) ^{C'}.
\end{equation}
Indeed, if \eqref{IS-tildehbound} holds, we obtain
\begin{align*}
\| e^{i \mathfrak{t} \partial _{\mathfrak{x}}^2/2} \phi\| _{L^6_{\mathfrak{t},\mathfrak{x}}}
& \ge \left\| \| e^{i \mathfrak{t} \partial _{\mathfrak{x}}^2/2} \phi \| _{L_{\mathfrak{x}}^6} \right\| _{L^6_{\mathfrak{t}}([-1,1])} \\
& \gtrsim \Big( \frac{\eps}{A} \Big) ^{C'} \left\| \lr{e^{i \mathfrak{t} \partial _{\mathfrak{x}}^2/2} \wt{h}, e^{i \mathfrak{t} \partial _{\mathfrak{x}}^2/2} \phi}_{L_{\mathfrak{x}}^{6/5},L_{\mathfrak{x}}^6} \right\| _{L^6_{\mathfrak{t}}([-1,1])} \\
& = \Big( \frac{\eps}{A} \Big) ^{C'} \left\| \lr{\wt{h}, \phi}_{L_{\mathfrak{x}}^{2}} \right\| _{L^6_{\mathfrak{t}}([-1,1])}
\gtrsim \eps \left( \frac{\eps}{A} \right) ^{C'+89}.
\end{align*}
In the sequel, we show \eqref{IS-tildehbound}.
Setting $\varrho (\xi ):= e^{-i\mathfrak{t} \xi ^2 /2} \sigma \big( \frac{\xi}{M} \big)$, we write
\[
e^{i \mathfrak{t} \partial _{\mathfrak{x}}^2/2} \wt{h} = \mathcal{F}^{-1}[\varrho ] \ast (\chi _r h_{\infty} ).
\]
Since
\[
\sup _{|\mathfrak{t}| \le 1, \xi \in \R} \left| \frac{d^{\alpha} \varrho}{d \xi ^{\alpha}} (\xi ) \right| \lesssim_{\alpha} M^{\alpha C_1}
\]
for all $\alpha \in \N _0$, we get
\begin{align*}
\| \mathcal{F}^{-1}[\varrho ] \| _{L^1}
\lesssim \| \lr{\cdot} \mathcal{F}^{-1}[ \varrho ] \| _{L^2}
\lesssim \| \varrho \| _{H^1}
\lesssim M^{3/2}.
\end{align*}
Thus, by Young's inequality, we obtain
\begin{align*}
\sup _{|\mathfrak{t}| \le 1} \| e^{i \mathfrak{t} \partial _{\mathfrak{x}}^2/2} \wt{h} \| _{L_{\mathfrak{x}}^{6/5}}
& = \sup _{|\mathfrak{t}| \le 1} \| \mathcal{F}^{-1}[\varrho ] \ast (\chi _r h_{\infty} ) \| _{L_{\mathfrak{x}}^{6/5}}
\le \sup _{|\mathfrak{t}| \le 1} \| \mathcal{F}^{-1}[\varrho ] \| _{L^1} \| \chi _r h_{\infty} \| _{L_{\mathfrak{x}}^{6/5}} \\
& \le \sup _{|\mathfrak{t}| \le 1} \| \mathcal{F}^{-1}[\varrho ] \| _{L^1} \| \chi _r \| _{L^{3}} \| h_{\infty} \| _{L_{\mathfrak{x}}^{2}}
\lesssim M^{3/2} r^{1/3} \\
&\lesssim \left( \frac{A}{\eps} \right) ^{3C_1/2+C_2/3},
\end{align*}
which concludes the proof.
\end{proof}

\begin{rmk}
After passing to a further subsequence, we can take the parameters in the conclusion of Theorem \ref{IS} satisfy the following:
\begin{itemize}
\item If $\lambda_n$ does not convergence to $+\infty$, then $\lambda_n \equiv 1$ and $\nu_n \equiv 0$.
\item Irrespective of the behavior of $\lambda_n$, we have either $t_n/\lambda_n^2 \to \pm \infty$ or $t_n \equiv 0$.
\end{itemize}
This fact follows from the same argument as in Corollary 4.10 in \cite{KSV12}.
\end{rmk}

\section{Linear profile decomposition}
In this section, at first, we state the linear profile decomposition as follows. 
\begin{thm}[Linear profile decomposition]
\label{lpd}
Let $\{ v_n \} \subset H^1(\R )$ be bounded and let $s \in [ \frac{1}{2}, \frac{11}{12})$. Then after passing to a subsequence, there exists $J_0\in [1,\infty]$ such that for any integer $j\in [1,J_0)$, there also exit the following:
\begin{itemize}
\item a function $\phi^j\in L^2(\mathbb{R})\backslash\{0\}$, 
\item a sequence $\{\lambda_n^j\}\subset[1,\infty)$ such that either $\lambda _n^j \rightarrow \infty$ or $\lambda_n^j\equiv 1$,
\item a sequence $\{\nu_n^j\} \subset \R$ such that $\nu_n^j\rightarrow\nu^j\in \R$, which is identically $0$ if $\lambda_n^j\equiv 1$, 
\item a sequence $\{(t_n^j,x_n^j)\}\subset \R\times \R$ such that either $t_n^j/(\lambda_n^j)^2\rightarrow\pm\infty$ or $t_n^j\equiv 0$.
\end{itemize}
Let $P^j_n$ denote the projections defined by\\
\[
P_n^j\phi^j := \begin{cases} \phi^j\in H^1(\R), & \text{if } \lambda_n^j\equiv 1, \\
 P_{\le (\lambda_n^j)^{\theta}}\phi^j, & \text{if } \lambda_n^j\rightarrow\infty, \end{cases}\ \text{with}\ \theta = \frac{1}{100}.
\]
Then for any $J\in [1,J_0)$, we have a decomposition
\begin{equation}
\label{4-1}
    v_n=\sum_{j=1}^JT_{x_n^j}e^{it_n^j\langle\partial_x\rangle}\Lo _{\nu _n^j}P_n^j\phi^j+w_n^J,
\end{equation}
satisfying
\begin{align}
&\lim_{J\rightarrow\infty}\limsup_{n\rightarrow\infty}\|\langle\partial_x\rangle^{s-1/2}e^{-it\langle\partial_x\rangle}w_n^J\|_{L_{t,x}^6(\R\times\R)}=0\label{4-2}\\
& \lim _{n \rightarrow \infty} \left\{ \| v_n \|_{H^1}^2 - \sum_{j=1}^J\| T_{x_n^j} e^{it_n^j \lr{\partial _x}} \Lo _{\nu _n^j} D_{\lambda _n^j} P_n^j \phi^j \| _{H^1}^2 - \| w _n^J \|_{H^1}^2 \right\} =0, \label{LD-dec-2} \\
&D_{\lambda _n^j}^{-1} \Lo _{\nu _n^j}^{-1} T_{x_n^j}^{-1} e^{-it_n^j \lr{\partial _x}} w_n^J \rightharpoonup 0 \text{ weakly in } L^2_x(\mathbb{R})\ \text{for any}\ j\le J. \label{LD-wlim-2}
\end{align}
Finally, we have the following asymptotic orthogonality condition: for any $j\ne j'$,
\begin{equation}
      \lim_{n\rightarrow\infty}\left\{\frac{\lambda_n^j}{\lambda_n^{j'}}+\frac{\lambda_n^{j'}}{\lambda_n^{j}}+\lambda_n^j|\nu_n^j-\nu_n^{j'}|+\frac{\left|s_n^{jj'}\right|}{(\lambda_n^{j'})^2}+\frac{\left|y_n^{jj'}\right|}{\lambda_n^{j'}}\right\}=0,\label{4-5}
\end{equation}
where $(-s_n^{jj'},y_n^{jj'}):=L_{\nu_n^{j'}}(t_n^{j'}-t_n^j,x_n^{j'}-x_n^j)$.
\end{thm}

Theorem \ref{lpd} follows from the inverse Strichartz inequality (Theorem \ref{IS}) and a straightforward modification of the proof of Theorem 5.1 in \cite{KSV12}.
Hence, we omit the details of the proof here.

We will use the following energy decoupling in \S \ref{sec:7}.
Because our energy has the exponential-type term, we need some modifications of the proof of Proposition 5.3 in \cite{KSV12}.

\begin{prop}[Energy decoupling]
Let $\{v_n\}_{n=0}^{\infty}$ be a bounded sequence in $H^1(\R)$. Then after passing to a subsequence, the linear profile decomposition (\ref{4-1}) satisfies the following: for any $J<J_0$, 
\begin{equation}
\label{4-9}
\lim_{n\rightarrow\infty}\left\{E(v_n)-\sum_{j=1}^JE\left(T_{x_n^j}e^{it_n^j\langle\partial_x\rangle}\Lo _{\nu _n^j}D_{\lambda_n^j}P_n^j\phi^j\right)-E(w_n^J)\right\}=0.
\end{equation}
\end{prop}

\begin{proof}
We will prove that the energy decouples in the inverse Strichartz theorem (Theorem \ref{IS}), that is, in the case $J=1$. The general case follows by induction. Moreover, we proved (\ref{LD-dec-2}). Thus it suffices to prove that
\begin{equation}
\label{4-11-1}
\lim_{n\rightarrow\infty}\left\{\int_{\mathbb{R}}\widetilde{\mathcal{N}}(\Re v_n)dx-\int_{\mathbb{R}}\widetilde{\mathcal{N}}(\Re \phi_n)dx-\int_{\mathbb{R}}\widetilde{\mathcal{N}}(\Re w_n)dx\right\}=0,
\end{equation}
where
\[
       \phi_n:=T_{x_n^j}e^{it_n^j\langle\partial_x\rangle}\Lo _{\nu _n^j}D_{\lambda_n^j}P_n^j\phi.
\]
with $P_n=1$ if $\lambda_n\equiv 1$ and $P_n=P_{\le \lambda_n^{\theta}}$ if $\lambda_n\rightarrow\infty$. In order to prove (\ref{4-11-1}), it suffices to prove that for any $l\in \mathbb{N}$ with $l\ge 3$
\begin{equation}
\label{4-12-1}
\lim_{n\rightarrow\infty}\left\{\|\Re v_n\|_{L_x^{2l}}^{2l}-\|\Re \phi_n\|_{L_x^{2l}}^{2l}-\|\Re w_n\|_{L_x^{2l}}^{2l}\right\}=0. 
\end{equation}
Indeed, since $\{v_n\}$ is bounded in $H^1(\mathbb{R})$, where we set $A:=\sup_n\|v_n\|_{H^1}$, by the Sobolev embedding $H^1(\mathbb{R})\subset L^{2l}(\mathbb{R})$, we have
\begin{align*}
     &\sum_{l=3}^{\infty}\frac{1}{l!}\left|\|\Re v_n\|_{L_x^{2l}}^{2l}-\|\Re \phi_n\|_{L_x^{2l}}^{2l}-\|\Re w_n\|_{L_x^{2l}}^{2l}\right|\\
     &\le \sum_{l=3}^{\infty}\frac{1}{l!}\left(\|\Re v_n\|_{H_x^1}^{2l}+\|\Re \phi_n\|_{H_x^1}^{2l}+\|\Re w_n\|_{H_x^1}^{2l}\right)\\
     &\le 3\sum_{l=3}^{\infty}\frac{A^{2l}}{l!} <3\exp (A^2) <\infty,
\end{align*}
which enables us to apply the dominated convergence theorem, to get (\ref{4-11-1})

The proof of (\ref{4-12-1}) is same as the proof of (5.11) in Proposition 5.3 in \cite{KSV12}. So we omit the details, which completes the proof of the proposition.
\end{proof}

\begin{prop}[Decoupling of nonlinear profiles]
\label{denon}
Let $\psi^j$ and $\psi^{j'}$ be in $C_0^{\infty}(\mathbb{R}\times\mathbb{R})$. Let $\nu_n^j,\nu_n^{j'}, (t_n^j,x_n^j), (t_n^{j'},x_n^{j'}), \lambda_n^j, \lambda_n^{j'}$ be parameters given in Theorem \ref{lpd}. We define $\psi_n^j$ by
\begin{equation}
\label{4-12}
\left[\psi_n^j(\cdot+t_n^j,\cdot+x_n^j)\circ L_{\nu_n^j}^{-1}\right](t,x):=\frac{e^{-it}}{\sqrt{\lambda_n^j}}\psi^j\left(\frac{t}{(\lambda_n^j)^2},\frac{x}{\lambda_n^j}\right)
\end{equation}
and $\psi_n^{j'}$ is defined in the similar manner. Then under the orthogonality condition (\ref{4-5}), we have
\begin{equation}
\label{4-13}
\lim_{n\rightarrow\infty}\|\psi_n^j\psi_n^{j'}\|_{L_{t,x}^3}=0.
\end{equation}
\end{prop}

\begin{proof}
This proposition can be proved in a similar manner as Proposition 5.5 in \cite{KSV12}.
\end{proof}

\section{Isolating NLS inside the nonlinear Klein-Gordon equation}
\label{sec:IS}

We recall the result obtained by Dodson \cite{Dod16} for the mass-critical nonlinear Schr\"odinger equation:
\begin{equation} \label{mNLS}
\Big( i \dt + \frac{1}{2} \dx^2 \Big)  w = \frac{5}{32} |w|^4 w .
\end{equation}

\begin{thm}[\cite{Dod16}] \label{scaNLS}
Let $w_0 \in L^2 (\R )$.
Then, there exists a unique global solution $w\in C(\R;L_x^2(\R))$ to \eqref{mNLS} with $w(0)= w_0$.
Furthermore, the solution $w$ satisfies the following estimate:
\[
\| w \| _{L^6_{t,x} (\R \times \R )} \le C(M(w_0)),
\]
As a consequence, $w$ scatters as $t\rightarrow \pm \infty$ in $L^2(\R)$, that is, there exists $w_{\pm} \in L^2 (\R )$ such that
\[
\lim _{t \rightarrow \pm \infty} \| w(t) - e^{it \dx^2 /2} w_{\pm} \| _{L^2_x} =0,
\]
where the double-sign corresponds. Conversely, for any $w_{\pm} \in L^2 (\R )$, there exists a unique global solution $w$ to \eqref{mNLS} so that the above holds.
\end{thm}

\begin{rmk}
The coefficient on the right hand side of \eqref{mNLS} is needed to extract $|v|^4v$ on the nonlinearity of \eqref{rNLKGexp}.
Indeed, we will use the following equality:
\begin{equation} \label{extract}
(\Re z)^5
= \frac{1}{16} ( 10 |z| ^4 \Re z + 5 |z|^2 \Re z^3 + \Re z^5)
= \frac{1}{16} ( 5 |z| ^4 (z+\overline{z}) + 5 |z|^2 \Re z^3 + \Re z^5)
\end{equation}
for $z \in \C$.
When $z = r e^{i \theta}$ in polar form, \eqref{extract} is equivalent to
\[
\cos^5 \theta = \frac{1}{16} ( 10 \cos \theta + 5 \cos 3\theta + \cos 5\theta).
\]
\end{rmk}

The goal in this section is to prove the following theorem.

\begin{thm} \label{isolate}
Let $\nu _n \rightarrow \nu \in \R$, $\lambda _n \rightarrow \infty$, and $\{ t_n \} , \, \{ x_n \} \subset \R$ be given.
Assume that either $t_n \equiv 0$ or $t_n /\lambda_n^2 \rightarrow \pm \infty$.
Let $\phi \in L^2(\R )$.
If we define
\[
\phi _n := T_{x_n} e^{it_n \fd } \Lo _{\nu _n} D_{\lambda _n} P_{\le \lambda _n^{\theta}} \phi
\]
for $\theta = \frac{1}{100}$, then for each $n$ sufficiently large, there exists a unique global solution $v_n$ to \eqref{rNLKGexp} with initial data $v_n(0)= \phi _n$, which satisfies
\[
\| v_n \| _{L_t^{\infty} (\R;H_x^1(\R))} + S_{\R} (\fd ^{1/2} v_n) \lesssim _{M(\phi )} 1.
\]
Furthermore, for any $s \in[ \frac{1}{2},1]$ and $\eps >0$, there exist $N_{\eps} \in \N$ and a function $\psi _{\eps} \in C_c^{\infty} (\R \times \R )$ such that for all $n >N_{\eps}$,
\[
\left\| \Re \Big\{ (\fd^{s-1/2} v_n) \circ L_{\nu _n}^{-1} (t+ \wt{t}_n, x+ \wt{x}_n) - \frac{e^{-it}}{\lambda _n^{1/2}} \psi _{\eps} \Big( \frac{t}{\lambda _n^2}, \frac{x}{\lambda _n} \Big) \Big\} \right\| _{L_{t,x}^6} < \eps ,
\]
where $(\wt{t}_n, \wt{x}_n) := L_{\nu _n} (t_n ,x_n)$.
\end{thm}

\begin{proof}
First, we consider the case $\nu _n \equiv 0$.
Then, since the Klein-Gordon equation is invariant under space translations, we may assume $x_n \equiv 0$.
That is, $\phi _n = e^{it_n \fd } D_{\lambda _n} P_{\le \lambda _n^{\theta}} \phi$ and we will show that
\[
\left\| (\fd^{s-1/2} v_n) (t+t_n ,x) - \frac{e^{-it}}{\lambda _n^{1/2}} \psi _{\eps} \Big( \frac{t}{\lambda _n^2}, \frac{x}{\lambda _n} \Big) \right\| _{L_{t,x}^6} < \eps .
\]

\begin{itemize}
\item In the case $t_n \equiv 0$, we let $w_n$ and $w_{\infty}$ be the solutions to \eqref{mNLS} with $w_n(0) = P_{\le \lambda _n^{\theta}} \phi$ and $w_{\infty} (0) = \phi$, respectively.

\item In the case $t_n/\lambda _n^2 \rightarrow -\infty$ ($+\infty$), we let $w_n$ and $w_{\infty}$ be the solutions to  \eqref{mNLS} that scatter forward (backward) to $e^{it \dx^2 /2} P_{\le \lambda _n^{\theta}} \phi$ and $e^{it \dx^2 /2} \phi$, respectively.
\end{itemize}
Theorem \ref{scaNLS} implies
\begin{equation} \label{boundwn}
S_{\R} (w_n) + S_{\R} (w_{\infty} ) \lesssim _{M(\phi )} 1.
\end{equation}

From the construction of $w_n$, we get the following.

\begin{lem} \label{lem:cor-per}
For $s \ge 0$, we have
\begin{align}
& \| |\dx |^s w_n\| _{L_t^{\infty} L_x^2} + \| |\dx |^s w_n\| _{L_{t,x}^6} \lesssim _{M(\phi )} \lambda _n^{s \theta}, \label{cor-per1} \\
& \| \fd ^s \dt w_n \| _{L_{t,x}^6} \lesssim _{M(\phi )} \lambda _n^{(s+2) \theta} . \notag
\end{align}
Furthermore, the identity is valid:
\begin{equation} \label{cor-per2}
\lim _{n \rightarrow \infty} \left\{ \| w_n-w_{\infty}\| _{L_t^{\infty} L_x^2} + \| w_n-w_{\infty}\| _{L_{t,x}^6} + \| D_{\lambda _n} (w_n - P_{\le \lambda _n^{\theta}} w_{\infty}) \| _{L_t^{\infty} H^{s}_x} \right\} =0.
\end{equation}
\end{lem}

\begin{proof}
Since
\[
\||\dx |^s P_{\le \lambda _n^{\theta}} \phi\| _{L^2} \lesssim \lambda _n^{s \theta} \|\phi\| _{L^2} ,
\]
\eqref{cor-per1} follows from a corollary of the local well-posedness.

By the fractional Leibniz rule and Sobolev embeddings $\dot{W}^{2/15,6} (\R ) \hookrightarrow L^{30} (\R )$, $\dot{H}^{7/15} (\R ) \hookrightarrow L^{30} (\R )$, we have
\begin{align*}
\| \fd ^s \dt w_n \| _{L_{t,x}^6}
& \lesssim \| \fd ^s \dx^2 w_n \| _{L_{t,x}^6} + \| \fd ^s w_n \| _{L_t^6 L_x^{30}} \| w_n \| _{L_t^{\infty} L_x^{30}}^4 \\
& \lesssim _{M(\phi )} \lambda _n^{(s+2)\theta} + \| \fd ^{s+2/15} w_n \| _{L_{t,x}^6} \| \fd ^{7/15} w_n \| _{L_t^{\infty} L_x^{2}}^4 \\
& \lesssim _{M(\phi )} \lambda _n^{(s+2)\theta}.
\end{align*}

The estimates for the first and second parts on the left hand side of \eqref{cor-per2} follows from the stability theory for the mass-critical Schr\"odinger equation (see \cite{TVZ08, KilVis13}).
We consider the third part on the left hand side of \eqref{cor-per2}.
From
\[
w_n-P_{\le \lambda ^{\theta}} w_{\infty}
= P_{\ge \lambda _n} w_n + P_{\le \lambda _n} (w_n-w_{\infty}) + P_{\lambda _n^{\theta} \le \cdot \le \lambda _n} w_{\infty} ,
\]
we get
\begin{align*}
& \| D_{\lambda _n} (w_n - P_{\le \lambda _n^{\theta}} w_{\infty})\| _{L_t^{\infty} H^{s}_x} \\
& = \| \lr{\lambda _n ^{-1} \dx}^s (w_n - P_{\le \lambda _n^{\theta}} w_{\infty})\| _{L_t^{\infty} L_x^2} \\
& \lesssim \lambda _n^{-s} \| |\dx |^s w_n\| _{L_t^{\infty} L_x^2} + \| w_n-w_{\infty}\| _{L_t^{\infty} L_x^2} + \| P_{\ge \lambda _n^{\theta}} w_{\infty}\| _{L_t^{\infty} L_x^2}.
\end{align*}
Since
\[
\| P_{\ge \lambda _n^{\theta}} w_{\infty}\| _{L_t^{\infty} L_x^2([T,\infty ) \times \R )}
\lesssim \| w_{\infty}-e^{it\dx^2 /2} w_+\| _{L_t^{\infty} L_x^2 ([T, \infty ) \times \R )} + \|P_{\ge \lambda _n^{\theta}} w_+\| _{L_x^2} ,
\]
choosing $T$ sufficiently large, we can make the first part on the right hand side small.
Then, letting $\lambda _n \rightarrow \infty$, we have
\[
\lim _{n \rightarrow \infty} \|P_{\ge \lambda _n^{\theta}} w_{\infty}\| _{L_t^{\infty} L_x^2([-T,T] \times \R )} =0,
\]
which shows the desired bound.
\end{proof}

Let $T \gg 1$ to be determined later.
We define
\[
\wt{v}_n(t) :=
\begin{cases}
e^{-it} D_{\lambda _n} w_n (t/\lambda _n^2), & \text{if } |t| \le T \lambda _n^2 , \\
e^{-i(t-T\lambda _n^2) \fd } \wt{v}_n(T \lambda _n^2), & \text{if } t > T\lambda _n ^2 , \\
e^{-i(t+T\lambda _n^2) \fd } \wt{v}_n(-T \lambda _n^2), & \text{if } t <- T\lambda _n ^2 .
\end{cases}
\]
By the Strichartz estimate (Lemma \ref{Str}) amd Lemma \ref{lem:cor-per}, we have
\begin{equation} \label{eq:iso-bound-sta0}
\begin{aligned}
\| |\dx|^{s} \wt{v}_n \| _{L_{t,x}^6}
&\lesssim \| |\dx|^s D_{\lambda_n} w_n (t/\lambda_n^2) \|_{L_{t,x}^6} + \| |\dx|^s e^{-i (t-T\lambda_n^2)} D_{\lambda_n} w_n (T) \|_{L_{t,x}^6 (T\lambda_n^2, \infty)} \\
&\quad + \| |\dx|^s e^{-i (t+T\lambda_n^2)} D_{\lambda_n} w_n (T) \|_{L_{t,x}^6 (-\infty,-T\lambda_n^2, \infty)} \\
&\lesssim \lambda_n^{-s} \| |\dx|^s w_n \|_{L_{t,x}^6} + \lambda_n^{-s} \| |\dx|^{s} \lr{\lambda_n^{-1} \dx}^{1/2} w_n \|_{L_t^{\infty} L_x^2} ) \\
&\lesssim \lambda_n^{-s(1-\theta)}
\end{aligned}
\end{equation}
for any $s \ge 0$.
Moreover, \eqref{boundwn} implies that
\begin{equation} \label{eq:iso-bound-sta}
\begin{aligned}
\| \wt{v}_n \| _{L_t^{\infty} H_x^{2}} + \| \fd ^{3/2} \wt{v}_n \| _{L_{t,x}^6}
& \lesssim \| D_{\lambda _n} w_n \| _{L_t^{\infty} H_x^{2}} + \| \wt{v}_n \| _{L_{t,x}^6} + \| |\dx|^{3/2} \wt{v}_n \|_{L_{t,x}^6} \\
& \lesssim _{M(\phi )} 1+\lambda _n^{-2} \| |\dx |^2 w_n \| _{L_t^{\infty} L_x^2} + \lambda _n^{-3(1-\theta)/2} \\
& \lesssim _{M(\phi )} 1.
\end{aligned}
\end{equation}

Here, we note that
\begin{equation} \label{eq:iso-inidata}
\lim _{T \rightarrow \infty} \limsup _{n \rightarrow \infty} \| \wt{v}_n (-t_n) - \phi _n \| _{H^{2}} =0 .
\end{equation}
Indeed, in the case $t_n \equiv 0$, \eqref{eq:iso-inidata} holds because $\wt{v}_n(-t_n) = D_{\lambda _n} P_{\le \lambda _n^{\theta}} \phi = \phi _n$.
On the other hand, in the case $t_n / \lambda _n^2 \rightarrow -\infty$, for any $T>0$ and $n \gg 1$, we have $-t_n> T\lambda _n^2$.
Hence,
\begin{align*}
& \|\wt{v}_n (-t_n) - \phi _n\| _{H^{2}} \\
& = \|D_{\lambda_n} w_n(T) - e^{-i T\lambda _n^2 (\fd -1)} D_{\lambda _n} P_{\le \lambda _n^{\theta}} \phi\| _{H^{2}} \\
& \le \|D_{\lambda _n}  [ w_n(T) - P_{\le \lambda _n^{\theta}} w_{\infty}(T) ]\| _{H^{2}} + \|D_{\lambda _n} P_{\le \lambda _n^{\theta}} [w_{\infty} (T) - e^{-i T \lambda _n^2 (\lr{\lambda _n^{-1} \dx}-1)}\phi ]\| _{H^{2}}
\end{align*}
Thanks to Lemma \ref{lem:cor-per}, the first part goes to zero.
The second part is estimated as follows.
\begin{align*}
& \|D_{\lambda _n} P_{\le \lambda _n^{\theta}} [w_{\infty} (T) - e^{-i T \lambda _n^2 (\lr{\lambda _n^{-1} \dx}-1)}\phi ]\| _{H^{2}} \\
& = \|\lr{\lambda _n^{-1} \dx}^2 P_{\le \lambda _n^{\theta}} [w_{\infty} (T) - e^{-i T \lambda _n^2 (\lr{\lambda _n^{-1} \dx}-1)}\phi ]\| _{L^2} \\
& \lesssim \| P_{\le \lambda _n^{\theta}}  [w_{\infty} (T) - e^{-i T \lambda _n^2 (\lr{\lambda _n^{-1} \dx}-1)}\phi ] \| _{L^2} \\
& \lesssim \|w_{\infty}(T) - e^{iT \dx^2 /2} \phi\| _{L^2} + \| \bm{1}_{|\xi| \le 2 \lambda_n^{\theta}} [1-e^{-i T \{ \lambda _n^2 (\lr{\lambda _n^{-1} \xi}-1)-\frac{1}{2} |\xi |^2\}}] \hat{\phi}\| _{L^2} .
\end{align*}
By \eqref{eq:KG-S} and Lebesgue's dominated convergence theorem, \eqref{eq:iso-inidata} holds.

\begin{prop}[Large time intervals] \label{prop:largedataint}
With the notation above,
\[
\lim _{T \rightarrow \infty} \limsup _{n \rightarrow \infty} \| \fd ^{3/2} \wt{v}_n \| _{L_{t,x}^6 ((T\lambda _n^2, \infty ) \times \R )} =0
\]
and analogously on the time interval $(-\infty , -T\lambda _n^2)$.
\end{prop}

\begin{proof}
From Theorem \ref{scaNLS}, there exists $w_+ \in L^2 (\R )$ such that
\[
\lim _{t \rightarrow \infty} \|w_{\infty} (t) - e^{it \dx^2 /2} w_+\| _{L^2} =0 .
\]
We decompose the integrand as follows: For $t> T\lambda_n^2$,
\begin{align*}
\wt{v}_n (t)
& = e^{-i(t- T\lambda _n^2) \fd } \wt{v}_n(T \lambda _n^2) \\
& = e^{-i(t- T\lambda _n^2) \fd } \left( \wt{v}_n(T \lambda _n^2) - e^{-iT \lambda _n^2} D_{\lambda_n} e^{iT\dx^2 /2} P_{\le \lambda _n^{\theta}} w_+ \right) \\
& \qquad + e^{-i(t- T\lambda _n^2) \fd } e^{-iT \lambda _n^2} D_{\lambda_n} e^{iT\dx^2 /2} P_{\le \lambda _n^{\theta}} w_+ \\
& =: h_1(t) + h_2(t).
\end{align*}
The Strichartz estimate (Lemma \ref{Str}) implies
\begin{align*}
& \| \fd ^{3/2} h_1 \| _{L_{t,x}^6 ((\lambda _n^2 T, \infty ) \times \R )} \\
& \lesssim \left\| \wt{v}_n(T \lambda _n^2) - e^{-iT \lambda _n^2} D_{\lambda_n} e^{iT\dx^2 /2} P_{\le \lambda _n^{\theta}} w_+ \right\| _{H^{2}} \\
& = \left\| e^{-iT \lambda _n^2} D_{\lambda _n} \left( w_n(T) - e^{iT\dx^2 /2} P_{\le \lambda _n^{\theta}} w_+ \right) \right\| _{H^{2}} \\
& \le \| D_{\lambda _n} (w_n(T) - P_{\le \lambda _n^{\theta}} w_{\infty}(T))\| _{H^{2}} + \left\| D_{\lambda _n} P_{\le \lambda _n^{\theta}} \left( w_{\infty}(T) - e^{iT\dx^2 /2} w_+ \right) \right\| _{H^{2}} .
\end{align*}
By Lemma \ref{lem:cor-per}, we can estimate the first part.
On the estimate of the second part,
\begin{align*}
& \left\| D_{\lambda _n} P_{\le \lambda _n^{\theta}} \left( w_{\infty}(T) - e^{iT\dx^2 /2} w_+ \right) \right\| _{H^{2}} \\
& \lesssim \left\| w_{\infty}(T) - e^{iT\dx^2 /2} w_+ \right\| _{L^2} + \lambda _n^{-2} \left\| |\dx |^2 P_{\le \lambda _n^{\theta}} \left( w_{\infty}(T) - e^{iT\dx^2 /2} w_+ \right) \right\| _{L^2} \\
& \lesssim \left\| w_{\infty}(T) - e^{iT\dx^2 /2} w_+ \right\| _{L^2} \rightarrow 0 \quad (T \rightarrow \infty ).
\end{align*}

Next, we estimate the Strichartz norm of $h_2$.
By the Strichartz estimate and
\[
\| \fd ^{1/2} D_{\lambda_n} P_{\le \lambda _n^{\theta}} (f-g) \| _{L_x^2}
\lesssim \| f-g \| _{L^2},
\]
we may assume that $w_+$ is a Schwartz function with compact frequency support.
Then, we have
\begin{equation} \label{boundh2}
\begin{aligned}
& \| \fd ^{3/2} h_2 \| _{L_{t,x}^6 ((T \lambda _n^2, \infty ) \times \R )} \\
& = \| D_{\lambda_n} \lr{\lambda _n^{-1} \dx}^{3/2} e^{-i(t- \lambda _n^2 T) \lr{\dx/\lambda _n}} e^{-iT \lambda _n^2} e^{iT\dx^2 /2} P_{\le \lambda _n^{\theta}} w_+ \| _{L_{t,x}^6 ((T \lambda _n^2, \infty ) \times \R )} \\
& \lesssim \lambda _n^{-1/3} \| e^{-i(t- \lambda _n^2 T) \lr{\dx/\lambda _n}} e^{-iT \lambda _n^2} e^{iT\dx^2 /2} P_{\le \lambda _n^{\theta}} w_+ \| _{L_{t,x}^6 ((T \lambda _n^2, \infty ) \times \R )}.
\end{aligned}
\end{equation}
We write
\begin{equation} \label{zetacon}
e^{-i(t- \lambda _n^2 T) \lr{\dx/\lambda _n}} e^{-iT \lambda _n^2} e^{iT\dx^2 /2} P_{\le \lambda _n^{\theta}} w_+
= \zeta \ast w_+,
\end{equation}
where 
\[
\zeta (x) = \frac{1}{\sqrt{2\pi}} \int _{\R} e^{i k(\xi )} \sigma \Big( \frac{\xi}{\lambda _n^{\theta}} \Big) d\xi , \quad
k(\xi ) := x \xi - (t-T\lambda _n^2) \lr{\lambda _n^{-1} \xi} - T \lambda _n^2 - T \xi ^2/2.
\]
We note that
\begin{align*}
& k'(\xi ) = x - (\lambda _n^{-2}t-T) \frac{\xi}{\lr{\lambda _n^{-1} \xi}} - T \xi , \\
& k''(\xi ) = - (\lambda _n^{-2} t-T) \frac{1}{\lr{\lambda _n^{-1} \xi}^3} - T .
\end{align*}
Since $k'$ is a strictly decreasing function and $k'(\xi ) \rightarrow \mp \infty$ as $\xi \rightarrow \pm \infty$, there exists $\xi _0 \in \R$ such that $k'(\xi _0)=0$.
Moreover,
\begin{equation*}
|k'' (\xi )| \sim \lambda _n^{-2} t
\end{equation*}
for $|\xi | \le 2 \lambda _n^{\theta}$ and $t \ge T \lambda _n^2$.
Indeed, for $t \ge 2 T \lambda _n^2$,
\[
\frac{1}{2} \lambda_n^{-2}t \le \lambda _n^{-2} t- T \lesssim |k''(\xi)| \le |\lambda_n^{-2}t-T|+T = \lambda_n^{-2}t .
\]
On the other hand, for $T \lambda _n^2 \le t \le 2 T \lambda_n^2$,
\[
\frac{1}{2} \lambda_n^{-2}t \le T \le |k''(\xi)| \le |\lambda_n^{-2}t-T| +T = \lambda_n^{-2}t .
\]

Let $\eps >0$ be a positive constant to be chosen later.
If $|\xi_0|\le 4\lambda_n^{\theta}$, we get
\[
|k'(\xi )| = |k'(\xi ) - k'(\xi _0)| = \left| (\xi - \xi _0) \int _0^1 k'' (s \xi + (1-s) \xi _0) ds \right|
\gtrsim |\xi-\xi_0| \lambda_n^{-2}t
\]
for $|\xi|\le 2\lambda_n^{\theta}$.
On the other hand, if $\xi_0>4\lambda_n^{\theta}$, noting that $k'(4\lambda_n^{\theta})>0$ because $k'$ is decreasing, we get
\[
|k'(\xi)| = k'(\xi) >k'(\xi)-k'(4\lambda_n^{\theta}) = (\xi-4\lambda_n^{\theta}) \int_0^1 k'' (s \xi + (1-s) 4\lambda_n^{\theta}) ds 
\gtrsim |\xi-4\lambda_n^{\theta}| \lambda _n^{-2} t
\]
for $|\xi|\le 2\lambda_n^{\theta}$.
The case $\xi_0<-4\lambda_n^{\theta}$ is similarly handled.
Therefore, by the integration by parts, we have
\begin{align*}
& \left| \int _{[-2\lambda _n^{\theta}, 2\lambda _n^{\theta}] \backslash [\xi _0 - \eps , \xi _0 + \eps ]} e^{i k(\xi )} \sigma \Big( \frac{\xi}{\lambda _n^{\theta}} \Big) d \xi \right| \\
& \lesssim \frac{1}{\eps \lambda _n^{-2}t} + \int _{[-2\lambda _n^{\theta}, 2\lambda _n^{\theta}] \backslash [\xi _0 - \eps , \xi _0 + \eps ]} \bigg( \frac{|k''(\xi)|}{k'(\xi)^2} + \frac{\lambda_n^{-\theta}}{|k'(\xi)|} \bigg) d \xi
\lesssim \frac{\lambda _n^2}{t \eps} .
\end{align*}
Thus, we obtain
\begin{align*}
| \zeta  (x)|
& \lesssim \left| \int _{[-2\lambda _n^{\theta}, 2\lambda _n^{\theta}] \backslash [\xi _0 - \eps , \xi _0 + \eps ]} e^{i k(\xi )}\sigma \Big( \frac{\xi}{\lambda _n^{\theta}} \Big) d \xi \right| + \left| \int _{[\xi _0 - \eps , \xi _0 + \eps ]} e^{i k(\xi )} \sigma \Big( \frac{\xi}{\lambda _n^{\theta}} \Big)  d \xi \right| \\
& \lesssim \frac{\lambda _n^2}{t \eps} + \eps .
\end{align*}
Here, taking $\eps = \lambda _n/t^{1/2}$, we get
\begin{equation*}
\| \zeta \| _{L_x^{\infty}} \lesssim \frac{\lambda _n}{t^{1/2}}
\end{equation*}
for $t \ge \lambda _n T$.

Hence, we have
\[
\| \zeta \ast w_+ \| _{L_x^{\infty}}
\lesssim \| \zeta \| _{L_x^{\infty}} \| w_+ \| _{L_x^1}
\lesssim \frac{\lambda _n}{t^{1/2}} .
\]
By \eqref{zetacon} and interpolating this estimate with the trivial $L^2$ bound, we obtain
\[
\| e^{-i(t- T \lambda _n^2) \lr{\dx/\lambda _n}} e^{-iT \lambda _n^2} e^{iT\dx^2 /2} P_{\le \lambda _n^{\theta}} w_+ \| _{L_x^6} \lesssim \left( \frac{\lambda _n}{t^{1/2}} \right) ^{2/3} .
\]
Integrating with respect to time, we have
\[
\text{R.H.S. of \eqref{boundh2}}
\lesssim T^{-1/6},
\]
which concludes the proof.
\end{proof}

On the middle interval $I_n := [-T\lambda _n^2, T\lambda _n^2]$, because $w_n$ is a solution to \eqref{mNLS}, a direct calculation shows that
\begin{align*}
& -i \dt \wt{v}_n \\
& = e^{-it} \left( -D_{\lambda_n} w_n \Big( \frac{t}{\lambda _n^2} \Big)-i \lambda _n^{-2} D_{\lambda _n} (\dt w_n) \Big( \frac{t}{\lambda _n^2} \Big) \right) \\
& = e^{-it} \left( -D_{\lambda_n} w_n \Big( \frac{t}{\lambda _n^2} \Big) + \frac{1}{2 \lambda _n^2} D_{\lambda _n} (\dx^2 w_n) \Big( \frac{t}{\lambda _n^2} \Big) -\frac{5}{32} \lambda _n^{-2} D_{\lambda _n} (|w_n|^4w_n) \Big( \frac{t}{\lambda _n^2} \Big) \right) \\
& = e^{-it} D_{\lambda _n} \left( -w_n \Big( \frac{t}{\lambda _n^2} \Big) + \frac{1}{2 \lambda _n^2} (\dx^2 w_n) \Big( \frac{t}{\lambda _n^2} \Big) \right) -\frac{5}{32} |\wt{v}_n|^4 \wt{v}_n .
\end{align*}
Moreover, \eqref{extract} implies that
\[
(\Re \wt{v}_n)^5 - \frac{5}{16} |\wt{v}_n|^4 \wt{v}_n
= \frac{5}{16} |\wt{v}_n|^4 \overline{\wt{v}_n} + \frac{5}{16} |\wt{v}_n|^2 \Re \wt{v}_n^3 + \frac{1}{16} \Re \wt{v}_n^5 .
\]
Hence, $\wt{v}_n$ on $I_n$ satisfies
\[
(-i \dt + \fd ) \wt{v}_n + \fd ^{-1} (\exp (|\Re \wt{v}_n|^2)-1-|\Re \wt{v}_n|^2) \Re \wt{v}_n = e_1+e_2+e_3+e_4+e_5+e_6 ,
\]
where
\begin{align*}
& e_1 := e^{-it} D_{\lambda _n} \left\{ \Big[ \lr{\lambda _n^{-1} \dx} -1 +\frac{\dx^2}{2\lambda _n^2} \Big] w_n \Big( \frac{t}{\lambda _n^2} \Big) \right\} , \\
& e_2 :=  \fd ^{-1} \left( \exp (|\Re \wt{v}_n|^2)-1-|\Re \wt{v}_n|^2-\frac{|\Re \wt{v}_n|^4}{2} \right) \Re \wt{v}_n , \\
& e_3 := \frac{1}{2} [\fd ^{-1}-1] (\Re \wt{v}_n)^5, \\
& e_4 := \frac{1}{32} \Re \left\{ e^{-5it} \left[ D_{\lambda _n} w_n \Big( \frac{t}{\lambda _n^2} \Big) \right] ^5 \right\} , \\
& e_5 := \frac{5}{32} \left| D_{\lambda _n} w_n \Big( \frac{t}{\lambda _n^2} \Big) \right| ^2 \Re \left\{ e^{-3it} \left[ D_{\lambda _n} w_n \Big( \frac{t}{\lambda _n^2} \Big) \right]^3 \right\}, \\
& e_6 := \frac{5}{32}  e^{it} \left| D_{\lambda _n} w_n \Big( \frac{t}{\lambda _n^2} \Big) \right| ^4 \overline{D_{\lambda _n} w_n \Big( \frac{t}{\lambda _n^2} \Big)} .
\end{align*}

We can treat $e_1$, $e_2$, and $e_3$ as errors in Proposition \ref{prop:stability}.
In fact, the equation
\begin{align*}
\lr{\lambda _n^{-1} \xi}-1 - \frac{|\xi |^2}{2\lambda _n^2}
& = \frac{|\xi |^2}{\lambda _n^2 (1+\lr{\lambda _n^{-1} \xi})} - \frac{|\xi |^2}{2\lambda _n^2} \\
& = \frac{|\xi |^2}{2\lambda _n^2} \frac{1-\lr{\lambda _n^{-1} \xi}}{1+\lr{\lambda _n^{-1} \xi}}
= \frac{|\xi |^4}{2\lambda _n^4} \frac{1}{(1+\lr{\lambda _n^{-1} \xi})^2}
\end{align*}
and Lemma \ref{lem:cor-per} imply that
\begin{equation} \label{est:error1}
\|e_1\| _{L_t^1 H_x^{2}(I_n \times \R )}
\lesssim T \lambda _n^{-2} \left( \|\dx^4 w_n\| _{L_t^{\infty} L_x^2} + \lambda _n^{-2} \||\dx |^{6} w_n\| _{L_t^{\infty} L_x^2} \right)
\lesssim T \lambda _n^{-2+4 \theta} .
\end{equation}
The Taylor expansion and Lemma \ref{lem:cor-per} yield
\begin{equation} \label{est:error2'}
\begin{aligned}
& \|\fd ^{5/2} e_2\| _{L_{t,x}^{6/5}(I_n \times \R )} \\
& \lesssim \sum _{l=3}^{\infty} \frac{1}{l!} \| \fd ^{3/2} (\Re \wt{v}_n)^{2l+1} \| _{L_{t,x}^{6/5} (I_n \times \R )} \\
& \lesssim \sum _{l=3}^{\infty} \frac{1}{(l-2)!} \| \fd ^{3/2} \Re \wt{v}_n \| _{L_{t,x}^{6} (I_n \times \R )} \| \Re \wt{v}_n \| _{L_{t,x}^{6} (I_n \times \R )}^{4} \| \Re \wt{v}_n \| _{L_{t,x}^{\infty} (I_n \times \R )}^{2l-4} \\
& \lesssim \| \lr{\lambda _n^{-1} \dx}^{3/2} w_n \| _{L_{t,x}^{6}}^{5} \sum _{l=3}^{\infty} \frac{1}{(l-2)!} \left( \lambda _n^{-1/2} \| w_n \| _{L_{t}^{\infty} H_x^1([-T,T] \times \R )} \right) ^{2l-4} \\
& \lesssim \sum _{l=3}^{\infty} \frac{1}{(l-2)!} \left( \lambda _n^{-1/2+\theta} \right) ^{2l-4} \\
& \lesssim \lambda _n^{-1+2\theta}.
\end{aligned}
\end{equation}
Moreover, by
\[
\lr{\xi} (\lr{\xi}^{-1} -1) = -\xi \frac{\xi}{1+\lr{\xi}} ,
\]
the fractional Leibniz rule, and Lemma \ref{lem:cor-per}, we have
\begin{equation} \label{est:error2}
\begin{aligned}
\|\fd ^{5/2} e_3 \| _{L_{t,x}^{6/5}(I_n \times \R )}
& \lesssim \| \fd ^{3/2} \dx (\Re \wt{v}_n)^5\| _{L_{t,x}^{6/5} (I_n \times \R )} \\
& \lesssim \lambda _n^{-1} \left\| D_{\lambda _n} (\lr{\lambda _n^{-1} \dx}^{3/2} \dx w_n) \Big( \frac{t}{\lambda _n^2} \Big) \right\| _{L_{t,x}^6} \left\| D_{\lambda _n} w_n \Big( \frac{t}{\lambda _n^2} \Big) \right\| _{L_{t,x}^6}^4 \\
& \lesssim \lambda _n^{-1} (\| \dx w_n \| _{L_{t,x}^6} + \lambda _n^{-3/2} \| |\dx|^{5/2} w_n \| _{L_{t,x}^6}) \| w_n \| _{L_{t,x}^6}^4 \\
& \lesssim \lambda _n^{-1+\theta} .
\end{aligned}
\end{equation}

Unfortunately, the parts $e_4$, $e_5$, $e_6$ are not small in either of the spaces $L_t^1H_x^{2}$ or $L_t^{6/5} W_x^{5/2,6/5}$.
However, they oscillate in space-time like $e^{it}$, and this allows us to modify $\wt{v}_n$, which approximately solves \eqref{rNLKGexp}.

\begin{lem} \label{lem:moderror}
For $j=4,5,6$, let $f_{n,j}$ solve
\[
(-i \dt + \fd ) f_{n,j} = e_j , \quad f_{n,j}(0) = 0 .
\]
Then, \[
\|f_{n,j}\| _{L_t^{\infty} H_x^{2}(I_n \times \R )} + \| \fd ^{3/2} f_{n,j}\| _{L_{t,x}^{6}(I_n \times \R )} \lesssim \lambda _n^{-2+2\theta} .
\]
\end{lem}

\begin{proof}
We will prove the lemma for $j=6$.
The argument for $j=4, 5$ is almost identical.
We compute
\begin{align*}
& (-i\dt + \fd ) \Big( f_{n,6} - \frac{1}{2} e_6 \Big) \\
& = \frac{5}{64} \Bigg\{ \frac{ie^{it}}{\lambda _n^{9/2}} [ \dt (w_n^2 \overline{w_n}^3)] \Big( \frac{t}{\lambda _n^2} , \frac{x}{\lambda _n} \Big) - \frac{e^{it}}{\lambda _n^{5/2}} [(\lr{\lambda _n^{-1} \dx}-1) (w_n^2 \overline{w_n}^3)] \Big( \frac{t}{\lambda _n^2}, \frac{x}{\lambda _n} \Big) \Bigg\} .
\end{align*}
Lemma \ref{lem:cor-per} and the fractional Leibniz rule yield that
\begin{align*}
& \left\| \fd ^{5/2} \left( \frac{ie^{it}}{\lambda _n^{9/2}} [ \dt (w_n^2 \overline{w_n}^3)] \Big( \frac{t}{\lambda _n^2}, \frac{x}{\lambda _n} \Big) \right) \right\| _{L_{t,x}^{6/5}(I_n \times \R )} \\
& \lesssim \lambda _n^{-2} \| \lr{\lambda _n^{-1} \dx}^{5/2} \dt ( w_n^2 \overline{w_n}^3) \| _{L_{t,x}^{6/5}} \\
& \lesssim \lambda _n^{-2} \left\{ \| \lr{\lambda _n^{-1} \dx}^{5/2} \dt w_n\| _{L_{t,x}^6} \| w_n \| _{L_{t,x}^6}^4 + \| \dt w_n\| _{L_{t,x}^6} \|\lr{\lambda _n^{-1} \dx}^{5/2} w_n\| _{L_{t,x}^6} \|w_n\| _{L_{t,x}^6}^3 \right\} \\
& \lesssim \lambda _n^{-2+2\theta} .
\end{align*}
On the other hand, from
\[
\fd (\fd -1) = -\dx^2 \frac{1}{1+\fd ^{-1}} ,
\]
and Lemma \ref{lem:cor-per}, we get
\begin{align*}
& \left\| \fd ^{5/2} \left( \frac{e^{it}}{\lambda _n^{5/2}} [(\lr{\lambda _n^{-1} \dx}-1) (w_n^2 \overline{w_n}^3)] \Big( \frac{t}{\lambda _n^2}, \frac{x}{\lambda _n} \Big) \right) \right\| _{L_{t,x}^{6/5}(I_n \times \R )} \\
& \lesssim \lambda _n^{-2} \| \lr{\lambda _n^{-1} \dx}^{3/2} \dx^2 (w_n^2 \overline{w_n}^3)\| _{L_{t,x}^{6/5}} \\
& \lesssim \lambda _n^{-2} \Big\{ \|\dx^2 w_n\| _{L_{t,x}^6} \|w_n\| _{L_{t,x}^6}^4 + \|\dx w_n\| _{L_{t,x}^6}^2 \|w\| _{L_{t,x}^6}^3 \\
& \hspace*{100pt} + \lambda _n^{-2} \big( \|\dx^4 w_n\| _{L_{t,x}^6} \|w_n\| _{L_{t,x}^6}^4 + \|\dx w_n\| _{L_{t,x}^6}^4 \|w\| _{L_{t,x}^6} \big) \Big\} \\
& \lesssim \lambda _n^{-2+2\theta} .
\end{align*}

By $\dot{H}^{2/5} (\R ) \hookrightarrow L^{10}(\R )$, and Lemma \ref{lem:cor-per}, we have
\begin{align*}
\| e_6 \| _{L_t^{\infty} H^{2}(I_n \times \R )}
& \lesssim \lambda_n^{-2} \| \lr{\lambda _n^{-1} \dx}^2(w_n^2 \overline{w_n}^3) \| _{L_t^{\infty} L_x^{2}} \\
& \lesssim \lambda _n^{-2} \|\lr{\lambda _n^{-1} \dx}^2 w_n\| _{L_t^{\infty} L_x^{10}} \|w_n\| _{L_t^{\infty} L_x^{10}}^4 \\
& \lesssim \lambda _n^{-2} \||\dx |^{2/5} \lr{\lambda _n^{-1} \dx}^2 w_n\| _{L_t^{\infty} L_x^2} \||\dx |^{2/5} w_n\| _{L_t^{\infty} L_x^2}^4 \\
& \lesssim \lambda _n^{-2+2\theta} .
\end{align*}
By the fractional Leibniz rule, $\dot{W}^{2/15,6} (\R ) \hookrightarrow L^{30} (\R )$, $\dot{H}^{7/15} (\R ) \hookrightarrow L^{30} (\R )$ and Lemma \ref{lem:cor-per}, we get
\begin{align*}
\| \fd ^{3/2} e_6 \| _{L_{t,x}^6 (I_n \times \R )}
& \lesssim  \lambda _n^{-2} \| \lr{\lambda _n^{-1} \dx}^{3/2} (w_n^2 \overline{w_n}^3) \| _{L_{t,x}^6} \\
& \lesssim \lambda _n^{-2} \| \lr{\lambda _n^{-1} \dx}^{3/2}w_n \| _{L_t^6 L_x^{30}} \| w_n \| _{L_t^{\infty} L_x^{30}}^4 \\
& \lesssim \lambda _n^{-2} \||\dx |^{2/15} \lr{\lambda _n^{-1} \dx}^{3/2}w_n\| _{L_{t,x}^6} \||\dx |^{7/15} w_n\| _{L_t^{\infty} L_x^2}^4 \\
& \lesssim \lambda _n^{-2+2\theta} .
\end{align*}
Combining the Strichartz estimate with estimates above, we obtain the desired bound.
\end{proof}

By using this lemma, we modify $\wt{v}_n$ as follows:
\[
\ttilde{v}_n (t) := \begin{cases} \wt{v}_n(t) - f_{n,4}(t) - f_{n,5}(t) - f_{n,6}(t) , & \text{if } |t| \le T \lambda _n^2 , \\
e^{-i(t-T\lambda _n^2) \fd } \wt{\wt{v}}_n(T \lambda _n^2), & \text{if } t > T\lambda _n^2 , \\
e^{-i(t+T\lambda _n^2) \fd } \wt{\wt{v}}_n(-T \lambda _n^2), & \text{if } t <- T\lambda _n^2 .
\end{cases}
\]

\begin{prop} \label{prop:mod-smallstab}
For each $\eps >0$, there exists $T>0$ and $N \in \N$ such that for each $n \ge N$,
\[
(-i \dt + \fd ) \ttilde{v}_n + \fd ^{-1} \mathcal{N} (\Re \ttilde{v}_n) = \wt{e}_1+\wt{e}_2+\wt{e}_3,
\]
with
\[
\| \wt{e}_1\| _{L_t^1 H_x^{2}}+ \|\fd ^{5/2} ( \wt{e}_2+ \wt{e}_3) \| _{L_{t,x}^{6/5}} \le \eps .
\]
Moreover,
\[
\| \ttilde{v}_n-\wt{v}_n \| _{L_t^{\infty} H_x^{2}} + \| \fd ^{3/2} (\ttilde{v}_n - \wt{v}_n) \| _{L_{t,x}^6} \le \eps .
\]
\end{prop}

\begin{proof}
Put
\begin{gather*}
\wt{e}_1= \begin{cases} e_1, & t \in I_n, \\ 0, & \text{otherwise}, \end{cases} \quad
\wt{e}_2= \begin{cases} e_2+e_3, & t \in I_n, \\ 0, & \text{otherwise}, \end{cases} \\
\wt{e}_3= \begin{cases} \fd ^{-1} \left[ \mathcal{N} ( \Re \ttilde{v}_n) - \mathcal{N} (\Re \wt{v}_n) \right], & t \in I_n, \\ \fd ^{-1} \mathcal{N} ( \Re \ttilde{v}_n), & \text{otherwise} . \end{cases} 
\end{gather*}
By \eqref{est:error1}, \eqref{est:error2'}, and \eqref{est:error2},
\[
\|\wt{e}_1\| _{L_t^1 H_x^{2}}+ \|\fd ^{5/2} \wt{e}_2\| _{L_{t,x}^{6/5}} \lesssim T \lambda _n^{-2+4\theta} + \lambda _n^{-1+2\theta}.
\]
By the fractional Leibniz rule, \eqref{eq:iso-bound-sta}, and Lemma \ref{lem:moderror},
\begin{align*}
& \| \fd ^{5/2} \wt{e}_3 \| _{L_{t,x}^{6/5} (I_n \times \R )} \\
& \le \sum _{l=2}^{\infty} \frac{1}{l!} \left\| \fd ^{3/2} ((\Re \ttilde{v}_n)^{2l+1} - (\Re \wt{v}_n)^{2l+1}) \right\| _{L_{t,x}^{6/5}(I_n \times \R)} \\
& \lesssim \Big\{ \| \fd ^{3/2} (\ttilde{v}_n - \wt{v}_n) \| _{L_{t,x}^6 (I_n \times \R )} ( \| \ttilde{v}_n \| _{L_{t,x}^6 (I_n \times \R)} + \| \wt{v}_n \| _{L_{t,x}^6 (I_n \times \R)})^4 \\
& \qquad + \| \ttilde{v}_n - \wt{v}_n \| _{L_{t,x}^6 (I_n \times \R )} ( \| \fd ^{3/2} \ttilde{v}_n \| _{L_{t,x}^6 (I_n \times \R)} + \| \fd ^{3/2} \wt{v}_n \| _{L_{t,x}^6 (I_n \times \R)})^4 \Big\} \\
& \qquad \times \sum _{l=2}^{\infty} \frac{1}{(l-2)!} \Big( \| \ttilde{v}_n \| _{L_{t,x}^{\infty}(I_n \times \R)} + \| \ttilde{v}_n \| _{L_{t,x}^{\infty}(I_n \times \R)} \Big) ^{2l-4} \\
& \lesssim \lambda _n^{-2+2\theta} .
\end{align*}
Thus, we obtain the desired bound on $I_n$.
For the complementary interval, by Sobolev's embedding $L^{\infty} (\R) \hookrightarrow H^1(\R)$, the Strichartz estimate (Lemma \ref{Str}), Lemma \ref{lem:moderror} and \eqref{eq:iso-bound-sta}, we get
\begin{align*}
& \| \ttilde{v}_n \| _{L_{t,x}^{\infty} (I_n^c \times \R )} \\
& \le \| \wt{v}_n \| _{L_{t}^{\infty} H_x^1 (I_n^c \times \R )} + \| f_{n,4} \| _{L_t^{\infty} H_x^{1}(I_n \times \R )} + \| f_{n,5} \| _{L_t^{\infty} H_x^{1}(I_n \times \R )} + \| f_{n,6} \| _{L_t^{\infty} H_x^{1}(I_n \times \R )} \\
& \lesssim _{M(\phi)} 1.
\end{align*}
Therefore, from the Taylor expansion, Proposition \ref{prop:largedataint}, and Lemma \ref{lem:moderror}, taking $T$ and $n$ sufficiently large, we have
\begin{align*}
\| \fd ^{5/2} \wt{e}_3 \| _{L_{t,x}^{6/5} (I_n^c \times \R )}
& \lesssim \| \fd ^{3/2} \ttilde{v}_n \| _{L_{t,x}^6 (I_n^c \times \R )}^{5} \sum _{l=2}^{\infty} \frac{1}{l!} \| \ttilde{v}_n \| _{L_{t,x}^{\infty} (I_n^c \times \R )}^{2l-4} \\
& \lesssim _{M(\phi)} \| \fd ^{3/2} \wt{v}_n \| _{L_{t,x}^6 (I_n^c \times \R )}^5 + \sum _{j=4}^6 \| f_{n,j} \| _{L_t^{\infty} H_x^2(I_n \times \R )}^5 \\
& < \frac{1}{2} \eps .
\end{align*}
Finally, from Lemmas \ref{Str} and \ref{lem:moderror}, we have
\begin{align*}
& \| \ttilde{v}_n-\wt{v}_{n} \| _{L_t^{\infty} H_x^{2} \cap \fd^{-3/2} L_{t,x}^6} \\
& \le \| f_{n,4} + f_{n,5} + f_{n,6} \| _{L_t^{\infty} H_x^{2} \cap \fd^{-3/2} L_{t,x}^6 (I_n \times \R )} \\
& \qquad + \| e^{-i(t-T\lambda _n^2) \fd } (  f_{n,4} + f_{n,5} + f_{n,6}) (T \lambda _n^2) \| _{L_t^{\infty} H_x^{2} \cap \fd^{-3/2} L_{t,x}^6 ((T\lambda _n, \infty ) \times \R )} \\
& \qquad + \| e^{-i(t+T\lambda _n^2) \fd } (  f_{n,4} + f_{n,5} + f_{n,6}) (-T \lambda _n^2) \| _{L_t^{\infty} H_x^{2} \cap \fd^{-3/2} L_{t,x}^6 ((- \infty , -T\lambda _n) \times \R )} \\
& \lesssim  \| f_{n,4} + f_{n,5} + f_{n,6} \| _{L_t^{\infty} H_x^{2} \cap \fd^{-3/2} L_{t,x}^6 (I_n \times \R )} \\
& \lesssim \lambda _n^{-2+2 \theta} .
\end{align*}
\end{proof}

We are ready to show Theorem \ref{isolate} with $\nu _0 \equiv 0$.
By \eqref{eq:iso-bound-sta} and Lemma \ref{lem:moderror},
\[
\|\ttilde{v}_n\| _{L_t^{\infty} H^{2}} + \| \fd ^{3/2} \ttilde{v}_n\| _{L_{t,x}^6} \lesssim _{M(\phi )} 1 .
\]
Thus, by \eqref{eq:iso-inidata} and Proposition \ref{prop:mod-smallstab}, Proposition \ref{prop:stability} with $s=2$ is applicable to $\ttilde{v}_n$.
Namely, there exists a solution $v_n$ to \eqref{rNLKGexp} with $v_n(0)=\phi _n$ satisfying $\| v_n \| _{L_t^{\infty}(\R;H_x^2(\R))}+ \| \fd ^{3/2} v_n \| _{L_{t,x}^6(\R \times \R)} \lesssim _{M(\phi )} 1$.
In addition, Proposition \ref{prop:mod-smallstab} implies that
\begin{equation} \label{6.19}
\lim _{n \rightarrow \infty} \left\{ \|v_n(t) -\wt{v}_n(t-t_n)\| _{L_t^{\infty} H^{2}} + \| \fd ^{3/2} (v_n(t) - \wt{v}_n(t-t_n)) \| _{L_{t,x}^6} \right\} =0.
\end{equation}
From the density of $C_c^{\infty} (\R ^2)$ in $L_{t,x}^6 (\R ^2)$, we can find $\psi _{\eps} \in C_c^{\infty} (\R ^2)$ satisfying
\begin{equation} \label{exstpsi}
\| e^{-it} D_{\lambda _n} [\psi _{\eps} (\lambda _n^{-2} t) - w_{\infty} (\lambda _n^{-2} t)] \| _{L_{t,x}^6}
= \| \psi _{\eps} - w_{\infty} \| _{L_{t,x}^6} < \frac{1}{2} \eps .
\end{equation}
By the triangle inequality, Lemma \ref{lem:cor-per}, Proposition \ref{prop:largedataint}, and Lebesgue's dominated convergence theorem,
\begin{equation} \label{estdiff}
\begin{aligned}
& \| \wt{v}_n - e^{-it} D_{\lambda _n} w_{\infty}(\lambda _n^{-2} t) \| _{L_{t,x}^6} \\
& \lesssim \| \wt{v}_n \| _{L_{t,x}^6 (I_n^c \times \R )} + \| w_n-w_{\infty} \| _{L_{t,x}^6} + \| w_{\infty} \| _{L_{t,x}^6(I_n^c \times \R )}
\rightarrow 0 .
\end{aligned}
\end{equation}
From \eqref{6.19}, \eqref{exstpsi}, and \eqref{estdiff}, we obtain
\begin{align*}
& \| v_n (t+t_n) - e^{-it} D_{\lambda _n} \psi _{\eps} (\lambda _n^{-2} t) \| _{L_{t,x}^6} \\
& \le \| v_n (t+t_n) - \wt{v}_n \| _{L_{t,x}^6} + \| \wt{v}_n - e^{-it} D_{\lambda _n} w_{\infty}(\lambda _n^{-2} t) \| _{L_{t,x}^6} \\
& \hspace*{100pt} + \| e^{-it} D_{\lambda _n} [ w_{\infty} (\lambda _n^{-2} t) - \psi _{\eps} (\lambda _n^{-2} t)] \| _{L_{t,x}^6} \\
& \rightarrow 0,
\end{align*}
which concludes the proof of Theorem \ref{isolate} with $\nu _n \equiv 0$ and $s=\frac{1}{2}$.

Moreover,\eqref{eq:iso-bound-sta0} yields that
\begin{equation*}
\| |\dx|^{s} v_n (t) \| _{L_{t,x}^6}
\lesssim \| \fd^{s} (v_n (t) - \wt{v}_n (t-t_n) \| _{L_{t,x}^6} + \| |\dx|^s \wt{v}_n  \|_{L^6_{t,x}}
\rightarrow 0
\end{equation*}
for any $s>0$.
In addition, since $v_n$ satisfies \eqref{rNLKGexp}, we use $\dot{W}^{2/15,6} (\R) \hookrightarrow L^{30} (\R)$ to obtain
\begin{align*}
\| (\dt +i)  v_n (t) \| _{L_{t,x}^6}
& \le \| (\fd - 1) v_n (t) \|_{L_{t,x}^6} + \| \fd^{-1} \mathcal{N} (\Re v_n) \| _{L_{t,x}^6} \\
& \lesssim \| |\dx| v_n (t) \|_{L_{t,x}^6} + \sum_{l=2}^{\infty} \frac{1}{l!} \| \Re v_n \| _{L_{t}^6 L_x^{30}} \| \Re v_n \| _{L_{t}^{\infty} L_x^{30}}^4  \| \Re v_n \| _{L_{t,x}^{\infty}}^{2l-4} \\
& \lesssim \| |\dx| v_n (t) \|_{L_{t,x}^6} + \| |\dx|^{\frac{2}{15}} v_n \| _{L_{t,x}^6} \sum_{l=2}^{\infty} \frac{1}{l!} \| v_n \| _{L_t^{\infty} H_x^1}^{2l} \\
& \rightarrow 0.
\end{align*}

Second, we consider the general case $\nu _n \not \equiv 0$.
From \ref{LB-comm} in Lemma \ref{LB}
\[
\phi _n = T_{x_n} e^{it_n \fd } \Lo _{\nu _n} D_{\lambda _n} P_{\le \lambda _n^{\theta}} \phi
= \Lo _{\nu _n} T_{\wt{x}_n} e^{i\wt{t}_n \fd } D_{\lambda _n} P_{\le \lambda _n^{\theta}} \phi .
\]
By spatial translation invariance, we may choose $x_n = \frac{\nu _n}{\lr{\nu _n}} t_n$, which implies $\wt{x}_n = 0$ and $\wt{t}_n = \frac{t_n}{\lr{\nu _n}}$.
By the case $\nu _n \equiv 0$, for sufficiently large $n$, there is a global solution $v_n^0$ to \eqref{rNLKGexp} with initial data
\[
v_n^0 (0) = e^{i\wt{t}_n \fd } D_{\lambda _n} P_{\le \lambda _n^{\theta}} \phi.
\]
Moreover, it obeys $\| v_n^0 \| _{L_t^{\infty}(\R;H_x^2(\R))} + \| \fd ^{3/2} v_n^0 \| _{L_{t,x}^6 (\R \times \R)} \lesssim _{M(\phi )} 1$,
\begin{equation} \label{vn0lim}
\lim_{n \rightarrow \infty} \left( \| \dx v_n^0 \|_{L_{t,x}^6} + \| (\dt+i) v_n^0 \|_{L_{t,x}^6} \right) =0,
\end{equation}
and for each $\eps >0$, there exists $\psi _{\eps}^0 \in C_c^{\infty}( \R \times \R )$ and $N_{\eps}^0 \in \N$ such that
\begin{equation} \label{eq:isonuzero}
\left\| \Re \Big\{ v_n^0 (t+ \wt{t}_n, x) - \frac{e^{-it}}{\lambda _n^{1/2}} \psi _{\eps}^0 \Big( \frac{t}{\lambda _n^2}, \frac{x}{\lambda _n} \Big) \Big\} \right\| _{L_{t,x}^6} < \eps
\end{equation}
whenever $n \ge N_{\eps}^0$.
Since $v_n^0$ solves \eqref{rNLKGexp}, $u_n^0 := \Re v_n^0$ solves \eqref{NLKGexp}.
Thus, by the Lorentz invariance, $u_n^1 := u_n^0 \circ L_{\nu _n}$ also solves \eqref{NLKGexp} and
\[
v_n^1 := (1+i \fd ^{-1} \dt ) u_n^1 = (1+i \fd ^{-1} \dt ) (\Re v_n^0 \circ L_{\nu _n})
\]
solves \eqref{rNLKGexp}.

\begin{prop} \label{prop:isogen}
For $n$ sufficiently large, $v_n^1$ is a global solution to \eqref{rNLKGexp}.
Moreover,
\[
\sup _{n \in \N} \Big\{ \| \Re v_n^1 \| _{L_{t,x}^{\infty}} + S_{\R} ( \fd^{1/2} v_n^1) \Big\} \lesssim _{M( \phi )} 1, \quad
\lim _{n \rightarrow \infty} \| v_n^1(0) - \phi _n \| _{H^1} =0.
\]
\end{prop}

\begin{proof}
By Corollary \ref{cor:boostsol}, $u_n^1$ is a strong solution to \eqref{NLKGexp}.
Hence, $v_n^1$ is a strong solution to \eqref{rNLKGexp}.
By the definition, we have
\begin{equation} \label{boundvn1}
\| \Re v_n^1 \| _{L_{t,x}^{\infty}} = \| \Re v_n^0 \| _{L_{t,x}^{\infty}} \le \| v_n^0 \| _{L_t^{\infty}H_x^1} \lesssim _{M(\phi)} 1.
\end{equation}
Since, by $v_n^0=(1+i(\fd^{-1}\dt))u_n^0$ and $u_n^0=\Re v_n^0$,
\[
\| \dt u_n^0 \| _{L_{t,x}^6}
\le \| \fd v_n^0 \| _{L_{t,x}^6} + \| \fd u_n^0 \| _{L_{t,x}^6}
\lesssim \| \fd v_n^0 \| _{L_{t,x}^6}
\lesssim _{M(\phi)} 1,
\]
we have
\begin{align*}
S_{\R} (\fd^{1/2} v_n^1)
& \lesssim \| \fd v_n^1 \| _{L_{t,x}^6}
\lesssim \| u_n^0 \| _{L_{t,x}^6} + \lr{\nu _n} ( \| \dt u_n^0 \| _{L_{t,x}^6} + \| \dx u_n^0 \| _{L_{t,x}^6}) \\
& \lesssim _{M(\phi)} 1.
\end{align*}

We decompose
\[
u_n^0 = u_n^{0, \text{lin}} + \wt{u}_n^0,
\]
where $u_n^{0, \text{lin}}$ solves the linear Klein-Gordon equation with initial data
\[
(1+i \fd ^{-1} \dt) u_n^{0,\text{lin}} (0) = v_n^0 (0) = \Lo _{\nu _n}^{-1} \phi _n .
\]
Since \eqref{LB-linear} implies
\[
(1+ i \fd ^{-1} \dt) [u_n^{0, \text{lin}} \circ L_{\nu _n}] (t) 
= e^{-it\fd} \Lo _{\nu _n} [v_n^0(0)]
= e^{-it\fd} \phi _n,
\]
we get
\[
\| v_n^1 (0) - \phi _n \| _{H^1} \le \| [\wt{u}_n^0 \circ L_{\nu _n}] (0, \cdot ) \| _{H^1} + \| \dt [\wt{u}_n^0 \circ L_{\nu _n}] (0, \cdot ) \| _{L^2} .
\]

Let
\[
\Omega _n := \{ (t,x) \colon 0< \lr{\nu _n} t < - \nu _n x \} \cup \{ (t,x) \colon - \nu _n x < \lr{\nu _n}t <0 \} .
\]
Then, $\partial \Omega _n = L_{\nu _n} (0, \R ) \cup ( \{ 0 \} \times \R )$.
By the same argument as in the proof of Corollary \ref{cor:boostsol}, $\wt{u}_n^0 (0) \equiv 0$, and the convergence of $\{ \nu _n \}$, we get
\begin{align*}
& \limsup _{n \rightarrow \infty} \left( \frac{1}{2} \| [\wt{u}_n^0 \circ L_{\nu _n}] (0, \cdot ) \| _{H^1}^2 + \frac{1}{2} \| \dt [\wt{u}_n^0 \circ L_{\nu _n}] (0, \cdot ) \| _{L^2} \right) \\
& \le \limsup _{n \rightarrow \infty} \lr{\nu_n} \iint _{\Omega _n} |\mathcal{N}(u_n^0 (t,x)) \nabla _{t,x} \wt{u}_n^0 (t,x)| dx dt \\
& \lesssim \limsup _{n \rightarrow \infty} \| u_n^0 \| _{L_{t,x}^6(\Omega _n)}^5 \| \nabla _{t,x} \wt{u}_n^0 \| _{L_{t,x}^6 (\R \times \R )} \sum _{l=2}^{\infty} \frac{1}{l!} \| u_n^0 \| _{L_{t,x}^{\infty}(\R \times \R)}^{2l-4}
\end{align*}
By \eqref{boundvn1}, it suffices to show that
\[
\| \nabla _{t,x} \wt{u}_n^0 \| _{L_{t,x}^6} \lesssim _{M(\phi )} 1
\]
for $n$ sufficiently large and
\begin{equation} \label{eq:isogennu}
\lim _{n \rightarrow \infty} \| u_n^0 \| _{L_{t,x}^6(\Omega _n)} =0 .
\end{equation}
By the triangle inequality, Lemma \ref{Str}, and Proposition \ref{prop:WP},
\begin{align*}
\| \nabla _{t,x} \wt{u}_n^0 \| _{L_{t,x}^6}
& \le \| \nabla _{t,x} u_n^{0,\text{lin}} \| _{L_{t,x}^6} + \| \nabla _{t,x} u_n^{0} \| _{L_{t,x}^6} \\
& \lesssim _{M( \phi )} \| v_n^0 (0) \| _{H^{3/2}} + \| \fd ^{3/2} D_{\lambda _n} P_{\le \lambda _n^{\theta}} \phi \| _{L^2} \\
& \lesssim _{M( \phi )} \| D_{\lambda _n} P_{\le \lambda _n^{\theta}} \phi \| _{H^{3/2}}
\lesssim _{M( \phi )} 1.
\end{align*}
Thanks to \eqref{eq:isonuzero}, the estimate \eqref{eq:isogennu} follows from
\[
\lim _{n \rightarrow \infty} \iint _{\Omega_n} \lambda _n^{-3} \left| \psi \Big( \frac{t- \wt{t}_n}{\lambda _n^2}, \frac{x}{\lambda _n} \Big) \right|^6 dx dt =0
\]
for every $\psi \in C_c^{\infty} (\R \times \R )$.
If $(t,x) \in \Omega_n$ and $\big( \frac{t- \wt{t}_n}{\lambda _n^2}, \frac{x}{\lambda _n} \big) \in \supp \psi$, then we have $|x| \lesssim _{\psi} \lambda _n$ and $|t|< |x|$.
Therefore,
\[
\lim _{n \rightarrow \infty} \iint _{\Omega_n} \lambda _n^{-3} \left| \psi \Big( \frac{t- \wt{t}_n}{\lambda _n^2}, \frac{x}{\lambda _n} \Big) \right|^6 dx dt
\lesssim _{\psi} \lim _{n \rightarrow \infty} \lambda _n^{-3} \lambda _n^2
= \lim _{n \rightarrow \infty} \lambda _n^{-1} =0.
\]
This completes the proof.
\end{proof}

From Propositions \ref{prop:WP}, \ref{prop:stability}, and \ref{prop:isogen}, we obtain that for sufficiently large $n$, there exists a global solution $v_n$ to \eqref{rNLKGexp} with $v_n(0) = \phi _n$ and $\| v_n \| _{L_t^{\infty}(\R; H_x^1(\R))} + S_{\R} (\fd^{1/2} v_n) \lesssim _{M(\phi )} 1$.
Moreover,
\[
\lim _{n \rightarrow \infty} \| \fd^{1/2} \Re \{ v_n - v_n^1 \} \| _{L_{t,x}^6} =0.
\]
Let $s \in [\frac{1}{2},1]$.
From $\Re v_n^1 = \Re v_n^0 \circ L_{\nu _n}$ and \eqref{vn0lim}, we have
\begin{align*}
&\left\| (\fd^{s-1/2}-\lr{\nu}^{s-1/2}) \Re v_n^1 \right\| _{L_{t,x}^6} \\
&\le \| (\fd^{s-1/2} - \lr{\nu_n}^{s-1/2}) (\Re v_n^0 \circ L_{\nu _n}) \|_{L_{t,x}^6} \\
&\quad + \| (\lr{\nu_n}^{s-1/2} - \lr{\nu}^{s-1/2}) (\Re v_n^0 \circ L_{\nu _n}) \|_{L_{t,x}^6} \\
& \lesssim \| (\dx - i \nu_n) (\Re v_n^0 \circ L_{\nu _n}) \|_{L_{t,x}^6} + |\nu_n-\nu| \| \Re v_n^0 \|_{L_{t,x}^6} \\
&\le |\nu_n| \| (\dt+i) \Re v_n^0 \|_{L_{t,x}^6} + \lr{\nu_n} \| \dx \Re v_n^0 \|_{L_{t,x}^6} + |\nu_n-\nu| \| \Re v_n^0 \|_{L_{t,x}^6}
\rightarrow 0,
\end{align*}
as $n \rightarrow \infty$.
We therefore obtain
\begin{align*}
& \left\| \Re \Big\{ (\fd^{s-1/2} v_n) \circ L_{\nu _n}^{-1} (t+ \wt{t}_n, x+ \wt{x}_n) - \lr{\nu}^{s-1/2} \frac{e^{-it}}{\lambda _n^{1/2}} \psi _{\eps}^{0} \Big( \frac{t}{\lambda _n^2}, \frac{x}{\lambda _n} \Big) \Big\} \right\| _{L_{t,x}^6} \\
& \lesssim \| \fd^{1/2} \Re \{ v_n - v_n^1 \} \| _{L_{t,x}^6} + \left\| (\fd^{s-1/2}-\lr{\nu}^{s-1/2}) \Re v_n^1 \right\| _{L_{t,x}^6} \\
& \quad + \lr{\nu}^{s-1/2} \left\| \Re \Big\{ v_n^1 \circ L_{\nu _n}^{-1} (t+ \wt{t}_n, x+ \wt{x}_n) - \frac{e^{-it}}{\lambda _n^{1/2}} \psi _{\eps}^{0} \Big( \frac{t}{\lambda _n^2}, \frac{x}{\lambda _n} \Big) \Big\} \right\| _{L_{t,x}^6}
\rightarrow 0,
\end{align*}
which concludes the proof of Theorem \ref{isolate}.
\end{proof}

\subsection{Remark on the focusing cases}
\label{S:6.1}

The arguments in \S \ref{sec:IS} are also valid for the focusing cases if we assume that the initial data are below the corresponding static solutions.
More precisely, we can treat the following focusing equation
\begin{equation} \label{rNLKGexpf}
-iv _t + \fd v - \fd ^{-1} \No(\Re v) =0
\end{equation}
or
\begin{equation} \label{rNLKGf}
-iv _t + \fd v - \frac{1}{2} \fd ^{-1} |\Re v|^4 \Re v =0.
\end{equation}

We recall the scattering result obtained by Dodson \cite{Dod15} for the focusing mass-critical nonlinear Schr\"odinger equation:
\begin{equation} \label{mNLSf}
\Big( i \dt + \frac{1}{2} \dx^2 \Big)  w = - \frac{5}{32} |w|^4 w .
\end{equation}
The standing wave solution $w_Q$ associated to \eqref{mNLSf} is
\[
w_Q (t,x) := e^{it} 2\sqrt[4]{\frac{2}{5}} Q (\sqrt{2} x),
\]
where $Q$ is the ground state, which is defined \eqref{GroundS}.
Note that $M(w_Q) = \frac{4}{\sqrt{5}} M(Q)$.

\begin{thm}[\cite{Dod15}] \label{scaNLSf}
Let $w_0 \in L^2 (\R )$ and assume that $M(w_0) < \frac{4}{\sqrt{5}} M(Q)$.
Then, there exists a unique global solution $w\in C(\R;L_x^2(\R))$ to \eqref{mNLS} with $w(0)= w_0$.
Furthermore, the solution $w$ satisfies
\[
\| w \| _{L^6_{t,x} (\R \times \R )} \le C(M(u_0)) .
\]
As a consequence, $w$ scatters as $t\rightarrow\pm\infty$ in $L^2_x(\R)$, that is, there exists $w_{\pm} \in L^2 (\R )$ such that the identity holds:
\[
\lim _{t \rightarrow \pm \infty} \| w(t) - e^{it \dx^2 /2} w_{\pm} \| _{L^2_x} =0 .
\]
Conversely, for any $w_{\pm} \in L^2 (\R )$, there exists a unique global solution $w$ to \eqref{mNLSf} so that the above holds.
\end{thm}

\begin{cor} \label{cor:isolate}
Let $\phi \in L^2(\R )$, and assume also that $M(\phi ) < \frac{4}{\sqrt{5}} M(Q)$.
Then, the same statement as in Theorem \ref{isolate} holds true for \eqref{rNLKGexpf} or \eqref{rNLKGf}. 
\end{cor}

\section{Minimal-energy blowup solutions} \label{sec:7}

In this section, we construct so-called minimal-energy blowup solutions in the contradiction argument. First we introduce the definition of almost periodicity modulo translations.

\begin{dfn}[Almost periodicity modulo translations] \label{D:apmt}
We say that a global solution $u$ to \eqref{NLKGexp} is {\it almost periodic modulo translations} (in $H^1_x \times L^2_x$) if there exist functions $x:\R \to \R$
and $C:\R^+ \to \R^+$ such that for every $t \in \R$ and $\eta > 0$, we have
\begin{gather} \label{E:apmt u}
\int_{|x-x(t)| > C(\eta)} \left( |u(t,x)|^2 + |u_x(t,x)|^2 + |u_t(t,x)|^2 \right) dx < \eta \\
\label{E:apmt uhat}
\int_{|\xi| > C(\eta)} \left( |\langle \xi \rangle\hat u(t, \xi)|^2 + |\hat u_t(t, \xi)|^2 \right) d\xi < \eta.
\end{gather}
We refer to $x(t)$ as the {\it spatial center function} and to $C$ as the {\it compactness modulus function}.
\end{dfn}

We suppose that Theorem \ref{thm:scat} fails. Then, there exists a critical energy $E_c>0$ such that 
if $u$ is a real-valued solution to (\ref{NLKGexp}) with $E(u,u_t)<E_c$, then the solution $u$ scatters. The goal of this section is to prove the following.

\begin{thm}\label{thm7-1}
Suppose that Theorem \ref{thm:scat} fails. Then there exists a global solution $u$ to (\ref{NLKGexp}) with $E(u,u_t)=E_c$. Moreover, $u$ is almost periodic modulo translations and blows up (that is, posesses infinite scattering size) both forward and backward in time.
\end{thm}

\begin{rmk}
The global solution $u$ constructed in Theorem \ref{thm7-1} has zero momentum, i.e., $P(u)=0$.
Indeed, if $P(u) \neq 0$, the Cauchy-Schwarz inequality yields that
\[
|P(u)| < \| u_x \|_{L^2} \| u_t \|_{L^2}
\le \frac{1}{2} \| u_x \|_{L^2}^2 + \frac{1}{2} \| u_t \|_{L^2}^2
\le E(u,u_t).
\]
Here, equality in the first inequality would imply that $u$ is a solution to a transport equation, namely $u(t,x) = u_0(x-kt)$ for some $k \in \R$, which is inconsistent with the fact that $u$ is a solution to \eqref{NLKGexp} and $P(u) \neq 0$.
By setting
\[
\nu := - \frac{P(u)}{\sqrt{E(u,u_t)^2-P(u)^2}},
\]
it follows from Corollary \ref{cor:boostsol} that $u^{\nu} := u \circ L_{\nu}$ is a global solution to \eqref{NLKGexp} with $P(u^{\nu})=0$ and
\[
E(u^{\nu}, u^{\nu}_t) = \sqrt{E(u,u_t)^2 - P(u)^2} < E(u,u_t) =E_c.
\]
This contradicts to the criticality of $E_c$.
\end{rmk}

Once we get the following Proposition \ref{P:palais smale}, Theorem \ref{thm7-1} follows from the same argument as in the proof of Theorem 7.19 in \cite{KSV12}.
Hence, we only prove Proposition \ref{P:palais smale}.

\begin{prop}
\label{P:palais smale}
Suppose that Theorem \ref{thm:scat} fails.
Let $s \in (\frac{1}{2}, \frac{11}{12})$ and let $E_c$ be the critical energy.
Assume that $\{u_n\}$ is a sequence of global solutions to \eqref{NLKGexp} satisfying 
\begin{align}
\label{eq:energy_lim}
	&E(u_n) \leq E_c \text{ and } \lim_{n \to \infty} E(u_n) =E_c,
	\\
\label{eq:Strichartz_lim}
	&\lim_{n \to \infty} S_{\leq0}(\left\langle \partial_x \right\rangle^{s-1/2}u_n)=\lim_{n \to \infty} S_{\geq0}(\left\langle \partial_x \right\rangle^{s-1/2}u_n)=\infty.
\end{align}
Then, after passing to a subsequence, $(u_n(0),\partial_t u_n(0))$ converges in $H^1\times L^2$, modulo translation. 
\end{prop}

Let $v_n := u_n + i \langle \partial_x \rangle^{-1} \partial_t u_n$ to work with the first-order version of our equation. Since we are considering defocusing case, we have
\begin{align}
\label{E:H1 vn}
\left\Vert v_n(0) \right\Vert_{H^1}^2 \lesssim E(v_n) \lesssim E_c < \infty,
\end{align}
which implies that $\{v_n(0)\}$ is a bounded sequence in $H^1$. By applying the linear profile decomposition (Theorem \ref{lpd}), 
\begin{align}
\label{eq:lpd}
v_n (0)= \sum_{j=1}^{J} \phi_n^j + w_n^J, \qquad 1\leq J < J_0
\end{align}
where $\phi_n^j:=T_{x_n^j} e^{it_n^j \left\langle \partial_x \right\rangle} \Lo_{\nu_n^j} D_{\lambda_n^j} P_n^j \phi^j$. 
We have $J_0\ge 2$ since \eqref{eq:Strichartz_lim} is not compatible with \eqref{4-2} when $J_0=1$. 
Passing to a subsequence, we may assume that $E(\phi_n^j)$ converges for each $j\in [1,J_0)$. 
By the energy decoupling (\ref{4-9}), we have
\begin{align}
\label{E:bound nrg sum}
 \lim_{n \to \infty} \left\{ \sum_{j=1}^{J} E(\phi_n^j  )+ E( w_n^J) \right\} = \lim_{n \to \infty} E(v_n)= E_c,
\end{align}
for each $J \in [1,J_0)$.
One of the following scenarios occurs by the positivity of the energy.

{\bf Case I.} There is only a single profile and it satisfies
\begin{align*}
\lim_{n \to \infty} E(\phi_n^1)=E_c.
\end{align*}

{\bf Case II.} There exists $\delta>0$ such that for every $1\leq j < J_0$,
\begin{align}
\label{E:small profiles}
\lim_{n \to \infty} E(\phi_n^j) < E_c -\delta.
\end{align}

In {\bf Case I}, we have $w_n \to 0$ in $H^1$ as $n\rightarrow\infty$. Moreover we devide the following three cases in this case:

{\bf Case I-A.} $\lambda_n^1 \to \infty$. 

{\bf Case I-B.} $\lambda_n^1 \equiv 1$ and $t_n^1 \to \pm \infty$.

{\bf Case I-C.} $\lambda_n^1 \equiv 1$ and $t_n^1 \equiv 1$.
\\
We will observe that both of the first and second cases do not occur.
Note that {\bf Case I-C} implies the conclusion of the proposition. 

In {\bf Case I-A}, by using Theorem \ref{isolate}, the stability theory (Proposition \ref{prop:stability}), and the fact $w_n^1 \to 0$ in $H^1$, we get a contradiction.

Next we consider {\bf Case I-B.} Suppose that $\lambda_n^1\equiv 1$ for all $n\in \N$ and  $t_n^1 \to \pm\infty$ as $n\rightarrow\infty$.  We only treat the case $t_n^1 \to -\infty$ since the other case can be treated in a similar manner.
By the Strichartz inequality (Lemma \ref{Str}), we see that $\langle \partial_x \rangle^{s-1/2} e^{-it\langle \partial_x \rangle}\phi^1 \in L^6_{t,x}(\R \times \R)$, and so
$$
\Vert \langle \partial_x \rangle^{s-1/2}  e^{-i(t-t_n^1)\langle \partial_x \rangle}  \phi^1 \Vert_{L^6_{t,x}([0,\infty) \times \R)} \to 0 \quad \text{as $n\to\infty$.}
$$
Hence by the local well-posedness (Proposition \ref{prop:WP} and see also Proposition \ref{prop:stability}), if $v_n^1$ is the unique solution to \eqref{rNLKGexp} with initial data $v_n^1(0) = \phi_n^1$,
then for $n$ sufficiently large, $S_{\geq 0}( \langle \partial_x \rangle^{s-1/2} v_n^1) < \infty$.  As in Case I-A, we can now use the stability lemma (Proposition \ref{prop:stability})
to conclude that for sufficiently large $n$, we see $S_{\geq 0}( \langle \partial_x \rangle^{s-1/2} v_n)<\infty$, which derives a contradiction. Next we consider {\bf Case II}.

{\bf Case II:}  We will show that this case is inconsistent with \eqref{eq:Strichartz_lim} by using the identity \eqref{eq:lpd} to produce a nonlinear profile decomposition of the $v_n$ and then
applying the stability theory (Proposition \ref{prop:stability}).  We begin by introducing nonlinear profiles $v_n^j$, whose definition depends on the behavior of the parameters $\lambda_n^j$.

First assume that $j$ is such that $\lambda_n^j \equiv 1$ for all $n\in \N$.  Then we see that $\phi^j \in H^1_x$ and
$$
\phi_n^j = T_{x_n^j}e^{it_n^j\langle \partial_x \rangle}\phi^j.
$$
If, in addition, $t_n^j \equiv 0$ for all $n\in \N$, then we set $v^j$ be the maximal-lifespan solution to \eqref{rNLKGexp} with $v^j(0) = \phi^j$.  If $t_n^j \to -\infty$ (respectively $t_n^j \to \infty$),
then we set $v^j$ be the maximal-lifespan solution to \eqref{rNLKGexp} which scatters forward (respectively backward) in time to $e^{-it\langle \partial_x \rangle}\phi^j$.

\begin{lem} \label{L:vj global}
In Case~II, if $\lambda_n^j \equiv 1$ for some $j$, then $v^j$ defined as above is global and scatters.
\end{lem}

\begin{proof}[Proof of Lemma \ref{L:vj global}]
This follows from the well-posedness result (Proposition \ref{prop:WP}).
If $t_n^j \equiv 0$, then $E(\phi_n^j) = E(\phi^j)=E(v^j)$ by the conservation of the energy. 
If $t_n^j \to \pm \infty$, then by using the dispersive estimate (Lemma \ref{lem:dispersive})%(see Lemma 4  in the file of ``inverse Strichartz for KG 1d Hs") 
\ and approximating $\phi^j$ in $H^1_x$ by Schwartz functions, we see that
\[
\lim_{n \to \infty} E(\phi_n^j) = \lim_{n \to \infty} \tfrac12\Vert\phi_n^j\Vert_{H^1_x}^2 = \tfrac12 \Vert\phi^j\Vert_{H^1_x}^2 =E(v^j).
\]
Since we are considering Case II, we have $E(v^j) < E_c $, which implies that $v^j$ scatters as $t\rightarrow\pm\infty$.
\end{proof}

If $\lambda_n^j \equiv 1$ for all $n\in \N$, we may define nonlinear profiles by
$$
v_n^j(t,x) := v^j(t-t_n^j,x-x_n^j).
$$

Next, we consider the case where $\lim_{n \to \infty}\lambda_n^j = \infty$.  Then by Theorem~\ref{isolate}, for sufficiently large $n$, we define $v_n^j$ as the unique solution to \eqref{rNLKGexp} with initial data $v_n^j(0) = \phi_n^j$.

\begin{lem} \label{L:nl profiles nice}
In Case~II, for each $j$ $($regardless of the behavior of the $\lambda_n^j)$ we have
\begin{align}
\label{E:same energy}
\lim_{n \to \infty} E(v_n^j) &= \lim_{n \to \infty} E(\phi_n^j),  \\
\label{E:bounded ST norm}
\lim_{n \to \infty} S_{\R}( \langle \partial_x \rangle^{s-1/2} v_n^j)^2 &\lesssim \lim_{n \to \infty} E(v_n^j).
\end{align}
Furthermore, for each $j$ and $\eps > 0$, there exists $\psi^{j}=\psi_{\eps}^{j} \in C^{\infty}_c(\R \times \R)$ and $N_{j,\eps}$ such that if $\psi_n^j$ is defined as in Theorem \ref{isolate} and $n > N_{j,\eps}$, then we have
\begin{equation} \label{E:psi approximates v}
\Vert \Re( \psi_n^j - \langle \partial_x \rangle^{s-1/2}  v_n^j)\Vert_{L^6_{t,x}} < \eps.
\end{equation}
\end{lem}

\begin{proof}[Proof of Lemma \ref{L:nl profiles nice}]
Equality \eqref{E:same energy} is a tautology if $\lambda_n^j \to \infty$ as $n\rightarrow\infty$ or $\lambda_n^j \equiv 1$ and $t_n^j \equiv 0$ for all $n\in \N$, since in these cases $v_n^j(0) = \phi_n^j$.
If $\lambda_n^j \equiv 1$ for all $n\in \N$ and $t_n^j \to \pm\infty$ as $n\rightarrow\infty$, then by the definition of $v_n^j$ and the local well-posedness result (Proposition \ref{prop:WP}) (see small data scattering statement)
%\eqref{E:scattering nrg}, 
we have
$$
E(v_n^j) = E(v^j) = \tfrac12 \Vert\phi^j\Vert_{H^1_x}^2 = \lim_{n \to \infty} E(\phi_n^j),
$$
where for the last equality, we have used the dispersive estimate (Lemma \ref{lem:dispersive}) as in the proof of Lemma~\ref{L:vj global}.

%Inequality \eqref{E:bound mass vnj} follows easily from \eqref{E:bound single mass} and the definition of~$v_n^j$.

When the value $\lim_{n\to\infty} E(\phi_n^j)$ is below the small data scattering threshold, \eqref{E:bounded ST norm} follows from the well-posedness result (Proposition \ref{prop:WP}). 
%Note that in the focusing case,\eqref{E:bound single mass} and \eqref{E:coercive E} imply that the energy controls the $H^1_x$ norm.  
On the other hand, by the identities \eqref{E:bound nrg sum}, the limiting energy can only exceed this small data scattering threshold for finitely many values of $j$.  For these cases, we invoke the estimate \eqref{E:small profiles} and the definition of the critical energy $E_c$.  As we are invoking the
contradiction hypothesis here, there is no hope of being explicit about the constant in \eqref{E:bounded ST norm} other than that it is independent of $j$.

As for \eqref{E:psi approximates v}, in the case $\lambda_n^j \equiv 1$ for all $n\in \N$, this follows from the fact that $v_n^j$ is just a translation of $\langle \partial_x \rangle^{s-1/2} v^j \in L^6_{t,x}(\R \times \R^2)$.
In the case $\lambda_n^j \to \infty$ as $n\rightarrow\infty$, this approximation follows from Theorem~\ref{isolate}. This completes the proof of the lemma.
\end{proof}

We continue to prove Proposition \ref{P:palais smale}. For $1 \leq J < J_0$, we set
$$
V_n^J(t) := \sum_{j=1}^J v_n^j(t) + e^{-it\langle \partial_x \rangle}w_n^J,
$$
which is defined globally in time for sufficiently large $n$ (depending on $J$).  Our immediate goal is to show
that $V_n^J(t)$ is a good approximation to $v_n(t)$ when $n$ and $J$ are sufficiently large by the stability theory (Proposition \ref{prop:stability}).

\begin{lem} \label{L:nlpd stability setup}
We have the following spacetime bounds on $V_n^J$
\begin{equation} \label{E:uniform st bounds}
\limsup_{J \to \infty} \limsup_{n \to \infty}\bigl\{ \Vert\langle \partial_x \rangle^{s-1/2} \Re V_n^J\Vert_{L^6_{t,x}} + \Vert V_n^J\Vert_{L^{\infty}_t H^{s}_x} \bigr\} < \infty.
\end{equation}
The $V_n^J$ are approximate solutions to \eqref{rNLKGexp} in the sense that
$$
(-i\partial_t +\langle \partial_x \rangle)V_n^J + \langle \partial_x \rangle^{-1} \mathcal{N}(\Re V_n^J) =  E_n^J,
$$
where
\begin{equation} \label{E:nearly nlkg}
\lim_{J \to \infty} \limsup_{n \to \infty} \Vert\langle \partial_x \rangle^{s+1/2} E_n^J\Vert_{L^{6/5}_{t,x}} = 0.
\end{equation}
Furthermore, for each $J$, we have
\begin{equation} \label{E:initial H1 ok}
\lim_{n \to \infty} \Vert v_n(0) - V_n^J(0)\Vert_{H^1_x} = 0.
\end{equation}
\end{lem}

\begin{proof}[Proof of Lemma \ref{L:nlpd stability setup}]

First we prove \eqref{E:initial H1 ok}.  By the triangle inequality and the definitions of $v_n^j$,
\begin{equation*}
\lim_{n\to\infty} \Vert v_n(0) - V_n^J(0)\Vert_{H^1_x} \leq \lim_{n\to\infty}  \sum_{j=1}^J\Vert v_n^j(0) - \phi_n^j\Vert_{H^1_x} =0.
\end{equation*}

We note that combining \eqref{E:bounded ST norm} and \eqref{E:bound nrg sum} yields
\begin{equation}\label{E:the dude}
\limsup_{J \to \infty}\limsup_{n \to \infty} \sum_{j=1}^J \Vert \langle \partial_x \rangle^{s-1/2} \Re v_n^j\Vert_{L^6_{t,x}}^2 \lesssim \lim_{J \to \infty}\lim_{n \to \infty} \sum_{j=1}^J E(v_n^j) \leq E_c.
\end{equation}

We now bound the $L_{t}^6 W_{x}^{s-1/2,6}$ term in \eqref{E:uniform st bounds}.  By \eqref{E:psi approximates v} and Proposition~\ref{denon}, the nonlinear profiles decouple
in the sense that whenever $j \neq j'$, we have
\begin{equation} \label{E:vs decoup}
\lim_{n \to \infty} \Vert (\langle \partial_x \rangle^{s-1/2} \Re v_n^j) (\langle \partial_x \rangle^{s-1/2} \Re v_n^{j'})\Vert_{L^3_{t,x}} = 0.
\end{equation}
Combining this with \eqref{4-2} and then using \eqref{E:the dude} show
\begin{align}\label{E:uniform A}
&\limsup_{J \to \infty}\limsup_{n \to \infty} \Vert \langle \partial_x \rangle^{s-1/2}\Re V_n^J\Vert_{L^6_{t,x}}^6
=\limsup_{J \to \infty}\limsup_{n \to \infty} \sum_{j=1}^J \Vert \langle \partial_x \rangle^{s-1/2} \Re v_n^j\Vert_{L^6_{t,x}}^6
    \lesssim E_c^3.
\end{align}
Thus we get the first estimate in \eqref{E:uniform st bounds}. 

Next, we prove \eqref{E:nearly nlkg}.  
A simple computation shows that
$$
E_n^J(t) = \langle \partial_x \rangle^{-1} \mathcal{N}\biggl( \Re \sum_{j=1}^J v_n^j(t) + \Re e^{-it\langle \partial_x \rangle}w_n^J\biggr) -  \langle \partial_x \rangle^{-1}\sum_{j=1}^J \mathcal{N} (\Re v_n^j),
$$
and so, by the triangle inequality, it is enough to show
\begin{gather}
\label{E:error from w}
\lim_{J \to \infty} \limsup_{n \to \infty} \left\Vert \langle \partial_x \rangle^{s-1/2} \left\{ \mathcal{N}\biggl(\sum_{j=1}^J \Re v_n^j + \Re e^{-it\langle \partial_x \rangle}w_n^J\biggr)
    -\mathcal{N} \biggl(\sum_{j=1}^J \Re v_n^j\biggr) \right\} \right\Vert_{L^{6/5}_{t,x}}   = 0 , \\
\label{E:error from parentheses}
\lim_{J \to \infty} \limsup_{n \to \infty} \left\Vert \langle \partial_x \rangle^{s-1/2}  \left\{ \mathcal{N} \biggl(\sum_{j=1}^J \Re v_n^j\biggr) - \sum_{j=1}^J  \mathcal{N} \bigl(\Re v_n^j\bigr) \right\} \right\Vert_{L^{6/5}_{t,x}} = 0.
\end{gather}

For simplicity, we set $A:=\sum_{j=1}^J \Re v_n^j$ and $B:= \Re e^{-it\langle \partial_x \rangle}w_n^J$. Since $\mathcal{N}(u)=\sum_{l=2}^{\infty} \frac{u^{2l+1}}{l!}$, \eqref{E:error from w} can be estimated as follows. 
\begin{align*}
	\left\| \langle \partial_x \rangle^{s-1/2} \bigl\{ \mathcal{N} (A+B) -\mathcal{N} (A) \bigr\} \right\|_{L_{t,x}^{6/5}}
	&= \left\| \sum_{l=2}^{\infty} \frac{1}{l!} \langle \partial_x \rangle^{s-1/2} \{ (A+B)^{2l+1} - A^{2l+1} \} \right\|_{L_{t,x}^{6/5}}
	\\
	& \lesssim \sum_{l=2}^{\infty} \frac{1}{l!} \left\| \langle \partial_x \rangle^{s-1/2} \{ (A+B)^{2l+1} - A^{2l+1} \} \right\|_{L_{t,x}^{6/5}}.
\end{align*}
Since we have
\begin{align*}
	(A+B)^{2l+1} - A^{2l+1} = \sum_{m=1}^{2l+1} \binom{2l+1}{m} A^{2l+1-m}B^{m},
\end{align*}
we get
\begin{align}
\label{eq7.20}
	&\sum_{l=2}^{\infty} \frac{1}{l!} \left\| \langle \partial_x \rangle^{s-1/2} \{ (A+B)^{2l+1} - A^{2l+1} \} \right\|_{L_{t,x}^{6/5}}
	\\ \notag
	& \quad \lesssim \sum_{l=2}^{\infty} \frac{1}{l!} \sum_{m=1}^{2l+1}  \binom{2l+1}{m}  \left\| \langle \partial_x \rangle^{s-1/2} (A^{2l+1-m}B^{m})\right\|_{L_{t,x}^{6/5}} .
\end{align}
If $1 \leq m \leq 4$, then, by the fractional Leibniz rule, we have
\begin{align*}
	\left\| \langle \partial_x \rangle^{s-1/2} (A^{2l+1-m}B^{m})\right\|_{L_{t,x}^{6/5}}
	&\lesssim l \bigg( \left\| \langle \partial_x \rangle^{s-1/2} A\right\|_{L_{t,x}^{6}} 
	\left\|  A\right\|_{L_{t,x}^{\infty}}^{2l-3}
	\left\|  A\right\|_{L_{t,x}^{6}}^{4-m}
	\left\| B\right\|_{L_{t,x}^{6}}^{m}
	\\
	&\quad +\left\| \langle \partial_x \rangle^{s-1/2} B\right\|_{L_{t,x}^{6}} 
	\left\|  A\right\|_{L_{t,x}^{\infty}}^{2l-3 } 
	\left\|  A\right\|_{L_{t,x}^{6}}^{4-m}
	\left\| B\right\|_{L_{t,x}^{6}}^{m-1} \bigg).
\end{align*}
If $5 \leq m \leq 2l$, we have
\begin{align*}
	\left\| \langle \partial_x \rangle^{s-1/2} (A^{2l+1-m}B^{m})\right\|_{L_{t,x}^{6/5}}
	&\lesssim l \bigg( \left\| \langle \partial_x \rangle^{s-1/2} A\right\|_{L_{t,x}^{6}} 
	\left\|  A\right\|_{L_{t,x}^{\infty}}^{2l-m}
	\left\|  B\right\|_{L_{t,x}^{6}}^{4}
	\left\| B\right\|_{L_{t,x}^{\infty}}^{m-4}
	\\
	&\quad +\left\| \langle \partial_x \rangle^{s-1/2} B\right\|_{L_{t,x}^{6}} 
	\left\|  A\right\|_{L_{t,x}^{\infty}}^{2l-m+1 } 
	\left\|  A\right\|_{L_{t,x}^{6}}^{4}
	\left\| B\right\|_{L_{t,x}^{\infty}}^{m-5} \bigg) .
\end{align*}
If $m=2l+1$, we have
\begin{align*}
	\left\| \langle \partial_x \rangle^{s-1/2} (A^{2l+1-m}B^{m})\right\|_{L_{t,x}^{6/5}}
	\lesssim l \left\| \langle \partial_x \rangle^{s-1/2} B \right\|_{L_{t,x}^{6}}
	\left\| B \right\|_{L_{t,x}^{6}}^{4} \left\| B \right\|_{L_{t,x}^{\infty}}^{2l-4}.
\end{align*}
We note that $ \|  \langle \partial_x \rangle^{s-1/2} A \|_{L_{t,x}^{6}}+ \|  \langle \partial_x \rangle^{s-1/2} B \|_{L_{t,x}^{6}} \le C$ holds by \eqref{4-2}, \eqref{E:the dude}, and \eqref{E:vs decoup}, where $C = C (E_c)$ is a positive constant independent of $n$ ant $J$.
Moreover, we also have $\|  A \|_{L_{t}^{\infty}H_{x}^{s}}+\|  B \|_{L_{t}^{\infty}H_{x}^{1}} \le C$.
While $\|  B \|_{L_{t}^{\infty}H_{x}^{1}} \le C$ is trivial, $\|  A \|_{L_{t}^{\infty}H_{x}^{s}} \le C$ is shown later. 
Therefore, \eqref{eq7.20} is estimated as follows. 
\begin{align*}
	&\sum_{l=2}^{\infty} \frac{1}{l!} \sum_{m=1}^{2l+1}  \binom{2l+1}{m}  \left\| \langle \partial_x \rangle^{s-1/2} (A^{2l+1-m}B^{m})\right\|_{L_{t,x}^{6/5}}
	\\
	&\quad \lesssim \sum_{l=2}^{\infty} \frac{1}{l!} \sum_{m=1}^{2l+1}  \binom{2l+1}{m} l C^{2l+1}
	\\
	&\quad  \le \sum_{l=2}^{\infty} \frac{1}{(l-1)!} (2C)^{2l+1}
	\\
	&\quad  = (2C)^3 \sum_{l=1}^{\infty} \frac{1}{l!} (4C^2)^l
	\\
	&\quad  \leq (2C)^3 \exp ( 4C^2).
\end{align*}
Thus, by \eqref{4-2} and the Lebesgue dominated convergence theorem, we obtain
\begin{align*}
	&\left\| \langle \partial_x \rangle^{s-1/2} \bigl\{ \mathcal{N} (A+B) -\mathcal{N} (A) \bigr\} \right\|_{L_{t,x}^{6/5}}
	\\
	&\quad \lesssim \sum_{l=2}^{\infty} \frac{1}{l!} \sum_{m=1}^{2l+1}  \binom{2l+1}{m}  \left\| \langle \partial_x \rangle^{s-1/2} (A^{2l+1-m}B^{m})\right\|_{L_{t,x}^{6/5}}
	\\
	& \quad \to 0
\end{align*}
as $n,J \to \infty$. Thus we obtain \eqref{E:error from w}.

%%%%%
Here, we show $\|A\|_{L_t^\infty H_{x}^s} \le C$.
Note that$A=\sum_{j=1}^J \Re v_n^j$ satisfies the following equation.
\begin{align*}
	(-i\partial_t +\langle \partial_x \rangle)A +  \langle \partial_x \rangle^{-1} \sum_{j=1}^J  \mathcal{N}(\Re v_n^j) =0.
\end{align*}
Since \eqref{E:same energy}, \eqref{E:bounded ST norm}, and \eqref{E:bound nrg sum} yield that $\| \Re v_n^j \|_{L_{t,x}^{6}} \leq \| \langle \partial_x \rangle^{s-1/2}  (\Re v_n^j) \|_{L_{t,x}^{6}} \leq C$ and $\| \Re v_n^j \|_{L_{t,x}^{\infty}} \le \| \Re v_n^j \|_{L_{t}^{\infty}H_x^1} \leq E(v_n^j)^{1/2} \le C$, we apply the Strichartz estimate (Lemma \ref{Str}) and the fractional Leibniz rule to obtain
\begin{align*}
	&\limsup_{n \to \infty} \left\| \langle \partial_x \rangle^{s} \int_{0}^{t} e^{-i(t-s)\langle \partial_x \rangle} \langle \partial_x \rangle^{-1} \sum_{j=1}^J  \mathcal{N}(\Re v_n^j)  ds \right\|_{L_t^{\infty}L_x^2}
	\\
	&\quad  \lesssim \limsup_{n \to \infty}  \left\| \langle \partial_x \rangle^{s-1/2} \sum_{j=1}^{J}  \mathcal{N}(\Re v_n^j) \right\|_{L_{t,x}^{6/5}}
	\\
	&\quad  \lesssim \limsup_{n \to \infty} \sum_{j=1}^{J} \sum_{l=2}^{\infty} \frac{1}{l!}  \left\| \langle \partial_x \rangle^{s-1/2}  (\Re v_n^j)^{2l+1} \right\|_{L_{t,x}^{6/5}}
	\\
	&\quad  \lesssim \limsup_{n \to \infty}  \sum_{j=1}^{J} \sum_{l=2}^{\infty} \frac{1}{l!} (2l+1) 
	\left\| \langle \partial_x \rangle^{s-1/2}  (\Re v_n^j) \right\|_{L_{t,x}^{6}}
	\left\|   \Re v_n^j \right\|_{L_{t,x}^{6}}^4
	\left\|   \Re v_n^j \right\|_{L_{t,x}^{\infty}}^{2l-4}
	\\
	&\quad  \lesssim \limsup_{n \to \infty} \sum_{j=1}^{J} \sum_{l=2}^{\infty} \frac{C^{2l-1}}{(l-1)!} 
	\left\| \langle \partial_x \rangle^{s-1/2}  (\Re v_n^j) \right\|_{L_{t,x}^{6}}^{2}
	\\
	&\quad  \lesssim \sum_{l=2}^{\infty} \frac{C^{2l-1}}{(l-1)!} \sum_{j=1}^{J} 
	\limsup_{n \to \infty} E(\phi_n^j)
	\\
	&\quad  \lesssim C^{3} \sum_{l=1}^{\infty} \frac{C^{2l}}{l!}
	\\
	&\quad  \le C^3 \exp (C^2).
\end{align*}
By \eqref{E:initial H1 ok} and \eqref{E:H1 vn}, we get
\begin{align*}
	&\limsup_{n \to \infty} \Vert A\Vert_{L^{\infty}_tH^{s}_x} 
	\\
	&\lesssim \limsup_{n \to \infty} \left\{\Vert v_n(0) \Vert_{H^{1}_x} + \Vert w_n^J\Vert_{H^{1}_x} + \left\| \langle \partial_x \rangle^{s} \int_{0}^{t} e^{-i(t-s)\langle \partial_x \rangle} \langle \partial_x \rangle^{-1} \sum_{j=1}^J  \mathcal{N}(\Re v_n^j)  ds \right\|_{L_t^{\infty}L_x^2} \right\}
	\\
	&\lesssim_{E_c} 1.
\end{align*}
%%%%%

Next, we consider \eqref{E:error from parentheses}.
We observe that
\begin{align*}
	& \left\Vert \langle \partial_x \rangle^{s-1/2}  \left\{ \mathcal{N} \biggl(\sum_{j=1}^J \Re v_n^j\biggr) - \sum_{j=1}^J  \mathcal{N} \bigl(\Re v_n^j\bigr) \right\} \right\Vert_{L^{6/5}_{t,x}}
	 \\ 
	 & \quad \lesssim  \sum_{l=2}^{\infty} \frac{1}{l!}  \left\Vert \langle \partial_x \rangle^{s-1/2} \left\{ \biggl(\sum_{j=1}^J \Re v_n^j\biggr)^{2l+1} - \sum_{j=1}^J  \bigl(\Re v_n^j\bigr)^{2l+1} \right\} \right\Vert_{L^{6/5}_{t,x}}.
\end{align*}
Since we have
\begin{align*}
	 &\biggl(\sum_{j=1}^J \Re v_n^j\biggr)^{2l+1} - \sum_{j=1}^J  \bigl(\Re v_n^j\bigr)^{2l+1}
	 \\
	 & \quad = \sum_{\substack{0 \le m_j <2l \\ m_1+m_2+\cdots+m_J =2l+1}} \frac{(2l+1)!}{m_1! m_2! \cdots m_J!} \bigl(\Re v_n^1\bigr)^{m_1} \bigl(\Re v_n^2\bigr)^{m_2} \cdots \bigl(\Re v_n^J\bigr)^{m_J}.
 \end{align*}
we get
 \begin{align*}
 	&\sum_{l=2}^{\infty} \frac{1}{l!}  \left\Vert \langle \partial_x \rangle^{s-1/2} \left\{ \biggl(\sum_{j=1}^J \Re v_n^j\biggr)^{2l+1} - \sum_{j=1}^J  \bigl(\Re v_n^j\bigr)^{2l+1} \right\} \right\Vert_{L^{6/5}_{t,x}}
	\\
	&\quad \leq  \sum_{l=2}^{\infty} \frac{1}{l!}  
	\sum_{\substack{0 \le m_j <2l \\ m_1+m_2+\cdots+m_J=2l+1} } \frac{(2l+1)!}{m_1! m_2! \cdots m_J!}
	\left\Vert \langle \partial_x \rangle^{s-1/2} \left\{ \bigl(\Re v_n^1\bigr)^{m_1} \bigl(\Re v_n^2\bigr)^{m_2} \cdots \bigl(\Re v_n^J\bigr)^{m_J} \right\} \right\Vert_{L^{6/5}_{t,x}}.
 \end{align*}

Now, by the fractional Leibniz rule, we have
 \begin{align*}
 	&\left\Vert \langle \partial_x \rangle^{s-1/2} \left\{ \bigl(\Re v_n^1\bigr)^{m_1} \bigl(\Re v_n^2\bigr)^{m_2} \cdots \bigl(\Re v_n^J\bigr)^{m_J} \right\} \right\Vert_{L^{6/5}_{t,x}}
	\\
	&\quad  \lesssim J l C^{2l-1}  \left(\sum_{j \neq k} 
	\left\Vert \langle \partial_x \rangle^{s-1/2}  \bigl(\Re v_n^{j} \cdot \Re v_n^k\bigr) \right\Vert_{L^{3}_{t,x}}
	+\sum_{j \neq k} 
	\left\Vert  \Re v_n^{j} \cdot \Re v_n^k \right\Vert_{L^{3}_{t,x}}\right)
	\\
	&\quad \le J l C^{2l+1}.
\end{align*}
Note that $\| \langle \partial_x \rangle^{s-1/2} (\Re v_n^{j} \cdot \Re v_n^k) \|_{L^{3}_{t,x}}$ and $\| \Re v_n^{j} \cdot \Re v_n^k \|_{L^{3}_{t,x}}$ tend to $0$ as $n \to \infty$ for any $j \neq k$. Indeed, it follows that $\| \Re v_n^{j} \cdot \Re v_n^k \|_{L^{3}_{t,x}} \to 0$ by an approximation argument and, by the interpolation, we have
\begin{align*}
	&\| \langle \partial_x \rangle^{s-1/2} (\Re v_n^{j} \cdot \Re v_n^k) \|_{L^{3}_{t,x}} 
	\\
	&\leq \| \Re v_n^{j} \cdot \Re v_n^k\|_{L^3}^{2(1-s)} \|  \langle \partial_x \rangle^{1/2} (\Re v_n^{j} \cdot \Re v_n^k) \|_{L^3}^{2s-1}
	\\
	&\lesssim  \| \Re v_n^{j} \cdot \Re v_n^k\|_{L^3}^{2(1-s)} 
	\left( \|  \langle \partial_x \rangle^{1/2} \Re v_n^{j} \|_{L^6} \| \Re v_n^k \|_{L^6} + \| \langle \partial_x \rangle^{1/2} \Re v_n^k \|_{L^6} \| \Re v_n^{j}  \|_{L^6} \right)^{2s-1}
	\\
	& \to 0
\end{align*}
since $\| \langle \partial_x \rangle^{1/2} \Re v_n^j \|_{L^6}$, $\| \Re v_n^j \|_{L^6}$ are bounded. 
By the Lebesgue dominated convergence theorem, we obtain for all $J$,
\begin{align}
\label{eq7.21}
	\limsup_{n \to \infty} \left\Vert \langle \partial_x \rangle^{s-1/2}  \left\{ \mathcal{N} \biggl(\sum_{j=1}^J \Re v_n^j\biggr) - \sum_{j=1}^J  \mathcal{N} \bigl(\Re v_n^j\bigr) \right\} \right\Vert_{L^{6/5}_{t,x}} = 0.
\end{align}
Thus, we get \eqref{E:error from parentheses}. 

Finally, we complete the proof of \eqref{E:uniform st bounds} by bounding the $L^{\infty}_tH^{s}_x$ norm.  By the Strichartz inequality, \eqref{E:initial H1 ok},
and then \eqref{E:H1 vn}, \eqref{E:uniform A}, and \eqref{E:nearly nlkg},
\begin{align*}
\limsup_{J \to \infty} & \limsup_{n \to \infty} \Vert V_n^J\Vert_{L^{\infty}_tH^{s}_x} \\
    &\lesssim \limsup_{J \to \infty} \limsup_{n \to \infty} \Bigl\{\Vert v_n(0)\Vert_{H^{1}_x} + C' + \Vert\langle \partial_x \rangle^{s-1/2} E_n^J\Vert_{L^{6/5}_{t,x}} \Bigr\} < \infty,
\end{align*}
where we used that 
\begin{align*}
	&\left\| \langle \partial_x \rangle^{s} \int_{0}^{t} e^{-i(t-s)\langle \partial_x \rangle} \langle \partial_x \rangle^{-1} \mathcal{N}(\Re V_n^J) ds \right\|_{L_t^{\infty}L_x^2}
	\\
	& \quad \lesssim \left\| \langle \partial_x \rangle^{s-1/2}  \mathcal{N}(\Re V_n^J) \right\|_{L_{t,x}^{6/5}}
	\\
	&\quad  \lesssim \sum_{l=2}^{\infty} \frac{1}{l!}  \left\| \langle \partial_x \rangle^{s-1/2}  (\Re V_n^J)^{2l+1} \right\|_{L_{t,x}^{6/5}}
	\\
	&\quad  \lesssim \sum_{l=2}^{\infty} \frac{1}{l!} (2l+1) 
	\left\| \langle \partial_x \rangle^{s-1/2}  (\Re V_n^J) \right\|_{L_{t,x}^{6}}
	\left\|   \Re V_n^J \right\|_{L_{t,x}^{6}}^4
	\left\|   \Re V_n^J \right\|_{L_{t,x}^{\infty}}^{2l-4}
	\\
	&\quad \lesssim \sum_{l=2}^{\infty} \frac{1}{(l-1)!} E_c^{5} C^{2l-4}
	\\
	&\quad \leq C',
\end{align*}
where we used $\| \Re V_n^J \|_{L_{t,x}^{\infty}} \lesssim \| \Re V_n^J \|_{L_{t}^{\infty}H_x^s} \leq \| A \|_{L_{t}^{\infty}H_x^s} + \| B \|_{L_{t}^{\infty}H_x^s} \leq C < \infty$.
This completes the proof of \eqref{E:uniform st bounds} and so also the lemma.
\end{proof}

By Lemma \ref{L:nlpd stability setup}, we may apply the stability theorem (Proposition \ref{prop:stability}) to conclude that in Case II, $v_n$ is defined globally in time and $S_{\R}(v_n) \lesssim_{E_c} 1$ for sufficiently large $n$. This contradicts \eqref{eq:Strichartz_lim} and so Case II cannot occur.  Tracing back, we see that the only possibility is Case I-C, and
so Proposition~\ref{P:palais smale} is proved.

\section{Death of a solition}
In this section, we prove the following theorem (Theorem \ref{thm8-1}), which completes the proof of the main result (Theorem \ref{thm:scat}).
\begin{thm}\label{thm8-1}
There is no non-scattering solution to (\ref{NLKGexp}) with almost periodicity modulo translation and zero momentum.  
\end{thm}

In order to prove the theorem, we need the following lemmas (Lemma \ref{lem0.1} and Lemma \ref{lem0.3}).

\begin{lem}[controlling $x$]
\label{lem0.1}
Let a global solution $u$ be almost periodic modulo translation and its momentum is zero, \textit{i.e.} $P(u)=0$. Then, for sufficiently small $\eta>0$, there exists $R=R(\eta) \gtrsim C(\eta)$ such that 
\[ |x(t)-x(0)| \leq R-C(\eta)\]
for any $t\in [0,t_0]$ where $t_0 > cR/\eta$, where $c$ is a positive constant. 
\end{lem}

This lemma means that $|x(t)|=o(|t|)$ as $|t| \to \infty$. 

\begin{proof}
We may assume that $x(0)=0$. Let $R>3C(\eta)$ and 
\[ t_0 := \inf \{ t>0: |x(t)| \leq R-C(\eta)\}. \]
By the finite speed of propagation, we have $t_0>0$ (see the proof of Lemma 7.4 in \cite{IMN11}). Note that we have $|x-x(t)| \geq C(\eta)$ if $|x|\geq R$ since $|x(t)|\leq R-C(\eta)$ for $t \in (0,t_0)$. Let $\phi \in C_{0}^{\infty}(\R)$ be an even function with $0\leq \phi \leq 1$ and
\[ \phi(x)=
\left\{
\begin{array}{ll}
	1 & \text{for } |x| \leq 1,
	\\
	0 & \text{for } |x| \geq 2.
\end{array}
\right. \]
We define 
\[ X_{R} (t):= \int_{\R} x \phi \left( \frac{x}{R}\right) e_u(t,x) dx,\]
where $e_u:= \frac{1}{2}|u|^2+\frac{1}{2}|u_x|^2+\frac{1}{2}|u_t|^2 +\frac{1}{2}\widetilde{\mathcal{N}}(u)$.
Since $u$ satisfies (\ref{NLKGexp}), we have
\begin{align*}
	\frac{d}{dt} X_R(t) 
	&= - \int_\R  \phi \left( \frac{x}{R}\right) u_x u_t dx - \frac{1}{R} \int_\R x  \phi' \left( \frac{x}{R}\right) u_x u_t dx
	\\
	&= \int_\R \left(1- \phi \left( \frac{x}{R}\right) \right) u_x u_t dx - \frac{1}{R} \int_\R x  \phi' \left( \frac{x}{R}\right) u_x u_t dx
	\\
	&= \int_{|x|\geq R} \left(1- \phi \left( \frac{x}{R}\right) \right) u_x u_t dx - \frac{1}{R} \int_{R\leq |x| \leq 2R} x  \phi' \left( \frac{x}{R}\right) u_x u_t dx,
\end{align*}
where we have used $P(u)=0$ in the second equality. This implies that 
\begin{align}
\label{eq0.2}
	\left|  \frac{d}{dt} X_{R}(t) \right|
	\le c \eta,
\end{align}
where $c$ is a positive constant.
By the triangle inequality, we have
\begin{align}
	\label{eq0.3}
	|X_{R}(0)|
	&\leq \int_{|x|\leq C(\eta)} |x| \phi \left( \frac{x}{R}\right) e_u dx +\int_{C(\eta) \leq |x|\leq 2R} |x| \phi \left( \frac{x}{R}\right) e_u dx
	\\
	\notag
	& \leq C(\eta) E(u,u_t) + R \eta.
\end{align}
Moreover, by the triangle inequality, we have
\begin{align*}
	&|X_R(t)|
	\\
	& \geq \left| x(t) \int_{\R} e_u(t) dx \right| - \left| \int_{\R} (x-x(t)) \phi \left( \frac{x}{R}\right) e_u(t) dx  \right|
	\\
	& \quad -|x(t)| \left| \left(1 - \phi \left( \frac{x}{R}\right) \right) \int_{\R} e_u(t) dx \right|
	\\
	& \geq \left| x(t) \int_{\R} e_u(t) dx \right| - \left| \int_{|x-x(t)|<C(\eta)} (x-x(t)) \phi \left( \frac{x}{R}\right) e_u(t) dx  \right|
	\\
	& \quad - \left| \int_{|x-x(t)|\geq C(\eta)} (x-x(t)) \phi \left( \frac{x}{R}\right) e_u(t) dx  \right| -|x(t)| \left| \int_{\R} \left(1 - \phi \left( \frac{x}{R}\right) \right)e_u(t) dx \right|
	\\
	&=:I-I\!\!I-I\!\!I\!\!I-I\!V.
\end{align*}
Then, for $t \in (0,t_0)$, we have
\begin{align*}
	I &= |x(t)| E(u,u_t),
	\\
	I\!\!I & \leq C(\eta) E(u,u_t),
	\\
	I\!\!I\!\!I & \leq \int_{|x-x(t)|\geq C(\eta), |x|\leq 2R} |x-x(t)| \phi \left( \frac{x}{R}\right) e_u(t) dx
	\\
	& \leq \int_{C(\eta) \leq |x-x(t)|\leq 2R+|x(t)|} |x-x(t)| \phi \left( \frac{x}{R}\right) e_u(t) dx
	\\
	& \leq (2R+|x(t)|) \int_{C(\eta) \leq |x-x(t)|} e_u(t) dx
	\\
	& \leq \left( R+ \frac{1}{2}|x(t)| \right) \eta ,
	\\
	I\!V & \leq |x(t)|  \int_{|x|\geq R} e_u(t) dx 
	\\
	& \leq |x(t)|  \int_{|x-x(t)|\geq C(\eta)} e_u(t) dx
	\\
	& \leq \frac{1}{2} |x(t)| \eta.
\end{align*}
Thus, we obtain 
\begin{align}
\label{eq0.4}
	|X_R(t)| \geq |x(t)|\left(E(u,u_t)- \eta \right) -R\eta - C(\eta)E(u,u_t). 
\end{align}
By (\ref{eq0.2}), (\ref{eq0.3}), and (\ref{eq0.4}), for $\tau \in (0,t_0)$, we get
\begin{align*}
	c \eta \tau 
	&\geq \left| \int_{0}^{\tau} \frac{d}{dt} X_{R}(t) dt \right|
	\\
	&\geq |X_{R}(\tau) |- |X_{R}(0)|  
	\\
	& \geq   |x(t)|\left(E(u,u_t)- \eta \right) -R \eta -C(\eta)E(u,u_t) - ( C(\eta) E(u,u_t) + R \eta).
\end{align*}
Letting $\eta<\frac{1}{2}E(u,u_t)$, then we have 
\[ \frac{1}{2} E(u,u_t) |x(t)| \leq c \eta \tau +2R\eta +2C(\eta)E(u,u_t), \]
and thus we get
\[ |x(t)| \leq \frac{2c}{E(u,u_t)} \eta \tau + \frac{4R\eta}{E(u,u_t)} + 4C(\eta)\]
Taking $\tau \to t_0$, then we obtain
\[ R-C(\eta) \leq \frac{2c}{E(u,u_t)} \eta t_0  + \frac{4R\eta}{E(u,u_t)} + 4C(\eta).\]
Letting $\eta<\frac{1}{8}E(u,u_t)$ and $R>20C(\eta)$, we have
\[ \frac{1}{4}R < \frac{1}{2}R-5C(\eta) \leq \frac{2c}{E(u,u_t)} \eta t_0.  \]
This means that 
\[ \frac{E(u,u_t)}{8c} \frac{R}{\eta} \leq t_0.\]
\end{proof}

\begin{lem}
\label{lem0.3}
Let a global solution $u$ be almost periodic modulo translation.
For any $\eps>0$, there exists $C=C(\eps)>0$ such that 
\begin{align*}
	\Vert u(t) \Vert_{L^2}^2 \leq \eps +C(\eps) \Vert u_x(t) \Vert_{L^2}^2
\end{align*}
for any $t \in \R$. 
\end{lem}

\begin{proof}
Since $u$ is almost periodic modulo translation, for any $\eps>0$, there exists $C=C(\eps)>0$ such that 
\[ \int_{|\xi| < \frac{1}{\sqrt{C(\eps)}}} |\hat{u}(t,\xi)|^2 d\xi<\eps. \]
By the Plancherel equality, we have
\begin{align*}
	\Vert u(t) \Vert_{L_x^2}^2
	&
	=\Vert \hat{u}(t) \Vert_{L_{\xi}^2}^2
	\\
	&\leq \eps + \int_{|\xi| \geq  \frac{1}{\sqrt{C(\eps)}}} |\hat{u}(t,\xi)|^2 d\xi
	\\
	&\leq \eps + C(\eps) \int_{|\xi| \geq  \frac{1}{\sqrt{C(\eps)}}} |\xi|^2 |\hat{u}(t,\xi)|^2 d\xi
	\\
	&\leq \eps +C(\eps) \Vert u_x(t) \Vert_{L^2}^2.
\end{align*}
\end{proof}

We define the localized virial identity as follows.
\[V_R(t):= \int_{\R} \phi \left( \frac{x}{R}\right)u_t(u+2xu_x ) dx. \]
A direct calculation gives that 
\begin{align*}
	V'_R(t) 
	&= -2\Vert u_x \Vert_{L^2}^2 - \int_{\R} \mathcal{N}(u)udx +\int_{\R} \widetilde{\mathcal{N}}(u)dx 
	\\
	&\quad +2\int_{R} \left( 1- \phi \left( \frac{x}{R}\right)\right) |u_x|^2 dx 
	+ \int_{R} \left( 1- \phi \left( \frac{x}{R}\right)\right) \mathcal{N}(u)u dx 
	\\
	&\quad - \int_{R} \left( 1- \phi \left( \frac{x}{R}\right)\right) \widetilde{\mathcal{N}}(u) dx 
	+\frac{1}{2R^2} \int_{\R}  \phi'' \left( \frac{x}{R}\right) |u|^2 dx
	\\
	&\quad -\int_{\R} \frac{x}{R} \phi' \left( \frac{x}{R}\right) (|u_x |^2 - |u|^2 -\widetilde{\mathcal{N}}(u) + |u_t|^2) dx. 
\end{align*}

\begin{proof}[Proof of Theorem \ref{thm8-1}]
We suppose that there exists a non-scattering solution to (\ref{NLKGexp}) with almost periodicity modulo translation and zero momentum. We may assume that $x(0)=0$ by the translation invariance of the equation (\ref{NLKGexp}). 
Since $3\widetilde{\mathcal{N}} (u) \leq \mathcal{N}(u)u$ and we have Lemma \ref{lem0.1}, we find
\[ V'_R(t) \leq -2\Vert u_x \Vert_{L^2}^2 - \frac{2}{3} \int_{\R} \mathcal{N}(u)udx+\eta, \]
for $t \in (0,t_0)$, where $R$ and $t_0$ satisfy the assumption in Lemma \ref{lem0.1}. Therefore, we get
\begin{align}
	\label{eq0.5}
	V_R(0)-V_R(t_0) 
	&
	=-\int_{0}^{t_0} V'_R(t)dt
	\\ \notag
	& \geq  2 \int_{0}^{t_0} \Vert u_x \Vert_{L^2}^2 dt + \frac{2}{3}\int_{0}^{t_0} \int_{\R} \mathcal{N}(u)udx dt -\eta t_0.
\end{align}
On the other hand, since Lemma \ref{lem0.3} gives us that
\begin{align*}
	 \frac{d}{dt} \int_{\R} u u_t dx
	 &= \Vert u_t \Vert_{L^2}^2 - \Vert u_x \Vert_{L^2}^2 - \Vert u \Vert_{L^2}^2 - \int_{\R} \mathcal{N}(u)udx
	 \\
	 & \geq \Vert u_t \Vert_{L^2}^2 - \Vert u_x \Vert_{L^2}^2 -\eps -C(\eps) \Vert u_x \Vert_{L^2}^2 - \int_{\R} \mathcal{N}(u)udx,
\end{align*}
we obtain 
\begin{align*}
	4E(u,u_t)
	&\geq \int_{0}^{t_0}   \frac{d}{dt} \int_{\R} u u_t dx dt
	\\ \notag
	&\geq  \int_{0}^{t_0}  \Vert u_t \Vert_{L^2}^2 dt  -\eps t_0 -\{C(\eps)+1\}  \int_{0}^{t_0}  \Vert u_x \Vert_{L^2}^2 dt -  \int_{0}^{t_0}  \int_{\R} \mathcal{N}(u)udx dt,
\end{align*}
where $\eps$ is chosen later, which is independent of $\eta$ and $R$.
This implies that 
\begin{align}
\label{eq0.6}
	\{C(\eps)+2\} \int_{0}^{t_0}  \Vert u_x \Vert_{L^2}^2 dt
	&+ \int_{0}^{t_0}  \int_{\R} \mathcal{N}(u)udx dt 
	\\ \notag
	&\geq \int_{0}^{t_0}  \Vert u_t \Vert_{L^2}^2 dt +  \int_{0}^{t_0}  \Vert u_x \Vert_{L^2}^2 dt -\eps t_0 -4E(u,u_t).
\end{align}
Combining (\ref{eq0.5}) with (\ref{eq0.6}), we obtain
\begin{align*}
	\frac{C(\eps)+3}{2}(V_R(0)-V_R(t_0)) 
	& \geq \{C(\eps)+3\} \int_{0}^{t_0}  \Vert u_x \Vert_{L^2}^2 dt+ \frac{C(\eps)+3}{3}  \int_{0}^{t_0}  \int_{\R} \mathcal{N}(u)udx dt
	\\
	& \quad - \frac{C(\eps)+3}{2} \eta t_0
	\\
	&\geq  \{C(\eps)+2\} \int_{0}^{t_0}  \Vert u_x \Vert_{L^2}^2 dt + \int_{0}^{t_0}  \int_{\R} \mathcal{N}(u)udx dt 
	\\
	& \quad - \frac{C(\eps)+3}{2} \eta t_0
	\\
	& \geq \int_{0}^{t_0} ( \Vert u_t \Vert_{L^2}^2 +\Vert u_x \Vert_{L^2}^2 )dt -\eps t_0 -4E(u,u_t) - \frac{C(\eps)+3}{2} \eta t_0.
\end{align*}
We find that there exists a positive constant $\delta$ such that $ \Vert u_t \Vert_{L^2}^2 +\Vert u_x \Vert_{L^2}^2>\delta$ for any $t \in\R$. Indeed, we have $\delta_{\text{scat}}>0$ satisfying that the solution scatters if $\left\Vert u \right\Vert_{H^1}^2 + \left\Vert u_t \right\Vert_{L^2}^2<\delta_{\text{scat}}$ by the small data scattering result. Taking $\eps=\delta_{\text{scat}}/4$  in Lemma \ref{lem0.3} and assuming $\Vert u_t (t)\Vert_{L^2}^2 +\Vert u_x(t) \Vert_{L^2}^2<\min \{ \frac{\delta_{\text{scat}}}{4} ,\frac{\delta_{\text{scat}}}{4C(\delta_{\text{scat}}/4) +1} \}=:\delta$ for some $t$, we find that $u$ scatters by Lemma \ref{lem0.3} and the small data scattering and thus get a contradiction. Thus, there exists a positive constant $\delta$ such that $ \Vert u_t \Vert_{L^2}^2 +\Vert u_x \Vert_{L^2}^2>\delta$ for any $t \in\R$. Fix $\eps=\delta/2$. We get 
\begin{align*}
	\frac{C(\eps)+3}{2}(V_R(0)-V_R(t_0)) 
	& \geq \delta t_0 -\eps t_0 -4E(u,u_t) - \frac{C(\eps)+3}{2} \eta t_0
	\\
	& = \frac{\delta}{2} t_0  -4E(u,u_t) - \frac{C(\delta/2)+3}{2} \eta t_0
	\\
	&\ge C\left\{ \frac{\delta}{2} - \frac{C(\delta/2)+3}{2} \eta \right\}  \frac{R(\eta)}{\eta} -4E(u,u_t).
\end{align*}
On the other hand, we have $V_R(0)-V_R(t_0)  \lesssim R(\eta)$. Therefore, for sufficiently small $\eta$, we obtain a contradiction since 
\[ C_{\delta}R(\eta) < C\left\{ \frac{\delta}{2} - \frac{C(\delta/2)+3}{2} \eta \right\}  \frac{R(\eta)}{\eta} -4E(u,u_t). \] 
\end{proof}

\section*{Acknowledgments}
This work was supported by JSPS KAKENHI Grant Numbers JP15K17571, JP16K17624, JP17J01263.


\begin{thebibliography}{99}

\bibitem{ConSau88}
P. Constantin, and J. Saut, \textit{Local smoothing properties of dispersive equations}, J. Amer. Math. Soc. \textbf{1} (1988), no. 2, 413--439. 

\bibitem{Dod15}
B. Dodson, \textit{Global well-posedness and scattering for the mass critical nonlinear Schr\"odinger equation with mass below the mass of the ground state}, Adv. Math. \textbf{285} (2015), 1589--1618

\bibitem{Dod16}
B. Dodson, \textit{Global well-posedness and scattering for the defocusing, $L^2$ critical, nonlinear Schr\"{o}dinger equation when $d=1$}, Amer. J. Math. \textbf{138} (2016), no. 2, 531--569. 

\bibitem{GinVel85}
J. Ginibre and G. Velo, \textit{Time decay of finite energy solutions of the nonlinear Klein-Gordon and Schr\"{o}dinger equations}, Ann. Inst. H. Poincar\'e Phys. Th\'eor. \textbf{43} (1985), no. 4, 399--442.

\bibitem{HLP88}
G. Hardy, J. Littlewood, and G. P\'{o}lya, ``Inequalities,'' Cambridge University Press, Cambridge, 1988.

\bibitem{IMM06}
S. Ibrahim, M. Majdoub, and N. Masmoudi, \textit{Global solutions for a semilinear, two-dimensional Klein-Gordon equation with exponential-type nonlinearity}, Comm. Pure Appl. Math. \textbf{59} (2006), no. 11, 1639--1658.

\bibitem{IMMN09}
S. Ibrahim, M. Majdoub, N. Masmoudi, and K. Nakanishi, Kenji, \textit{Scattering for the two-dimensional energy-critical wave equation}, Duke Math. J. \textbf{150} (2009), no. 2, 287--329.

\bibitem{IMN11}
S. Ibrahim, N. Masmoudi, and K. Nakanishi, \textit{Scattering threshold for the focusing nonlinear Klein-Gordon equation}, Anal. PDE \textbf{4} (2011), no. 3, 405--460.
Correction to the article \textit{Scattering threshold for the focusing nonlinear Klein-Gordon equation}, Anal. PDE \textbf{9} (2016), no. 2, 503--514.

\bibitem{KPV91}
C. E. Kenig, G. Ponce, and L. Vega, \textit{Oscillatory integrals and regularity of dispersive equations}, Indiana Univ. Math. J. \textbf{40} (1991), no. 1, 33--69.

\bibitem{KKSV12}
R. Killip, S. Kwon, S. Shao, and M. Visan, \textit{On the mass-critical generalized KdV equation}, Discrete Contin. Dyn. Syst. \textbf{32} (2012), no. 1, 191--221. 

\bibitem{KSV12}
R. Killip, B. Stovall, and M. Visan, \textit{Scattering for the cubic Klein-Gordon equation in two space dimensions}, Trans. Amer. Math. Soc. \textbf{364} (2012), no. 3, 1571--1631.

\bibitem{KilVis13}
R. Killip and M. Visan, ``Nonlinear Schr\"{o}dinger equations at critical regularity,'' Evolution equations, 325-437, Clay Math. Proc., \textbf{17}, Amer. Math. Soc., Providence, RI, 2013.

\bibitem{LieLos01}
E. H. Lieb and M. Loss, ``Analysis. Second edition'' Graduate Studies in Mathematics, \textbf{14}, American Mathematical Society, Providence, RI, 2001.

\bibitem{NakOza98}
M. Nakamura and T. Ozawa, \textit{Nonlinear Schr\"odinger equations in the Sobolev space of critical order}, J. Funct. Anal. \textbf{133} (1998), no. 2, 364--380.

\bibitem{NakOza01}
M. Nakamura and T. Ozawa, \textit{The Cauchy problem for nonlinear Klein-Gordon equations in the Sobolev spaces}, Publ. Res. Inst. Math. Sci. \textbf{37} (2001), no. 3, 255--293.

\bibitem{Nak08}
K. Nakanishi, \textit{Transfer of global wellposedness from nonlinear Klein-Gordon equation to nonlinear Schr\"{o}dinger equation}, Hokkaido Math. J. \textbf{37} (2008), no. 4, 749--771.

\bibitem{Sjo87}
P. Sj\"olin, \textit{Regularity of solutions to the Schr\"{o}dinger equation}, Duke Math. J. \textbf{55} (1987), no. 3, 699--715. 

\bibitem{Tao03}
T. Tao, \textit{A sharp bilinear restrictions estimate for paraboloids}, Geom. Funct. Anal. 13 (2003), no. 6, 1359--1384. 

\bibitem{TVZ08}
T. Tao, M. Visan, Monica, and X. Zhang, \textit{Minimal-mass blowup solutions of the mass-critical NLS}, Forum Math. \textbf{20} (2008), no. 5, 881--919. 

\bibitem{Veg88}
L. Vega, \textit{Schr\"{o}dinger equations: pointwise convergence to the initial data}. Proc. Amer. Math. Soc. \textbf{102} (1988), no. 4, 874--878
\end{thebibliography}
\end{document}